\documentclass[letterpaper, 11pt]{article}
\usepackage{bbm}
\usepackage{amsfonts}
\usepackage{amsmath}
\usepackage{amsthm,mathrsfs}
\usepackage{amsfonts,amssymb}
\usepackage{float}
\usepackage{pdfsync}
\usepackage{graphicx}
\usepackage{caption}
\usepackage{hyperref}
\allowdisplaybreaks[4]

\usepackage{amsthm}
\newtheorem{theorem}{Theorem}[section]

\newtheorem{proposition}{Proposition}[section]
\newtheorem{lemma}{Lemma}[section]
\newtheorem{definition}{Definition}[section]

\newtheorem{remark}{Remark}[section]
\theoremstyle{remark}
\theoremstyle{corollary}
\newtheorem{corollary}{Corollary}[section]
\theoremstyle{remark}
\theoremstyle{step}

\numberwithin{equation}{section}

\def\dsum{\displaystyle\sum}
\def\dint{\displaystyle\int}
\def\dfrac{\displaystyle\frac}
\def\dsup{\displaystyle\sup}
\def\dinf{\displaystyle\inf}
\def\dlim{\displaystyle\lim}
\def\r{\right}
\def\l{\left}
\def\rz{\rr^n}

\def\fz{\infty}
\def\fz{\infty}
\def\az{\alpha}
\def\supp{{\mathop\mathrm{\,supp\,}}}

\def\loc{{\mathop\mathrm{\,loc\,}}}
\def\lz{\lambda}
\def\dz{\delta}

\def\ez{\epsilon}

\def\bz{\beta}

\def\gz{{\gamma}}
\def\vz{\varphi}
\def\tz{\theta}
\def\sz{\sigma}

\def\wt{\widehat}

\def\oz{\Omega}

\def\pz{{\prime}}

\def\p{\partial}

\def\Xint#1{\mathchoice	{\XXint\displaystyle\textstyle{#1}}	{\XXint\textstyle\scriptstyle{#1}}	{\XXint\scriptstyle\scriptscriptstyle{#1}} {\XXint\scriptscriptstyle\scriptscriptstyle{#1}}	\!\int}\def\XXint#1#2#3{{\setbox0=\hbox{$#1{#2#3}{\int}$}	 \vcenter{\hbox{$#2#3$}}\kern-.5\wd0}}\def\dashint{\Xint-}

\def\supp{{\mathop\mathrm{\,supp\,}}}
\def\rr{{\mathbb R}}
\def\rn{{{\rr}^n}}

\def\nn{{\mathbb N}}

\def\pz{\partial}

\def\nn{{\mathbb N}}

\begin{document}
\title{{ Wiener's criterion for degenerate parabolic  equations}
\footnote
{Supported by the National Natural Science Foundation of China (No.11771023).}}

\date{}\maketitle
\vspace{-2cm}
\begin{center}
Xi Hu and Lin Tang

\end{center}
{\bf Abstract:}  In this paper, we prove  Wiener's criterion for parabolic  equations  with singular and degenerate coefficients. To be precise,
we study the problem of the regularity of boundary points for the Dirichlet problem for  degenerate parabolic  equations, and give a geometric characterization
 of those boundary points that are regular.
\vspace{0.2cm}
\vspace{0.2cm}

\section{Introduction and main results}
$\mathcal{L}$
A long-standing open problem in degenerate  parabolic potential theory has been that of
the convergence of the solution of the Dirichlet problem to the prescribed boundary value.  For Laplace operator, such a question was settled by Wiener
in his celebrated 1924 paper \cite{WI}. Later, Littman, Stapacchia and Weinberger \cite{LSW} generalized  Wiener's result to elliptic equations with discontinuous coefficients, and
Fabes,  Jerison and  Kenig \cite{FJK} gave the Wiener test for degenerate elliptic equations. For heat operator,
in 1982, Wiener's criterion for the heat equation proved by Evans and Gariepy in \cite{EG}. The Evans-Gariepy Wiener test was extend to
parabolic operators in divergence form  with smooth coefficients by Garofalo and Lanconelli in \cite{GL}, and with $C^1$-Dini continuous coefficients
by Fabes, Garofalo and Lanconelli in \cite{FGL}. We also mention \cite{KL} (and reference therein) for some recent developments in quasilinear parabolic settings.

The main aim of this paper is to prove Wiener's criterion for parabolic operators with singular and degenerate coefficients.
More precisely,  we consider boundary regularity for the degenerate parabolic  equation
$$\mathscr{L}_au:=D_t(|y|^au)-\text{div}_{Y}(|y|^a\nabla_{Y}u),\quad a\in (-1,1),\eqno{(1.1)}$$
where $\nabla_{Y}=\nabla_{y',y},~Y=(y',y)\in \rr^{n-1}\times \rr,  n\geq 2$.

The degenerate parabolic  equation $\mathscr{L}_a$ has been studied  by many authors; see \cite{A,BD1, BD2,D,Fe,I}.
In particular, this is the so-called extension operator the fractional heat equation, $(\partial_t-\Delta )^su=0~\text{for}~0<s<1$ is a space-time nonlocal equation,
it admits an extension procedure based on which the fractional heat operator can be studied through a local  degenerate parabolic operator $\mathscr{L}_a$  with $a=1-2s$, see \cite{NS,ST}.

A bounded open set $U\subset \rr^{n+1}$ is said to be $\mathscr{L}_a$-regular if for any continuous $f:\p U\to\rr$ there exists a (unique) function $H^U_f$, such that $H^U_f$ satisfies (1.1) in $U$, and for which for every $(Y_0,t_0)\in \p U$ $$\mathop{\lim_{(Y,t)\to(Y_0,t_0)}}_{(Y,t)\in U}H^U_f(Y,t)=f(Y_0,t_0).$$

Given an open subset $\Omega \subset \rr^{n+1}$, a function $w:\Omega\to \overline{\rr}$ is said to be $\mathscr{L}_a$-superparabolic in $\Omega$ if : (i) $-\infty<w\leq +\infty, w<+\infty$ in a dense subset of $\Omega;$ (ii) $w$ is lower semicontinuous; (iii) if $U\subset\overline{U}\subset\Omega$ is a $\mathscr{L}_a$-regular open set, $f\in C(\p U;\rr)$, and $f\leq u$ on $\p U$, then $H^U_f\leq u$ in $U.$
Similarly, $u$ is $\mathscr{L}_a$-subparabolic in $\Omega$ if $-w$ is $\mathscr{L}_a$-superparabolic. We say that $w$ is $\mathscr{L}_a$-parabolic if $w$ is both $\mathscr{L}_a$-subparabolic and $\mathscr{L}_a$-superparabolic.

For an arbitrary bounded open subset $\Omega\subset \rr^{n+1}$, and an arbitrary $f\in C(\p \Omega)$, a generalized solution $u$ satisfying (1.1) in $\Omega$ and $u=f$ on $\p\Omega$,
is provided by the Perron-Wiener-Brelot-Bauer method as $$u=\inf\{w~|~w~\text{is}~\mathscr{L}_a\text{-superparabolic}~\text{in}~\Omega~\text{and}~\liminf w\geq f~\text{on}~\p\Omega\}.\eqno(1.2)$$
It is not true in general that $u$ attains continuously the boundary value $f$. A point $(Y_0,t_0)\in \p \Omega$ is said to be $\mathscr{L}_a$-regular if
$$\lim_{(Y,t)\to(Y_0,t_0)}u(Y,t)=f(Y_0,t_0)$$
for any $f\in C(\p\Omega)$, where $u$ is defined as in (1.2).

The main result is to obtain a ``geometric'' characterization of regular boundary points $\xi_0=(X_0,t_0)$ in this paper. Theorem 1.1 below provides such a criterion, the convergence or divergence of a certain series involving the relative capacity of $\Omega^c$ (complement of $\Omega$) within the regions between various level surfaces of the fundamental solution with pole at $\xi_0=(X_0,t_0)$. We need more definitions to make this precise.

Let $\Gamma(X,t;Y,\tau)$  be the fundamental solution of the degenerate equation $\mathscr{L}_a$; see (2.8) in Section 2.

If $K$ is a closed subset of $\rr^{n+1}$, denote by $M^+(K)$ the collection of all nonnegative Radon measures on $\rr^{n+1}$, with support in $K$. For $\mu\in M^+(\rr^{n+1})$ we may set $$P_\mu(X,t):=\int_{\rr^{n+1}}\Gamma(X,t;Y,\tau)d\mu(Y,\tau),\qquad  (X,t)\in\rr^{n+1};\eqno(1.3)$$
$P_\mu$ is the 
potential of $\mu.$

Let now $K$ denote a compact subset of $\rr^{n+1}$. The $\mathscr{L}_a$-capacity of $K$ is $$\mathrm{cap}(K):=\sup\{\mu(\rr^{n+1})~|~\mu\in M^+(K), P_\mu\leq 1~\text{in}~ \rr^{n+1}\}.\eqno(1.4)$$
Finally for $(X_0,t_0)=(x'_0,x_0,t_0)$
and define for each $\lambda>0,$
$$
\Omega_\lambda(X_0,t_0):=\{(Y,\tau)\in \rr^{n+1}|~\Gamma(X_0,t_0;Y,\tau)>(4\pi \lambda)^{-\frac{n+a}{2}}(1+\frac{x_0^2}\lambda)^{-\frac a2}\}.\eqno(1.5)$$
Set
$$\begin{array}{cl}
&A((X_0,t_0),\lz^k):=\\
&\left\{(Y,\tau)\in\rr^{n+1}|~( 4\pi\lz^{k+1})^{-\frac{n+a}{2}}(1+\frac{x_0^2}{\lz^{k+1}})^{-\frac a2}
\geq \Gamma(X_0,t_0;Y,\tau)\geq (4\pi\lz^{k})^{-\frac{n+a}{2}}(1+\frac{x_0^2}{\lz^{k}})^{-\frac a2}\right\};\end{array}\eqno(1.6)$$
this is the set of points between the level surfaces $\Gamma=(4\pi\lz^{k+1})^{-\frac{n+a}{2}}(1+\frac{x_0^2}{\lz^{k+1}})^{-\frac a2}$ and $\Gamma=(4\pi\lz^{k})^{-\frac{n+a}{2}}(1+\frac{x_0^2}{\lz^{k}})^{-\frac a2}$.

Our criterion for boundary regularity is the following
\begin{theorem}\label{1.1}
Given a bounded open subset $\Omega\subset \rr^{n+1}$, a point $\xi_0=(x'_0,x_0,t_0)\in \p\Omega$ is $\mathscr{L}_a$-regular if and only if  for every $0<\lambda<1$ $$\sum^\infty_{k=1}\lz^{-k(n+a)/2}\l(1+\frac{x_0^2}{\lz^k}\r)^{-\frac a2}\mathrm{cap}(\Omega^c\cap A (\xi_0,\lz^{k}))=+\infty,\eqno(1.7)$$
where $\Omega^c=\rr^{n+1}\setminus\oz$. The behavior of the series in $(1.7)$ does not depend on $\lz\in(0,1)$.
\end{theorem}
We remark that if  (1.7) with $a=0$, which is just the Wiener's criterion for the heat equation proved by Evans and Gariepy in \cite{EG}. In particular, it is worth pointing out that the special case of $x_0=0$ in Theorem 1.1 will give an important  application in \cite{ht}.

The paper is organized as follows. In  Section 2, we  give the fundamental solution of the degenerate equation $\mathscr{L}_a$, which is a basic result. In Section 3,
we  discuss  the Dirichlet problem on space-time boxes, which plays a crucial role  in whole paper. In Section 4, we   introduce and study Perron methods and  barriers.
We  show the potential theory of the operator $\mathscr{L}_a$ in Section 5. In Section 6, we  give  Green functions and establish mean value formulas and a strong Harnack inequality. We  study the
smooth approximation of $\mathscr{L}_a$-superparabolic functions in Section 7.
Finally, we  prove Theorem 1.1 in Section 8.

Throughout this paper, we let $C$ denote  constants that are
independent of the main parameters involved but whose value may
differ from line to line.

\section{Fundamental solution of  $\mathscr{L}_a$}

As in \cite{AC}, we set
$$u(x,y,t)=t^{-\frac {a+1}2}e^{-\frac{(x-y)^2}{4t}}h(\frac{xy}t)$$
is a solution of
$$u_t-u_{yy}-\frac ay u_y=0\quad {\rm for }\ y\not= 0, t>0.\eqno(2.1)$$
Let $z=\frac {xy}t$,   from (2.1), we know that $h(z)$  satisfies the following equation
$$zh''+(z+a)h'+\frac a{2}h=0.\eqno(2.2)$$
From page 100 in \cite{WO}, we know that the equation (2.2)  have two solutions, that is,
$$h_1(z)=F(\frac a2,a,-z)=e^{-\frac z2}\dsum_{k=0}^\fz\frac{(\frac z4)^{2k}}{k!\Gamma(k+\nu+1)},$$
where and in what follows, $\nu=\frac{a-1}2$, and
$$h_2(z)=(\frac z4)^{1-a}F(1-\frac a2,2-a,-z)=(\frac z4)^{1-a}e^{-\frac z2}\dsum_{k=0}^\fz\frac{(\frac z4)^{2k}}{k!\Gamma(k-\nu+1)}.$$
Inspired by these above, by a direct computation, we can show that
$$\displaystyle\wt u(x,y,t)=(4t)^{-\frac {a+1}2}e^{-\frac{(x-y)^2}{4t}}\l(e^{-\frac{xy}{2t}}\dsum_{k=0}^\fz\frac{(\frac{xy}{4t})^{2k}}{k!\Gamma(k+\nu+1)}
+\frac{xy}{4t}\l(\l|\frac{xy}{4t}\r|\r)^{-a}e^{-\frac{xy}{2t}}\dsum_{k=0}^\fz\frac{(\frac{xy}{4t})^{2k}}{k!\Gamma(k-\nu+1)}\r)\eqno(2.3)$$ is a solution of (2.1).

\begin{remark} We   compute directly to show that
$\wt u_x(x,y,t)$ is also  a  solution of $(2.1)$. By symmetry, we know that $\wt u_y(x,y,t)$ is a  solution of the following equation
$$u_t-u_{xx}-\frac ax u_x=0\quad {\rm for }\ x\not= 0, t>0.$$
This fact  will be used in next section.
\end{remark}

Now we estimate $\wt u$. We first recall  the definition of the modified Bessel function of the first kind and order $\nu$, that is,
$$I_\nu(z)=\dsum_{k=0}^\fz\frac{(\frac z2)^{2k+\nu}}{k!\Gamma(k+\nu+1)},\quad z>0.$$
Let $z=|\frac{xy}{t}|$, $\wt u$ can be written as follows
$$\begin{array}{cl}
&\wt u=(4t)^{-\frac {a+1}2}e^{-\frac{(x-y)^2}{4t}}e^{-\frac{xy}{2t}}\l(\dsum_{k=0}^\fz\frac{(\frac{xy}{4t})^{2k}}{k!\Gamma(k+\nu+1)}
+\text{sgn}\l(\frac{xy}{4t}\r)\l(\l|\frac{xy}{4t}\r|\r)^{1-a}\dsum_{k=0}^\fz\frac{(\frac{xy}{4t})^{2k}}{k!\Gamma(k-\nu+1)}\r)\\
&\quad =(4t)^{-\frac {a+1}2}e^{-\frac{x^2+y^2}{4t}}\l(\frac{z}{4}\r)^{-\nu}\l[I_\nu(\frac z2)+\text{sgn}(\frac{xy}{4t})I_{-\nu}(\frac z2)\r].
\end{array}$$
From page 181 in  \cite{WO}, and note that $-1<\nu=\frac {a-1}2<0$, we then show that for $z>0$,
\begin{align*}
 I_\nu(z)-I_{-\nu}(z)&=\dfrac 1\pi\dint_0^\pi e^{z\cos\tz}\cos\nu\tz d\tz-\frac {\sin \nu\pi}\pi\dint_0^{+\fz}e^{-z cht-\nu t}dt\\
 & \quad-\l[\dfrac 1\pi\dint_0^\pi e^{z\cos\tz}\cos \nu\tz d\tz+\frac {\sin \nu\pi}\pi\dint_0^{+\fz}e^{-z cht+\nu t}dt\r]\\
 &\geq 0.
 \end{align*}
From this, we know that  for $x, y\in\rr$ and $t>0$
$$\wt u(x,t;y)\ge 0.$$
From page 203 in \cite{WO}, we know that
$$I_\nu(z)+I_{-\nu}(z)\sim \sqrt{\frac 2\pi} z^{-\frac 12}e^z, \quad {\rm if}\ z \to+\fz$$ and
$$(\frac z2)^{-\nu}(I_\nu(z)+I_{-\nu}(z))\to 1/\Gamma(\frac{\nu+1}2), \quad {\rm if}\ z \to 0^+.$$
In addition, from page 202 in \cite{WO}, we know that
$$I_\nu(z)-I_{-\nu}(z)\sim \sin(-\nu\pi)\sqrt{\frac 2\pi} z^{-\frac 12}e^{-z}, \quad {\rm if}\ z \to+\fz$$ and
$$(\frac z2)^{-\nu}(I_\nu(z)-I_{-\nu}(z))\to 1/\Gamma(\frac{\nu+1}2), \quad {\rm if}\ z \to 0^+.$$
From these, we have for any $x,y\in\rr$ and $t>0$
$$c t^{-\frac {1+a}2}e^{-\frac {|x-y|^2}{4t}}\l(1+\l|\frac{xy}{t}\r|\r)^{-\frac a2}\le\wt u\le C t^{-\frac {1+a}2}e^{-\frac {|x-y|^2}{4t}}\l(1+\l|\frac{xy}{t}\r|\r)^{-\frac a2},\eqno(2.4)$$
where $C, c$ depend only on $ a$.

From (2.4), we can get that
$$\begin{array}{cl}
&C_1 t^{-\frac {1+a}2}e^{-\frac {|x-y|^2}{2t}}\max\{(1+|\frac{x^2}{t}|)^{-\frac a2},(1+|\frac{y^2}{t}|)^{-\frac a2}\}\\
&\qquad\qquad\qquad\le \wt u\le C_2 t^{-\frac {1+a}2}e^{-\frac {|x-y|^2}{6t}}\min\{(1+|\frac{x^2}{t}|)^{-\frac a2},(1+|\frac{y^2}{t}|)^{-\frac a2}\},\end{array}\eqno(2.5)$$
where $C_1, C_2$ depend only on $ a$.

Next, we show the semigroup property of $\wt u$.

From Proposition 2.3 in \cite{G}, let $a>-1$, for any $(x,t)\in \rr\times(0,+\fz)$ one has
$$P_t^{(a)}(x):=\int_{-\fz}^{+\fz}\wt u(x,y,t)|y|^ady=1.\eqno(2.6)$$

Adapting the same argument in  Proposition 2.4 in \cite{G}, let $a>-1$, for every $x,\eta\in\rr$ and every $0<s,t<+\fz$ one has
$$\wt u(x,\eta,t+s)=\int_{-\fz}^{+\fz}\wt u(x,y,t)\wt u(y,\eta,s)|y|^ady.\eqno(2.7)$$

An immediately consequence of (2.7), for $a>-1$ and every $0<s,t<+\fz$ we have the following semigroup property:
$$P_t^{(a)}\circ P_s^{(a)}=P_{t+s}^{(a)}.$$

Let us turn to introduce the fundamental solution of  $\mathscr{L}_a$:

let $X=(x',x), Y=(y',y)\in\rr^n$, we set \begin{align*}
\Gamma(X,t;Y,\tau)&=(4\pi(t-\tau))^{-\frac{n-1}2}e^{-\frac{|x'-y'|^2}{4(t-\tau)}}(4(t-\tau))^{-\frac {a+1}2}e^{-\frac{(x-y)^2}{4(t-\tau)}}e^{-\frac{xy}{2(t-\tau)}}
\\
&\quad\times\l(\dsum_{k=0}^\fz\frac{\l(\frac{xy}{4(t-\tau)}\r)^{2k}}{k!\Gamma(k+\nu+1)}
+\frac{xy}{4(t-\tau)}\l(\frac{|xy|}{4(t-\tau)}\r)^{-a}\dsum_{k=0}^\fz\frac{\l(\frac{xy}{4(t-\tau)}\r)^{2k}}{k!\Gamma(k-\nu+1)}\r)\\\tag{2.8}
&=2^{-1-a}(4\pi)^{-\frac{n-1}2}(t-\tau)^{-\frac{n+a}2}e^{-\frac{|X-Y|^2}{4(t-\tau)}}e^{-\frac{xy}{2(t-\tau)}}\\
&\quad\times \l(\dsum_{k=0}^\fz\frac{\l(\frac{xy}{4(t-\tau)}\r)^{2k}}{k!\Gamma(k+\nu+1)}
+\frac{xy}{4(t-\tau)}\l(\frac{|xy|}{4(t-\tau)}\r)^{-a}\dsum_{k=0}^\fz\frac{\l(\frac{xy}{4(t-\tau)}\r)^{2k}}{k!\Gamma(k-\nu+1)}\r)\\
&:=C_{n,a}(t-\tau)^{-\frac{n+a}2}e^{-\frac{|X-Y|^2}{4(t-\tau)}}F\l(\dfrac{xy}{t-\tau}\r),\end{align*}
for  $t>\tau$, otherwise it is 0. Here and in what follows,  $C_{n,a}=2^{-1-a}(4\pi)^{-\frac{n-1}2}$.

From (2.3) and (2.8), it is easy to see that $\Gamma(X,t;Y,\tau)$ satisfies
$$|y|^aD_\tau\Gamma +\text{div}(|y|^a\nabla_Y\Gamma)=|y|^a(D_\tau\Gamma+\Delta_Y\Gamma+\frac a yD_y\Gamma)=0,\ y\not= 0.\eqno(2.9)$$

\begin{theorem}\label{2.1}
 $\Gamma(X,t;Y,\tau)$ defined in $(2.8)$ is the fundamental solution of  $\mathscr{L}_a$, that is, $\Gamma(X,t;Y,\tau)$ satisfies the following Cauchy problem:
\begin{enumerate}
\item[$\mathrm{(i)}$] $\Gamma$ is a weak solution of $\mathscr{L}_a$, that is,  if $\Gamma$ satisfies  $$\dint_{-\infty}^t\dint_\rn|y|^a(\Gamma D_\tau\vz+\nabla_Y\Gamma_Y\nabla\vz)dYd\tau=0$$
for every test function $\vz\in C_0^1(\rz\times(-\fz,t))$, where $C_0^1(E)$ denotes the set of $C^1$ functions with compact support in $E$, and
\item[$\mathrm{(ii)}$]$$\dlim_{t\to\tau^+}\dint_\rn\Gamma(X,t;Y,\tau)g(Y)|y|^adY=g(X)$$
for all $g\in C(\rz)\cap L^\fz(\rn)$.

\end{enumerate}

\end{theorem}

\begin{proof}
To prove (i). Note that for $\forall\ \ez>0$
\begin{align*}
\dint_{-\fz}^t\dint_\rn|y|^a(\Gamma D_\tau\vz+\nabla_Y\Gamma\nabla_Y\vz)dYd\tau&=\dint_{-\fz}^t\dint_{|y|\leq\ez}|y|^a(\Gamma D_\tau\vz+\nabla_Y\Gamma\nabla_Y\vz)dYd\tau\\
&\quad
+\dint_{-\fz}^t\dint_{|y|>\ez}|y|^a(\Gamma D_\tau\vz+\nabla_Y\Gamma\nabla_Y\vz)dYd\tau\\
&=I_1^\ez+I_2^\ez.
\end{align*}
Obviously,
$$\dlim_{\ez\to 0^+}I_1^\ez=0.\eqno(2.10)$$
For $I_2^\ez$, by (2.9) and divergence theorem, we have
\begin{align*}
I_2^\ez &=\dint_{-\fz}^t\dint_{|y|>\ez}-(|y|^a D_\tau\Gamma +{\rm div}(|y|^a\nabla_Y\Gamma))\vz+{\rm div}(|y|^a\nabla_Y\Gamma\vz)dYd\tau\\
&=\dint_{-\fz}^t\dint_{|y|>\ez}{\rm div}(|y|^a\nabla_Y\Gamma\vz)dYd\tau\\
&=\dint_{-\fz}^t\dint_{y=\ez}|y|^a D_y\Gamma\vz dy'd\tau-\dint_{-\fz}^t\dint_{y=-\ez}|y|^a D_y\Gamma\vz dy'd\tau.
\end{align*}
Notice that
$$\lim_{\ez\to 0^+}|\ez|^a D_y\Gamma(X,t;\ez,\tau)\vz(X,t;\ez,\tau)=\lim_{\ez\to 0^+}|\ez|^a D_y\Gamma(X,t;-\ez,\tau)\vz(X,t;-\ez,\tau).$$
Hence,
$$\dlim_{\ez\to 0^+}I_2^\ez=0.\eqno(2.11)$$
Combining (2.10) and (2.11), we prove (i).

Let us turn to prove (ii).  In fact, we will give a slight stronger result as follows:
$$\dlim_{t\to\tau^+,~X\to X_0}\dint_\rn\Gamma(X,t;Y,\tau)g(Y)|y|^adY=g(X_0).\eqno(2.12)$$
By (2.6), we have
$$\dint_\rn\Gamma(X,t;Y,\tau)|y|^adY=1.\eqno(2.13)$$
Let $u(X,t,\tau)=\int_\rn\Gamma(X,t;Y,\tau)g(Y)|y|^adY$, $B(X_0,r)=\{X\in\rn: |X-X_0|<r\}$ and $B(X_0,r)^c=\rn\setminus B(X_0,r)$.
Fix $X_0\in\rn,\ \ez>0$, choose $\dz>0$ such that if $|Y-X_0|<\dz$, then $ |g(Y)-g(X_0)|<\ez$. If $|X-X_0|<\frac\dz 2$, by (2.13), we then have

\begin{align*}
|u(X,t,\tau)-g(X_0)|&=\l|\dint_\rn\Gamma(X,t;Y,\tau)[g(X_0)-g(Y)]|y|^adY\r|\\
&\le \dint_{B(X_0,\dz)}\Gamma(X,t;Y,\tau)|g(X_0)-g(Y)||y|^adY\\
&\quad+\dint_{B(X_0,\dz)^c}\Gamma(X,t;Y,\tau)|g(X_0)-g(Y)||y|^adY\\
&:=J_1+J_2.
\end{align*}
For $J_1$, obviously, we have
$$J_1\le\ez \dint_\rn\Gamma(X,t;Y,\tau)|y|^adY=\ez.$$
For $J_2$, note that $|Y-X|\geq \frac12|Y-X_0|$ if $Y\in B(X_0,\dz)^c$. Consequently, by (2.5), we have
\begin{align*}
J_2&\le 2\|g\|_{L^\fz(\rn)}\dint_{B(X_0,\dz)^c}\Gamma(X,t;Y,\tau)|y|^adY\\
&\le \dfrac C{(t-\tau)^{\frac {n+a}2}}\dint_{B(X_0,\dz)^c}e^{-\frac {|X-Y|^2}{6(t-\tau)}}\l(1+\frac{y^2}{t-\tau}\r)^{-\frac a2}|y|^adY.
\end{align*}
If $a\ge 0$, then
$$
J_2\le \frac C{(t-\tau)^{\frac {n}2}}\dint_{B(X_0,\dz)^c}e^{-\frac {|X-Y|^2}{6(t-\tau)}}dY\le C\dint_{B(X_0,\frac\dz{\sqrt{t-\tau}})^c}e^{-\frac {|X-Y|^2}{6}}dY.
$$
So,
$$J_2\to 0,\quad {\rm if}\ t\to \tau^+.$$
If $a<0$, we have
\begin{align*}
J_2&\le \dfrac C{(t-\tau)^{\frac {n+a}2}}\dint_{B(X_0,\dz)^c}e^{-\frac {|X-Y|^2}{6(t-\tau)}}\l(1+\frac{y^2}{t-\tau}\r)^{-\frac a2}|y|^adY\\
&=\dfrac C{(t-\tau)^{\frac {n+a}2}}\dint_{B(X_0,\dz)^c\cap\{y^2\le t-\tau\}}e^{-\frac {|X-Y|^2}{6(t-\tau)}}\l(1+\frac{y^2}{t-\tau}\r)^{-\frac a2}|y|^adY\\
&\quad+\dfrac C{(t-\tau)^{\frac {n+a}2}}\dint_{B(X_0,\dz)^c\cap\{y^2> t-\tau\}}e^{-\frac {|X-Y|^2}{6(t-\tau)}}\l(1+\frac{y^2}{t-\tau}\r)^{-\frac a2}|y|^adY\\
&:=J_2^1+J_2^2.
\end{align*}
For $J_2^1$, we have
\begin{align*}
J_2^1&\le \dfrac C{(t-\tau)^{\frac {n+a}2}}\dint_{B(X_0,\dz)^c\cap\{y^2\le t-\tau\}}e^{-\frac {|X-Y|^2}{6(t-\tau)}}|y|^adY\\
&\le \dfrac {Ce^{-\frac {\dz^2}{48(t-\tau)}}}{(t-\tau)^{\frac {n+a}2}}\dint_{\{y^2\le t-\tau\}}e^{-\frac {|X-Y|^2}{12(t-\tau)}}|y|^adY\\
&\le \dfrac {Ce^{-\frac {\dz^2}{48(t-\tau)}}}{(t-\tau)^{\frac {n+a}2}}\dint_{\{y^2\le t-\tau\}}|y|^ady\int_{\rr^{n-1}}e^{-\frac {|x'-y'|^2}{12(t-\tau)}}dY'\\
&\le Ce^{-\frac {\dz^2}{(t-\tau)}}(t-\tau)^{-\frac {n-1}2}\dint_{\rr^{n-1}}e^{-\frac {|x'-y'|^2}{12(t-\tau)}}dY'\le Ce^{-\frac {\dz^2}{48(t-\tau)}}.
\end{align*}
So,
$$J^1_2\to 0,\quad {\rm if}\ t\to \tau^+.$$
For $J_2^2$, we have
$$
J_2^2\le \frac C{(t-\tau)^{\frac {n}2}}\dint_{B(X_0,\dz)^c}e^{-\frac {|X-Y|^2}{6(t-\tau)}}dY\le C\dint_{B(X_0,\frac\dz{\sqrt{t-\tau}})^c}e^{-\frac {|X-Y|^2}{6}}dY.
$$
So,
$$J_2^2\to 0,\quad {\rm if}\ t\to \tau^+.$$
Thus, (2.12) is proved, so (ii) is also proved.
\end{proof}

We finally prove the uniqueness of weak solutions within the restricted class of energy solutions, i.e., weak solutions satisfying the additional assumption for any $T_1<T_2$ and $\oz\in\rn$
$$u\in L^2((T_1,T_2); H^1(\oz,|x|^a)),\eqno(2.14)$$
that is, if $u$ satisfies
$$
\l(\dint_{T_1}^{T_2}\dint_\oz |u(X,t)|^2|x|^adXdt\r)^{\frac12}+\l(\dint_{T_1}^{T_2}\dint_\oz |\nabla_X u(X,t)|^2|x|^adXdt\r)^{\frac12}<\fz
.$$

\begin{theorem}\label{2.2}
Let $g\in L^1(\rz)$. Cauchy problem in Theorem $2.1$ has at most one weak solution in the class of functions satisfying hypothesis $(2.14)$ with $T_1=0,T_2=T$ and $\oz=\rn$.
\end{theorem}
\begin{proof}
Let $u$ and $\widetilde{u}$ be two weak solutions to Cauchy problem in Theorem 2.1, we take the following as a test function in the weak formulation:
$$\vz(X,t)=\int_t^T(u-\widetilde{u})(X,s)ds,\quad 0\le t\le T,$$ with $\vz=0$ for $t>T$. Then
$$\begin{array}{cl}
&\dint_0^T\dint_\rn|x|^a(u-\widetilde{u})^2(X,s)dXdt\\
&\qquad+\dint_0^T\dint_\rn|x|^a\nabla_X(u-\widetilde{u})(X,t)\int_t^T\nabla_X(u-\widetilde{u})(X,s)dsdXdt=0,
\end{array}$$
which turns into
$$\dint_0^T\dint_\rn|x|^a(u-\widetilde{u})^2(X,s)dXdt+\frac 12\dint_0^T\dint_\rn|x|^a|\nabla_X(u-\widetilde{u})(X,t)|^2dXdt=0.
$$
Since both integrands are nonnegative, they must be identically 0, and so $u=\widetilde{u}$.
\end{proof}

\section{The Dirichlet problem on   the space-time box}
 We consider a cube $Q=(a_1,b_1)\times(a_2, b_2)\times\cdots\times(a_n,b_n)$ in $\rn$, and a bounded time interval $[0,T]$. We denote  $\oz$ by the space-time box
  $Q_T=Q\times[0,T]$, by $L$ the lateral surface $\pz Q\times (0,T]$, and $I$ the initial surface $ \overline{Q}\times\{0\}$, so that the normal boundary $\pz_p Q_T=L\cup I$.
 The Dirichlet problem on $Q$ consists of showing
that, for any continuous function $f$ on $\pz_p Q_T$, there is a $\mathscr{L}_a$-parabolic  $ u$ on $Q_T$ which
has a continuous extension by $f$ to $\pz_p Q_T$.

To state the main result in this section, we first give the definition of the $\mathscr{L}_a$-parabolic   on the general domain $\oz$.

\begin{definition}
Let $\oz$ be a domain in $\rr^{n+1}$ and suppose that $u\in L^2((t_1,t_2),H^1(Q,|x|^a))$ whenever the closure of $Q\times(t_1,t_2)$ is contained in $\oz$.
Then $u$ is called a weak solution in $\oz$ of the $\mathscr{L}_a$-parabolic equation if
$$\iint_\oz|x|^a(-u\vz_t+\nabla_X u\cdot \nabla \vz)dtdX=0$$
whenever $\vz\in C_0^1(\oz)$. If, in addition, $u$ is continuous, then $u$ is called $\mathscr{L}_a$-parabolic in $\oz$. Further, we say that $u$ is a supersolution of $\mathscr{L}_a$ if
the above integral is  nonnegative whenever $\vz\in C_0^1(\oz)$ is nonnegative. A function $v$ is a subsolution of $\mathscr{L}_a$ if $-v$ is a supersolution.
\end{definition}
We remark that the definition  of $\mathscr{L}_a$-parabolic here coincides to that defined in Section 1, we will show this in Theorem 4.1 and Proposition 4.3 later.

The main result in this section is as follows.
\begin{theorem}
Let $f\in C(\pz_p Q_T)$. Then there is a function $u\in C(\overline{ Q_T})$ such that $u$ is a $\mathscr{L}_a$-parabolic in $\overline{\oz}\setminus\pz_p\oz$ and $u=f$ on $\pz_p Q_T$.
\end{theorem}
\begin{remark}Using Theorem $3.1$ repeatedly, one easily extends the existence results:
suppose that $\oz$ is a union of finitely many boxes $Q_i\times(t_i,T_i)$ and that $f$ is a continuous function on the parabolic boundary ${\cal L}$ of $\oz$. Then there is a unique
$\mathscr{L}_a$-parabolic function $u$, continuous on $ \overline{\oz}$. To verify this, one just has to begin with the earliest boxes.
 Here the parabolic boundary ${\cal L}$ is understood to be that
part of the Euclidean boundary $\pz\oz$ of $\oz$ that lies in the union of the parabolic boundaries of the boxes $Q_i\times(t_i,T_i)$.
\end{remark}

The proof of Theorem 3.1 is long and complicated, so we need to establish some lemmas.

We  first show that it is enough to prove the first part of the theorem with the addition hypothesis that
$f=0$ on $I$.

\begin{lemma}
 Suppose that, given any function $g\in(\pz_p Q_T)$ such that $g=0$ on $I$, there is a $\mathscr{L}_a$-parabolic $u_g$ on $Q_T$ such that
 $$ \dlim_{(X,t)\to (Y,s)}u_g(X,t)=g(Y,s),\quad (Y,s)\in\pz_p Q_T.  $$
 Then, given an arbitrary function $f\in C(\pz_p Q_T)$, there is a $\mathscr{L}_a$-parabolic $u_f$ on $Q_T$ such that
 $$ \dlim_{(X,t)\to (Y,s)}u_f(X,t)=f(Y,s),\quad (Y,s)\in\pz_p Q_T.  $$
\end{lemma}
\begin{proof}
The proof of the lemma is standard by using (ii) of Theorem 2.1. We omit the details here.
\end{proof}

In order to deal with the boundary condition we will need the double-layer potential is given by
$$u(X,t):=\dint_0^t\dint_{\pz Q}\frac{\pz \Gamma(X,t;Y,\tau)}{\pz v(Y)}\vz(Y,\tau)|y|^ad\sz(Y)d\tau.\eqno(3.1)$$
The double-layer potential is well defined  for $X\in S(0,T]=\rn\times(0,T]$ with the time integral to understood
as an improper integral with respect to the upper limit (see Lemma 3.2), and we assume the
unit normal vector $v$ to the boundary $\pz Q$ to be directed into the exterior of $Q$.

\begin{lemma}
 The double-layer potential $u$ with $\vz\in C{(\pz Q\times[0,T])}$ is  well define and finite on $S(0,T]$. More precisely, there exists a constant $C$ depending only on $n,a$ such that $|u(X,t)|\le C$ for any $(X,t)\in S(0,T]$.
\end{lemma}
\begin{proof}
We first recall the definition of $\Gamma(X,t;Y,\tau)$ (see (2.8)):
$$\Gamma(X,t;Y,\tau)=C_{n,a}(t-\tau)^{-\frac{n+a}2}e^{-\frac{|X-Y|^2}{4(t-\tau)}}F\l(\dfrac{xy}{t-\tau}\r),$$
where
\begin{align*}
F\l(\dfrac{xy}{t-\tau}\r)&=e^{-\frac{xy}{2(t-\tau)}} \l(\dsum_{k=0}^\fz\frac{\l(\frac{xy}{4(t-\tau)}\r)^{2k}}{k!\Gamma(k+\nu+1)}
+\frac{xy}{4(t-\tau)}\l(\frac{|xy|}{4(t-\tau)}\r)^{-a}\dsum_{k=0}^\fz\frac{\l(\frac{xy}{4(t-\tau)}\r)^{2k}}{k!\Gamma(k-\nu+1)}\r)\\
&:=\displaystyle e^{-\frac{xy}{2(t-\tau)}}\displaystyle\l(\frac {z}{4}\r)^{-\nu}\l(I_\nu\l(\frac z2\r)+\text{sgn}(xy)I_{-\nu}\l(\frac z2\r)\r),
\end{align*}
where $z=|\frac{xy}{t-\tau}|$ and $\nu=\frac {a-1}2$.

Thus,  we have
\begin{align*}
\nabla_Y\Gamma(X,t;Y,\tau)&=C_{n,a}(t-\tau)^{-\frac{n+a}2}\l(\dfrac{X-Y}{2(t-\tau)}\r)e^{-\frac{|X-Y|^2}{4(t-\tau)}}F\l(\dfrac{xy}{t-\tau}\r)\\\tag{3.2}
&\quad+C_{n,a}(t-\tau)^{-\frac{n+a}2}e^{-\frac{|X-Y|^2}{4(t-\tau)}}F'\l(\dfrac{xy}{t-\tau}\r),
\end{align*}
where $F'=D_yF$ for $z>0$, and
\begin{align*}
F'\l(\dfrac{xy}{t-\tau}\r)&=-\dfrac x{2(t-\tau)}e^{-\frac{xy}{2(t-\tau)}}\l(\frac z4\r)^{-\nu}\l(I_\nu\l(\frac z2\r)+\text{sgn}(xy)I_{-\nu}\l(\frac z2\r)\r)\\\tag{3.3}
&\quad+\dfrac x{2(t-\tau)}e^{-\frac{xy}{2(t-\tau)}}\l(\frac z4\r)^{-\nu}\l(\text{sgn}(xy)I_{\nu+1}\l(\frac z2\r)+I_{-\nu-1}\l(\frac z2\r)\r)\\
&:=-\dfrac x{2(t-\tau)}H.
\end{align*}
Next, we only consider the case $x\not=0$, if $x=0$, then it will be more simple.

Recall $Q=(a_1,b_1)\times(a_2, b_2)\times\cdots\times(a_n,b_n)$ in $\rn$. Set $\pz Q_i=\{y_i=a_i,\ b_i\}\cap\overline{ Q}$ for $i=1,2,\cdots,n-1$, and
$\pz Q_n=\{y=a_n,\ b_n\}\cap \overline{Q}$, and $d=\text{diam}(Q)$.

Let $(X,t)\in \rn\times(0,T]$, $X$ projects onto hyperplane $\pz Q_i$ is $Z_i$, and  let $v_i(Y)$ denote out normal vector of $\pz Q_i$ with $i=1,2,\cdots,n$. Applying (3.2) and (3.3), it is easy to show that
$$\l|\dfrac{\pz \Gamma(X,t;Y,\tau)}{\pz v_i(Y)}\r|\le C(t-\tau)^{-\frac{n+a}2}\l(\dfrac{|X-Z_i|}{t-\tau}\r)e^{-\frac{|X-Y|^2}{4(t-\tau)}}F\l(\dfrac{xy}{t-\tau}\r)
:=I_1,\quad i=1,2,\cdots,n-1,\eqno(3.4)$$ and
\begin{align*}\tag{3.5}
\l|\dfrac{\pz \Gamma(X,t;Y,\tau)}{\pz v_n(Y)}\r|&\le C(t-\tau)^{-\frac{n+a}2}\l(\dfrac{|X-Z_n|}{t-\tau}\r)e^{-\frac{|X-Y|^2}{4(t-\tau)}}F\l(\dfrac{xy}{t-\tau}\r)\\
&\quad +C(t-\tau)^{-\frac{n+a}2}\dfrac {|x|}{(t-\tau)}e^{-\frac{|X-Y|^2}{4(t-\tau)}}|H|:=I_2
\end{align*}
for $y\not=0$,  we define for $y=0$ as follows

\begin{align*}\tag{3.6}\dfrac{\pz \Gamma(X,t;Y,\tau)|y|^a}{\pz v_n(Y)}\bigg|_{y=0}&:=\dlim_{y\to 0}\l[\frac{\pz \Gamma(X,t;Y,\tau)}{\pz v_n(Y)}
|y|^a\r]\\
&=C_{n,a}(1-a)(t-\tau)^{-\frac {n+a}2}\dfrac x{t-\tau}\l(\dfrac {|x|}{t-\tau}\r)^{-a}e^{-\frac{|x'-y'|^2}{4(t-\tau)}}e^{-\frac{x^2}{4{(t-\tau)}}}.
\end{align*}

We first consider the case  $i=1,2,\cdots, n-1$. Then, by (3.4), we have
$$\dint_0^t\dint_{\pz Q_i}\l|\frac{\pz \Gamma(X,t;Y,\tau)}{\pz v_i(Y)}\vz(Y,\tau)\r||y|^ad\sz(Y)d\tau\le C
\dint_0^t\dint_{\pz Q_i}I_1|y|^ad\sz(Y)d\tau:=J_1.$$

For $J_1$, we consider two cases of $a$.

Case 1:  $0\le a<1$. By (2.5),  note that
$|X-Y|^2=|X-Z_i|^2+|Z_i-Y|^2$, we obtain
\begin{align*}
J_1&\le C\dint_0^T\dint_{\pz Q_i} {|X-Z_i|}{}t^{-\frac {n+2}2}e^{-\frac{|X-Y|^2}{6t}}\l(\frac{y^2}t\r)^{\frac a2}\l(1+\frac{y^2}t\r)^{-\frac a2}d\sz(Y)dt\\
&\le C\dint_0^T\dint_{\pz Q_i} {|X-Z_i|}{}t^{-\frac {n+2}2}e^{-\frac{|X-Y|^2}{6t}}d\sz(Y)dt\\
&\le C\dint_0^{|X-Z_i|^2}\dint_{\pz Q_i}\frac {|X-Z_i|}{|X-Y|^{n+2}}d\sz(Y)dt\\
&\quad+ C\dint_{|X-Z_i|^2}^Tt^{-\frac 54}\dint_{\pz Q_i}\frac {|X-Z_i|}{|X-Y|^{n-\frac 12}}d\sz(Y)dt\\
&\le C\dint_0^{|X-Z_i|^2}\dint_0^{d}\frac {|X-Z_i|\rho^{n-2}}{(|X-Z_i|^2+\rho^2)^{\frac{n+2}2}}d\rho dt\\
&\quad+ C\dint_{|X-Z_i|^2}^Tt^{-\frac 54}\dint_0^{d}\frac {|X-Z_i|\rho^{n-2}}{(|X-Z_i|^2+\rho^2)^{n-\frac 12}}d\rho dt\\
&\le C\dint_0^{+\fz}\dfrac{\lz^{n-2}d\lz}{(1+\lz^2)^{\frac{n+2}2}}+C\dint_0^{+\fz}\dfrac{\lz^{n-2}d\lz}{(1+\lz^2)^{\frac{n-\frac 12}2}}\le C.
\end{align*}

Case 2:  $-1<a<0$.  By (2.5), we have
\begin{align*}
J_1&\le C\dint_0^T\dint_{\pz Q_i}{|X-Z_i|}{}t^{-\frac {n+2}2}e^{-\frac{|X-Y|^2}{6t}}\l(\frac{y^2}t\r)^{\frac a2}\l(1+\frac{y^2}t\r)^{-\frac a2}d\sz(Y)dt\\
&\le C\dint_0^T\dint_{\{y^2<t\}\cap\pz Q_i}{|X-Z_i|}{}t^{-\frac {n+2}2}e^{-\frac{|X-Y|^2}{6t}}\l(\frac{y^2}t\r)^{\frac a2}d\sz(Y)dt\\
&\quad+C\dint_0^T\dint_{\{y^2\ge t\}\cap\pz Q_i}{|X-Z_i|}{}t^{-\frac {n+2}2}e^{-\frac{|X-Y|^2}{6t}}d\sz(Y)dt\\
&:=J_{11}+J_{12}.
\end{align*}
For $J_{11}$, note that
$|X-Y|^2=|X-Z_i|^2+|Z_i-Y|^2$, we get
\begin{align*}
J_{11}&\le C|X-Z_i|\dint_{|X-Z_i|^2}^T\dint_{\{y^2<t\}\cap\pz Q_i}t^{-\frac {n+a+2}2}e^{-\frac{|X-Y|^2}{6t}}|y|^{a}d\sz(Y)dt\\
&\quad + C|X-Z_i|\dint_0^{|X-Z_i|^2}\dint_{\{y^2<t\}\cap\pz Q_i}t^{-\frac {n+a+2}2}e^{-\frac{|X-Y|^2}{6t}}|y|^{a}d\sz(Y)dt\\
&\le C|X-Z_i|\dint_{|X-Z_i|^2}^Tt^{-\frac 32}dt+ C|X-Z_i|\dint_0^{|X-Z_i|^2}t^{-\frac 12}|X-Z_i|^{-2}dt\\
&\le C.
\end{align*}
For $J_{12}$, similar to the proof of Case 1 in $J_1$, we obtain
$$J_{12}\le C\dint_0^T\dint_{\pz Q_i}{|X-Z_i|}{}t^{-\frac {n+2}2}e^{-\frac{|X-Y|^2}{6t}}d\sz(Y)dt\le C.$$

We now consider the case $i=n$. We need to consider two cases about $y$.

Case 1:  $y\not= 0$. Without of loss generality, we assume $y=a_n\not= 0$.

Subcase 1: $|x|\le 2|a_n|$.
By (3.5), we have
$$\dint_0^t\dint_{\pz Q_n}\l|\frac{\pz \Gamma(X,t;Y,\tau)}{\pz v_n(Y)}\vz(Y,\tau)\r||y|^ad\sz(Y)d\tau\le C
\dint_0^t\dint_{\pz Q_n}I_2|y|^ad\sz(Y)d\tau:=J_2.$$
Thus,
\begin{align*}
J_2&\le C\dint_0^t\dint_{\pz Q_n}\l|\dfrac{\pz \Gamma(X,t;Y,\tau)}{\pz v_n(Y)}\r||y|^{a}d\sz(Y)d\tau\\
&\le C\dint_0^t\dint_{\pz Q_n}(t-\tau)^{-\frac{n+a}2}\l(\dfrac{|X-Z_n|}{t-\tau}\r)e^{-\frac{|X-Y|^2}{4(t-\tau)}}F\l(\dfrac{xy}{t-\tau}\r)|y|^{a}d\sz(Y)d\tau\\
&\quad +C\dint_0^t\dint_{\pz Q_n}(t-\tau)^{-\frac{n+a}2}\dfrac {|x|}{(t-\tau)}e^{-\frac{|X-Y|^2}{4(t-\tau)}}|H||y|^{a}d\sz(Y)d\tau\\
&:=J_{21}+J_{22}.
\end{align*}
 For $J_{21}$, by (2.5),  we then  have
\begin{align*}
J_{21}&\le C\dint_0^T\dint_{\pz Q_n}{|X-Z_n|}{}t^{-\frac {n+2}2}e^{-\frac{|X-Y|^2}{6t}}\l(\frac{a_n^2}t\r)^{\frac a2}\l(1+\frac{a_n^2}t\r)^{-\frac a2}d\sz(Y)dt\\
&\le C\dint_0^{a_n^2}\dint_{ \pz Q_n}{|X-Z_n|}{}t^{-\frac {n+2}2}e^{-\frac{|X-Y|^2}{6t}}d\sz(Y)dt\\
&\quad+C\dint_{a_n^2}^{T}\dint_{\pz Q_n}{|X-Z_n|}{}t^{-\frac {n+2}2}e^{-\frac{|X-Y|^2}{6t}}\l(\frac{a_n^2}t\r)^{\frac a2}d\sz(Y)dt\\
&:=J_{21}^1+J_{21}^2.
\end{align*}
Similar to the proof of Case 1 in $J_1$, we have
$$J_{21}^1\le C\dint_0^{T}\dint_{ \pz Q_n}{|X-Z_n|}{}t^{-\frac {n+2}2}e^{-\frac{|X-Y|^2}{6t}}d\sz(Y)dt\le C.$$
If $0\le a<1$ for  $J^2_{21}$ ,  similar to the proof of Case 1 in $J_1$, we also get
$$J^2_{21}\le C\dint_0^T\dint_{\pz Q_n}{|X-Z_n|}{}t^{-\frac {n+2}2}e^{-\frac{|X-Y|^2}{6t}}d\sz(Y)dt\le C.$$
If $-1<a< 0$ for  $J^2_{21}$, let $\ez=\frac {1+a}2>0$ such that $-1<a-\ez=\frac {a-1}2< 0$, note that $|X-Z_n|=|x-a_n|\le 3|a_n|$, we then have
\begin{align*}
J^2_{21}&\le C|a_n|^a\dint_{a_n^2}^T\dint_{\pz Q_n}{|X-Z_n|}{}t^{-\frac {n+2+a}2}e^{-\frac{|X-Y|^2}{6t}}d\sz(Y)dt\\
&\le C |a_n|^a\dint_{a_n^2}^T\dint_0^{d}\frac {t^{-1-\frac \ez2}|X-Z_n|\rho^{n-2}}{(|X-Z_n|^2+\rho^2)^{\frac{n-\ez+a}2}}d\rho dt\\
&\le C |a_n|^{a-\ez}|X-Z_n|^{\ez-a}\le  C |a_n|^{a-\ez}|3a_n|^{\ez-a}\le C.
\end{align*}

Now estimate $J_{22}$. Notice that (see page 202 in \cite{WO}),
$$|H|\le C\l[1+\l(\frac{|xy|}{t-\tau}\r)^{-a}\r],\quad {\rm if}\quad  \frac{|xy|}{t-\tau}\le 1, $$
and
$$|H|\le C\l(1+\frac{|xy|}{t-\tau}\r)^{-1-\frac a2},\quad {\rm if}\quad  \frac{|xy|}{t-\tau}\ge 1. $$

From these and (2.5),  we have
\begin{align*}
J_{22}&\le C\dint_{|xa_n|}^T\dint_{\pz Q_n}t^{-\frac n2}e^{-\frac{|X-Y|^2}{6t}}\frac{|x|}t\l(\frac{a_n^2}t\r)^{\frac a2}d\sz(Y)dt\\
&\quad+C\dint_{|xa_n|}^T\dint_{\pz Q_n}t^{-\frac n2}e^{-\frac{|X-Y|^2}{6t}}\frac{|x|}t\l(\frac{a_n^2}t\r)^{\frac a2}\l(\frac{|xa_n|}{t}\r)^{-a}d\sz(Y)dt\\
&\quad+C\dint_0^{|xa_n|}\dint_{\pz Q_n}t^{-\frac n2}e^{-\frac{|X-Y|^2}{6t}}\frac{|x|}t\l(\frac{a_n^2}t\r)^{\frac a2}\l(1+\frac{|xa_n|}t\r)^{-1-\frac a2}d\sz(Y)dt\\
&:=J^1_{22}+J^2_{22}+J^3_{22}.
\end{align*}

For $J^1_{22}$, note that $|a_n|\ge|x|/2$,  we have
\begin{align*}
J^1_{22}&\le C\dint_{|xa_n|}^T\dint_{\pz Q_n}t^{-\frac n2}e^{-\frac{|X-Y|^2}{6t}}\frac{|x|}t\l(\frac{a_n^2}t\r)^{\frac a2}d\sz(Y)dt\\
&\le C|x||a_n|^a\dint_{|xa_n|}^T t^{-\frac{3+a}2}dt\le C|x||a_n|^a|xa_n|^{-\frac{1+a}2}\le C\l(\frac{|x|}{|a_n|}\r)^{\frac {1-a}2}\le C.
\end{align*}

For $J^2_{22}$, note that $|a_n|\ge|x|/2$, we have
\begin{align*}
J^2_{22}&\le C\dint_{|xa_n|}^T\dint_{\pz Q_n}t^{-\frac n2}e^{-\frac{|X-Y|^2}{6t}}\frac{|x|}t\l(\frac{a_n^2}t\r)^{\frac a2}\l(\frac{|xa_n|}t\r)^{-a}d\sz(Y)dt\\
&\le C|x|^{1-a}\dint_{|x|^2/2}^Tt^{-1+\frac{a-n}2}\dint_{\pz Q_n}e^{-\frac{|X-Y|^2}{6t}}d\sz(Y)dt\\
&\le C|x|^{1-a}\dint_{|x|^2/2}^Tt^{-1+\frac{a-1}2}dt\le C.
\end{align*}

For $J^3_{22}$, note that $|a_n|\ge|x|/2$, we then have
\begin{align*}
J^3_{22}&\le C|x|\dint_0^{|xa_n|}\dint_{\pz Q_n}t^{-1-\frac n2}e^{-\frac{|X-Y|^2}{8t}}\l(\frac{a_n^2}t\r)^{\frac a2}\l(1+\frac{|a_n^2|}t\r)^{-1-\frac a2}d\sz(Y)dt\\
&\le C|x|\dint_0^{|xa_n|}\dint_{\pz Q_n}t^{-\frac n2}e^{-\frac{|X-Y|^2}{8t}}|a_n|^{-2}d\sz(Y)dt\\
&\le C|x|\dint_0^{|xa_n|}t^{-\frac 12}|a_n|^{-2}dt\le C.
\end{align*}

Subcase 2: $2|a_n|\le |x|$. Applying (3.2) and (3.3), it is easy to show that
\begin{align*}\tag{3.7}
\l|\dfrac{\pz \Gamma(X,t;Y,\tau)}{\pz v_n(Y)}\r|&\le C(t-\tau)^{-\frac{n+a}2}\l(\dfrac{|y|}{t-\tau}\r)e^{-\frac{|X-Y|^2}{4(t-\tau)}}\l(1+\dfrac{y^2}{t-\tau}\r)^{\frac a2}\\
&\quad+C(t-\tau)^{-\frac{n+a}2}\l(\dfrac{|x|}{t-\tau}\r)e^{-\frac{|X-Y|^2}{4(t-\tau)}}|\widehat{H}|,
\end{align*}
where
$$|\widehat{H}|\le C\l(\frac{|xy|}{t-\tau}\r)^{-a},\quad {\rm if}\quad  \frac{|xy|}{t-\tau}\le 1, $$
and
$$|\widehat{H}|\le C\l(1+\frac{|xy|}{t-\tau}\r)^{-\frac a2},\quad {\rm if}\quad  \frac{|xy|}{t-\tau}\ge 1. $$

By (3.7), we have
$$\begin{array}{cl}
&\dint_0^t\dint_{\pz Q_n}\l|\frac{\pz \Gamma(X,t;Y,\tau)}{\pz v_n(Y)}\vz(Y,\tau)\r||y|^ad\sz(Y)d\tau\\
&\le C\dint_0^t\dint_{\pz Q_n}\l|\dfrac{\pz \Gamma(X,t;Y,\tau)}{\pz v_n(Y)}\r||y|^{a}d\sz(Y)d\tau\\
&\le C\dint_0^t\dint_{\pz Q_n}(t-\tau)^{-\frac{n+a}2}\l(\dfrac{|y|}{t-\tau}\r)e^{-\frac{|X-Y|^2}{4(t-\tau)}}\l(1+\dfrac{y^2}{t-\tau}\r)^{\frac a2}|y|^ad\sz(Y)d\tau\\
&\quad+C\dint_0^t\dint_{\pz Q_n}(t-\tau)^{-\frac{n+a}2}\l(\dfrac{|x|}{t-\tau}\r)e^{-\frac{|X-Y|^2}{4(t-\tau)}}|\widehat{H}||y|^ad\sz(Y)d\tau\\
&:=J_{31}+J_{32}.
\end{array}$$
 For $J_{31}$,  note that
$2|a_n|\le |x|$, so $|x-a_n|\ge |x|/2$,  we then get
\begin{align*}
J_{31}&\le C|a_n|\dint_0^T\dint_{\pz Q_n}{t^{-\frac {n+2}2}}e^{-\frac{|X-Y|^2}{6t}}\l(\frac{a_n^2}t\r)^{\frac a2}\l(1+\frac{a_n^2}t\r)^{-\frac a2}d\sz(Y)dt\\
&\le C|a_n|\dint_0^{a_n^2}\dint_{\pz Q_n}{t^{-\frac {n+2}2}}e^{-\frac{|X-Y|^2}{6t}}d\sz(Y)dt\\
&\quad+C|a_n|\dint_{a_n^2}^T\dint_{\pz Q_n}{t^{-\frac {n+2}2}}e^{-\frac{|X-Y|^2}{6t}}\l(\frac{a_n^2}t\r)^{\frac a2}d\sz(Y)dt\\
&\le C|a_n||x|^{-2}\dint_0^{a_n^2}t^{-\frac 12}dt+C|a_n|^{1+a}\dint_{a_n^2}^Tt^{-\frac{3+a}2}dt\le C.
\end{align*}
Now estimate $J_{32}$.
From (3.7),  we have
\begin{align*}
J_{32}&\le C\dint_{|xa_n|}^T\dint_{\pz Q_n}t^{-\frac n2}e^{-\frac{|X-Y|^2}{6t}}\frac{|x|}t\l(\frac{a_n^2}t\r)^{\frac a2}\l(\frac{|xa_n|}{t}\r)^{-a}d\sz(Y)dt\\
&\quad+C\dint_0^{|xa_n|}\dint_{\pz Q_n}t^{-\frac n2}e^{-\frac{|X-Y|^2}{6t}}\frac{|x|}t\l(\frac{a_n^2}t\r)^{\frac a2}\l(1+\frac{|xa_n|}t\r)^{-\frac a2}d\sz(Y)dt\\
&:=J^1_{32}+J^2_{32}.
\end{align*}
For $J^1_{32}$, note that $|x-a_n|\ge |x|/2$, we have
\begin{align*}
J^1_{32}&\le C\dint_{0}^T\dint_{\pz Q_n}t^{-\frac n2}e^{-\frac{|X-Y|^2}{6t}}\frac{|x|}t\l(\frac{a_n^2}t\r)^{\frac a2}\l(\frac{|xa_n|}t\r)^{-a}d\sz(Y)dt\\
&\le C\dint_0^{|x|^2}\dint_{\pz Q_n}t^{-\frac n2}e^{-\frac{|X-Y|^2}{6t}}\frac{|x|}t\l(\frac{a_n^2}t\r)^{\frac a2}\l(\frac{|xa_n|}t\r)^{-a}d\sz(Y)dt\\
&\quad + C\dint_{|x|^2}^T\dint_{\pz Q_n}t^{-\frac n2}e^{-\frac{|X-Y|^2}{6t}}\frac{|x|}t\l(\frac{a_n^2}t\r)^{\frac a2}\l(\frac{|xa_n|}t\r)^{-a}d\sz(Y)dt\\
&\le C|x|^{1-a}\l(|x|^{-2}\dint_0^{|x|^2}t^{-\frac {1-a}2}dt+\dint_{|x|^2}^Tt^{-1+\frac{a-1}2}dt\r)\le  C.
\end{align*}
For $J^2_{32}$, by (3.7), note that $2|a_n|\le |x|$ and  $|x-a_n|\ge |x|/2$, we then get
\begin{align*}
J^2_{32}&\le C|x|^{1-\frac a2}|a_n|^{\frac a2}\dint_0^{|xa_n|}\dint_{\pz Q_n}t^{-\frac {n}2}e^{-\frac{|X-Y|^2}{8t}}|x|^{-2}d\sz(Y)dt\\
&\le C|x|^{-1-\frac a2}|a_n|^{\frac a2}\dint_0^{|xa_n|}t^{-\frac 12}dt\le C.
\end{align*}

Case 2: $y=0$. Note that if $y=0$, then by (3.6), we have
\begin{align*}
u(X,t)&=\dint_0^t\dint_{\pz Q_n}\dlim_{y\to 0}\l[\frac{\pz \Gamma(X,t;Y,\tau)}{\pz v_n(Y)}\vz(Y,\tau)|y|^a\r]d\sz(Y)d\tau\\
&=C_{n,a}(1-a)\dint_0^t\dint_{\pz Q_n}s^{-\frac {n+a}2}\frac xs\l(\frac {|x|}s\r)^{-a}e^{-\frac{|x'-y'|^2}{4s}}e^{-\frac{x^2}{4s}}\vz(y',0,t-s)dy'ds.
\end{align*}
Therefore,
\begin{align*}
|u(X,t)|&\le C\dint_0^T\dint_{\pz Q_n}t^{-\frac {n+a}2}\l(\frac {|x|}t\r)^{1-a}e^{-\frac{|x'-y'|^2}{4t}}e^{-\frac{x^2}{4t}}dy'dt\\
&\le C|x|^{1-a}\dint_{x^2}^T t^{-1+\frac {a-1}2}dt+C|x|^{-1-a}\dint_0^{x^2}t^{\frac{a-1}2}dt\le C.
\end{align*}
\end{proof}

\begin{lemma}
The double-layer potential $u$ with $\vz\in C{(\pz Q\times [0,T])}$ is continuous on $S(0,T]\setminus (\pz Q\times[0,T])$.

\end{lemma}

\begin{proof} For any fixed $(X,t)\in S(0,T]\setminus (\pz Q\times[0,T])$, we can assume ${\rm dist}\{(X,t); \pz Q\times[0,T]\}>\dz>0$.
 Let $Q_i$ be the same as Lemma 3.2 for $i=1,2,\cdots,n$.

  By (3.2), we have for $(Y,\tau)\in \pz Q\times[0,T]$ and  $i=1,2,\cdots,n-1$,
\begin{align*}\l|\dfrac{\pz \Gamma(X,t;Y,\tau)}{\pz v_i(Y)}\r||y|^a&\le C(t-\tau)^{-\frac{n+a}2}|y|^a\l(\dfrac{|X-Y|}{t-\tau}\r)e^{-\frac{|X-Y|^2}{4(t-\tau)}}F\l(\dfrac{xy}{t-\tau}\r)\\
&\le C(\dz^{-\frac{n+1}2}+\dz^{-\frac{n+1+a}2}|y|^a).\end{align*}

 Then for  $i=1,2,\cdots,n-1$,
$$\dint_0^t\dint_{\pz Q_i}\dz^{-\frac{n+1}2}+\dz^{-\frac{n+1+a}2}|y|^ad\sz(Y)d\tau\le C(\dz^{-\frac{n+1}2}++\dz^{-\frac{n+1+a}2}).$$
Similarly, by (3.3), we have for $(Y,\tau)\in \pz Q\times[0,T]$ and  $y=a_n\not=0$,
\begin{align*}
\l|\dfrac{\pz \Gamma(X,t;Y,\tau)}{\pz v_n(Y)}\r||y|^a&\le C(t-\tau)^{-\frac{n+a}2}|y|^a\l(\dfrac{|X-Y|}{t-\tau}\r)e^{-\frac{|X-Y|^2}{4(t-\tau)}}F\l(\dfrac{xy}{t-\tau}\r)\\
&\quad +C(t-\tau)^{-\frac{n+a}2}\dfrac {|x|}{(t-\tau)}|y|^ae^{-\frac{|X-Y|^2}{4(t-\tau)}}|H|\\
&\le C(\dz^{-\frac{n+1}2}+\dz^{-\frac{n+1+a}2}|a_n|^a)+C(\dz^{-\frac{n+2+a}2}+\dz^{-\frac{n+2-a}2})\\
&\quad+ C|a_n|^a(\dz^{-\frac{n+1+a}2}+\dz^{-\frac{n+1}2}):=C_\dz,
\end{align*}
 and for $y=0$  by (3.6), we get
$$\l|\dfrac{\pz \Gamma(X,t;Y,\tau)}{\pz v_n(Y)}\r||y|^a=\l|\dlim_{y\to 0}\l[\frac{\pz \Gamma(X,t;Y,\tau)}{\pz v_n(Y)}
|y|^a\r]\r|\le C\dz^{-\frac{n+1}2}.$$
Then

$$\dint_0^t\dint_{\pz Q_n}C_\dz d\sz(Y)d\tau\le CC_\dz,$$
and
$$\dint_0^t\dint_{\pz Q_n}\dz^{-\frac{n+2-a}2}d\sz(Y)d\tau\le C\dz^{-\frac{n+2-a}2}.$$
It should be pointed out that here all $C$ depend only on $n,a,T,\text{diam}(Q)$. From these, and Lebesgue's Dominated Convergence Theorem show
that $u$ is continuous at $(X, t)$.
\end{proof}

\begin{remark}From Lemma  $3.3$, we know that the double-layer potential  is continuous in $Q\times(0,T]$. In addition it can be continuously extended into $Q\times[0,T]$ with
initial values $u(\cdot,0)=0$ in $Q$. Furthermore, for points $(X,t)\not\in \pz Q\times[0,T]$ and $x\not=0$ we can interchange differentiation and integration,
and from Remark $2.1$, we know that the double-layer potential is a weak solution to $\mathscr{L}_a$.
\end{remark}

Before  considering the boundary continuity of  the double-layer potential, we need the following lemma.
\begin{lemma}
Let ${\cal P}$  denote  the set of corners of $Q\times[0,T]$, then
\begin{align*}
u_0(X,t)&=\dint_{-\fz}^t\dint_{\pz Q}\frac{\pz \Gamma(X,t;Y,\tau)}{\pz v(Y)}|y|^ad\sz(Y)d\tau\\
&=
\left\{ \begin{aligned}
&-1,\quad\ (X,t)\in Q\times[0,T];\\
&-\frac 12,\quad (X,t)\in \pz Q\times[0,T]\setminus {\cal P};\\
&-\frac 14,\quad (X,t)\in {\cal P};\\
& 0,\qquad\ (X,t)\in\rz\setminus  \overline{Q}\times[0,T].
\end{aligned}\right.
\end{align*}

\end{lemma}
\begin{proof}
From \cite{WO} and \cite{K}, we know that
\begin{align*}
v_0(X,t)&=\dint_{-\fz}^t\dint_{\pz Q}\frac{\pz H(X,t;Y,\tau)}{\pz v(Y)}d\sz(Y)d\tau\\
&=-\dlim_{\ez\to 0^+}\dint_Q H(X,t;Y,t-\ez)dY\\
&=
\left\{ \begin{aligned}
&-1,\quad\ (X,t)\in Q\times[0,T];\\
&-\frac 12,\quad (X,t)\in \pz Q\times[0,T]\setminus {\cal P};\\
&-\frac 14,\quad (X,t)\in {\cal P};\\
& 0,\qquad\ (X,t)\in\rz\setminus  \overline{Q}\times[0,T],
\end{aligned}\right.
\end{align*}
where $H(X,t;Y,\tau)$ is a classical heat kernel, that is,
$$H(X,t;Y;\tau)=(4\pi)^{-\frac n2}e^{-\frac {|X-Y|^2}{4(t-\tau)}}$$
for $t>\tau$, otherwise it is $0$.

In addition, it is easy to see that
$$u_0(X,t)=-\dlim_{\ez\to 0^+}\dint_Q \Gamma(X,t;Y,t-\ez)|y|^adY.$$
From this and  (ii) in Theorem 2.1, we get that
$$
u_0(X,t)=
\left\{ \begin{aligned}
&-1,\quad (X,t)\in Q\times[0,T];\\
&0,\qquad\ (X,t)\in\rz\setminus  \overline{Q}\times[0,T].
\end{aligned}\right.
$$
It remains to  consider the case of the boundary $\pz Q\times[0,T]$. Note that if
$(X,t)\in \pz Q\times[0,T]$, then
$$\begin{array}{cl}
\dlim_{\ez\to 0^+}\dint_QH(X,t;Y,t-\ez)dY&=
\left\{ \begin{aligned}

&-\frac 12,\quad (X,t)\in \pz Q\times[0,T]\setminus {\cal P};\\
&-\frac 14,\quad (X,t)\in {\cal P}.
\end{aligned}\right.
\end{array}\eqno(3.8)$$

We next consider two cases about $X=(x',x)$.

Case 1. $x\not=0$. For $0<\dz<1/2$, we have for $\forall\ \ez>0$
\begin{align*}
&\dint_Q||y|^a\Gamma(X,t;Y,t-\ez)-H(X,t;Y,t-\ez)|dY\\
&\le \dint_{\{|y-x|<\dz |x|\}\cap Q}||y|^aG(X,t;Y,t-\ez)-H(X,t;Y,t-\ez)|dY\\
&\quad+\dint_{\{|y-x|\ge\dz |x|\}\cap Q}||y|^aG(X,t;Y,t-\ez)-H(X,t;Y,t-\ez)|dY\\
&:=I_1+I_2.
\end{align*}
For $I_1$, since $|y-x|<\dz |x|$, so $xy\ge |x|^2/2>0$. Notice that
$$|y|^a\Gamma(X,t;Y,t-\ez)=\l(\frac yx\r)^{\frac a2}H(X,t;Y,t-\ez)\times\l(1+O\l(\frac{\ez}{xy}\r)\r),\quad {\rm if}\ \ez\to 0^+.\eqno(3.9)$$
In fact, from \cite {WO, G}, we know that
$$\Gamma(X,t;Y,t-\ez)=(4\pi)^{-\frac {n}2}\ez^{-\frac{n+a}2}e^{-\frac {|X-Y|^2}{4\ez}}\l(\frac {xy}{\ez}\r)^{-\frac a2}\l(1+O\l(\frac{\ez}{xy}\r)\r), \ \ez\to 0^+,$$
so
$$|y|^a\Gamma(X,t;Y,t-\ez)=\l(\frac yx\r)^{\frac a2}(4\pi)^{-\frac {n}2}\ez^{-\frac{n}2}e^{-\frac {|X-Y|^2}{4\ez}}\l(1+O\l(\frac{\ez}{xy}\r)\r), \ \ez\to 0^+.$$
By (3.9), if $\dz$ is small enough, we then get that
\begin{align*}
I_1&\le \dint_{\{|y-x|<\dz |x|\}}\l|\l(\frac yx\r)^{\frac a2}-1\r|H(X,t;Y,t-\ez)dY\\
&\quad +C\dint_{\{|y-x|<\dz |x|\}}\l(\frac yx\r)^{\frac a2}\frac{\ez}{xy}H(X,t;Y,t-\ez)dY\\
&\le C\l(\dz+\frac{\ez}{x^2}\r),
\end{align*}
where $C$ depends only on $a$.

For $I_2$, we have
\begin{align*}
I_2&\le \dint_{\{|y-x|\ge\dz |x|\}\cap Q}|y|^aG(X,t;Y,t-\ez)dY+\dint_{\{|y-x|\ge\dz |x|\}\cap Q}H(X,t;Y,t-\ez)dY\\
&:=I_{21}+I_{22}.
\end{align*}
For $I_{21}$, we get
\begin{align*}
I_{21}&\le C\dint_{\{|y-x|\ge\dz |x|\}\cap Q}\ez^{-\frac{n+a}2}e^{-\frac {|X-Y|^2}{6\ez}}\l(1+\frac{y^2}\ez\r)^{-\frac a2}|y|^adY\\
&\le C\dint_{\{|y-x|\ge\dz |x|,y^2<\ez\}\cap Q}\ez^{-\frac{n+a}2}e^{-\frac {|X-Y|^2}{6\ez}}|y|^adY\\
&\quad+ C\dint_{\{|y-x|\ge\dz |x|,y^2\ge \ez\}\cap Q}\ez^{-\frac{n}2}e^{-\frac {|X-Y|^2}{6\ez}}dY\\
&\le Ce^{-\frac {(\dz x)^2}{12\ez}}\dint_{\{y^2<\ez\}}\ez^{-\frac{n+a}2}e^{-\frac {|X-Y|^2}{12\ez}}|y|^adY\\
&\quad+ Ce^{-\frac {(\dz x)^2}{12\ez}}\dint_{\rn}\ez^{-\frac{n}2}e^{-\frac {|X-Y|^2}{12\ez}}dY\le Ce^{-\frac {(\dz x)^2}{12\ez}}.
\end{align*}
Similarly,
$$I_{22}\le Ce^{-\frac {(\dz x)^2}{12\ez}}.$$
From these, taking $\dz=\ez^{\frac 14}$, then
$$I_1+I_2\le C\l(\dz+\frac{\ez}{x^2}+e^{-\frac {(\dz x)^2}{12\ez}}\r)=C\l(\ez^{\frac 14}+\frac{\ez}{x^2}+e^{-\frac { x^2}{12\sqrt{\ez}}}\r).$$
Hence,
$$I_1+I_2\to0,\qquad \ez\to 0^+.$$
From these and (3.8), we get
\begin{align*}
\dlim_{\ez\to 0^+}\dint_Q |y|^a\Gamma(X,t;Y,t-\ez)dY&=
\dlim_{\ez\to 0^+}\dint_Q H(X,t;Y,t-\ez)dY\\
&=\left\{ \begin{aligned}
&-\frac 12,\quad (X,t)\in \pz Q\times[0,T]\setminus {\cal P};\\
&-\frac 14,\quad (X,t)\in {\cal P}.
\end{aligned}\right.
\end{align*}

Case 2: $x=0$, that is, $X=(x',0)$. We have
\begin{align*}
u_0(X,t)&=\dint_{-\fz}^t\dint_{\pz Q}\frac{\pz \Gamma(X,t;Y,\tau)}{\pz v(Y)}|y|^ad\sz(Y)d\tau\\
&=\dfrac {C_{n,a}}{2\Gamma {(\frac {1+a}2)}}\dint_{-\fz}^t\dint_{\pz Q}\dfrac{-|X-Y|}{(t-\tau)^{\frac{n+a}2+1}}e^{-\frac {|X-Y|^2}{4(t-\tau)}}<\frac{Y-X}{|Y-X|},v(Y)>|y|^ad\sz(Y)d\tau\\
&=-\dfrac {C_{n,a}}{2\Gamma {(\frac {1+a}2)}}\dint_{-\fz}^t\dint_{\pz Q}\dfrac{|X-Y|^{n+a}}{(t-\tau)^{\frac{n+a}2+1}}e^{-\frac {|X-Y|^2}{4(t-\tau)}}dw(Y)d\tau,
\end{align*}
where
$$dw(Y)=<\frac{Y-X}{|Y-X|},v(Y)>\frac {|y|^ad\sz(Y)}{|X-Y|^{n-1+a}}.$$
Let $\eta=\frac{|X-Y|}{2\sqrt{t-\tau}}$ and $d\eta=\frac{|X-Y|}{4(t-\tau)^{\frac32}}$, we obtain
\begin{align*}
&\dint_{-\fz}^t\dfrac{|X-Y|^{n+a}}{(t-\tau)^{\frac{n+a}2+1}}e^{-\frac {|X-Y|^2}{4(t-\tau)}}d\tau\\
&=2^{n+1+a}\dint_{-\fz}^t\l(\frac{|X-Y|}{2\sqrt{t-\tau}}\r)^{n-1+a}e^{-(\frac{|X-Y|}{2\sqrt{t-\tau}})^2}\frac{|X-Y|}{4(t-\tau)^{\frac32}}d\tau\\
&=2^{n+1+a}\dint_0^\fz \eta^{n-1+a}e^{-\eta^2}d\eta\\
&=2^{n+a}\Gamma(\frac{n+a}2).
\end{align*}
Notice that
\begin{align*}
&\dint_{\{
|X-Y|>r\}\cap \pz Q}<\frac{Y-X}{|Y-X|},v(Y)>\frac {|y|^ad\sz(Y)}{|X-Y|^{n-1+a}}\\
&\quad+\dint_{\{
|X-Y|=r\}\cap  Q}<\frac{Y-X}{|Y-X|},v(Y)>\frac {|y|^ad\sz(Y)}{|X-Y|^{n-1+a}}\\
&=\dlim_{\ez\to 0}\dint_{\{|y|>\ez,
|X-Y|>r\}\cap \pz Q}<\frac{Y-X}{|Y-X|},v(Y)>\frac {|y|^ad\sz(Y)}{|X-Y|^{n-1+a}}\\
&\quad+\dlim_{\ez\to 0}\dint_{\{|y|>\ez,
|X-Y|=r\}\cap  Q}<\frac{Y-X}{|Y-X|},v(Y)>\frac {|y|^ad\sz(Y)}{|X-Y|^{n-1+a}}\\
&=c_a\dlim_{\ez\to 0+}\dint_{\{|y|>\ez,
|X-Y|>r\}\cap  D}\text{div}(|y|^a\nabla_Y|X-Y|^{-(n-2+a)})dY\\
&\quad+\dlim_{\ez\to 0}\dint_{\{y=\ez,
|X-Y|\ge r\}\cap  Q}|\ez|^a\nabla_y|X-Y|^{-(n-2+a)}dy'\\
&\quad-\dlim_{\ez\to 0}\dint_{\{y=-\ez,
|X-Y|\ge r\}\cap  Q}|\ez|^a\nabla_y|X-Y|^{-(n-2+a)}dy'\\
&=0,
\end{align*}
where $c_a=-1/(n-2+a)$, in the last equality, we use the fact that
$$\dlim_{y\to 0}|y|^a\nabla_y|X-Y|^{-(n-2+a)}=0,$$ and
$u=|X-Y|^{-(n-2+a)}$ is a solution of $\text{div}(|y|^a\nabla_Y u)=0$ if $y\not= 0$; see \cite{CS}.

Let $\wt C_{n,a}=\frac {C_{n,a}}{2\Gamma {(\frac {1+a}2)}}2^{n+a}\Gamma(\frac{n+a}2)=\pi^{-\frac {n-1}2}\Gamma(\frac{n+a}2)/2\Gamma {(\frac {1+a}2)}$. It follows that
\begin{align*}
u_0(X,t)&=-\wt C_{n,a}\dint_{\pz Q}<\frac{Y-X}{|Y-X|},v(Y)>\frac {|y|^ad\sz(Y)}{|X-Y|^{n-1+a}}\\
&=-\wt C_{n,a}\dlim_{r\to 0}
\dint_{\{
|X-Y|>r\}\cap \pz Q}<\frac{Y-X}{|Y-X|},v(Y)>\frac {|y|^ad\sz(Y)}{|X-Y|^{n-1+a}}\\
&=\wt C_{n,a}\dlim_{r\to 0}\dint_{\{|X-Y|=r\}\cap  Q}<\frac{Y-X}{|Y-X|},v(Y)>\frac {|y|^ad\sz(Y)}{|X-Y|^{n-1+a}}\\
&=-\wt C_{n,a}\dlim_{r\to 0}\dfrac 1{r^{n-1+a}}\dint_{\{|X-Y|=r\}\cap Q}|y|^ad\sz(Y).
\end{align*}

Note that
\begin{align*}
\dint_{\{|X-Y|=r\}}|y|^ad\sz(Y)&=2\dint_{|x'-y'|\le r}(r^2-|x'-y'|^2)^{\frac {a-1}2}dy'\\
&=2r^{n-1+a}\dfrac{2\pi^{\frac {n-1}2}}{\Gamma (\frac {n-1}2)}\dint_0^1(1-s^2)^{\frac {a-1}2}s^{n-2}ds\\
&=2r^{n-1+a}\pi^{\frac {n-1}2}\dfrac{\Gamma {(\frac {1+a}2)}}{\Gamma(\frac{n+a}2)}.
\end{align*}
So,
\begin{align*}
u_0(X,t)
&=\left\{ \begin{aligned}
&-\frac 12,\quad (X,t)\in \pz Q\times[0,T]\setminus {\cal P};\\
&-\frac 14,\quad (X,t)\in {\cal P}.
\end{aligned}\right.
\end{align*}
\end{proof}

\begin{lemma}
 If $\vz\in C{(\pz Q\times[0,T])}$, then
 $$u_{\pm}(X_0,t_0)=\dint_{0}^t\dint_{\pz Q}\frac{\pz \Gamma(X_0,t_0;Y\tau)}{\pz v(Y)}\vz(Y,\tau)|y|^ad\sz(Y)d\tau-\frac 12\dz_{\pm}\vz(X_0,t_0),\eqno(3.10)$$
 for $X_0\in\pz Q,\ t_0\in(0,T]$, where $\dz_{+}=\frac 12$ and $\dz_{-}=\frac 32$ if $(X_0,t_0)$ is a corner point, otherwise $\dz_{\pm}=\pm 1$. Here  the integral exits as in
 proper integral
 $$u_{+}(X_0,t_0):=\dlim_{(X,t)\to (X_0,t_0)}u(X,t),\quad {\rm if }\ (X,t)\in \rr^{n+1}\setminus \overline{ Q}\times [0,T];$$ and
 $$u_{-}(X_0,t_0):=\dlim_{(X,t)\to (X_0,t_0)}u(X,t),\quad {\rm if}\  (X,t)\in  Q\times (0,T].$$
\end{lemma}
\begin{proof}
Let $u_0$ be as in Lemma 3.4, $\vz\in C{(\pz Q\times[0,T])}$ and $(X_0,t_0)\in \pz Q\times(0,T]$. Then for any $(X,t)\in S(0,T]$, we have
\begin{align*}
u(X,t)-\vz(X_0,t_0)u_0(X,t)&=\dint_{0}^t\dint_{\pz Q}\frac{\pz \Gamma(X,t;Y,\tau)}{\pz v(Y)}\l[\vz(Y,\tau)-\vz(X_0,t_0)\r]|y|^ad\sz(Y)d\tau\\
&\quad -\vz(X_0,t_0)\dint_{-\fz}^0\dint_{\pz Q}\frac{\pz \Gamma(X,t;Y,\tau)}{\pz v(Y)}|y|^ad\sz(Y)d\tau\\
&=\dint_{0}^t\dint_{\pz Q}\frac{\pz \Gamma(X,t;Y,\tau)}{\pz v(Y)}\l[\vz(Y,\tau)-\vz(X_0,t_0)\r]|y|^ad\sz(Y)d\tau\\
&\quad -\vz(X_0,t_0)\dint_{ Q} \Gamma(X,t;Y,0)|y|^adY\\
&:=V_1(X,t)+V_2(X,t).
\end{align*}
We will show that the function $u(X,t)-\vz(X_0,t_0)u_0(X,t)$ is defined and finite on $S(0,T]$, and is continuous at $(X_0,t_0)$.

In fact, if $u(X,t)-\vz(X_0,t_0)u_0(X,t)$  is continuous at $(X_0,t_0)$, then by Lemma 3.4, we immediately get (3.10). In addition, if we write
 $$\widetilde{u}(X,t)=u(X,t)- \dz(X,t)\vz(X,t),\eqno(3.11)$$
 where $\dz(X,t)=\frac 14$ if $(X,t)$ is a corner point, otherwise $\dz(X,t)=0$, then it is easy to see that the restriction of $\widetilde{u}(X,t)$ to $\pz Q\times(0, T]$
 is continuous. Furthermore, due to the absolute of the integral (3.1) as a set function, $u(X,t)\to 0$ as $t\to 0^+$.
 So the restriction of $u$ to $L$ has a continuous extension by $0$ to
 $\bar L$.

 It remains to prove $u(X,t)-\vz(X_0,t_0)u_0(X,t)$  is continuous at $(X_0,t_0)$.

We begin with $V_2$. Notice that
\begin{align*}
\dint_{ Q} \Gamma(X,t;Y,0)|y|^adY&\le C\dint_Q t^{-\frac {n+a}2}e^{-\frac{|X-Y|^2}{6t}}\l(1+\frac {y^2}{t}\r)^{-\frac a2}|y|^adY\\
&= C\dint_{\{|y|^2\le {t}\}\cap Q} {t}^{-\frac {n+a}2}e^{-\frac{|X-Y|^2}{6{t}}}|y|^adY\\
&\quad+ C\dint_{\{|y|^2> {t}\}\cap Q} {t}^{-\frac {n}2}e^{-\frac{|X-Y|^2}{6{t}}}dY\\
&\le C.
\end{align*}
Therefore $V_2$ is finite on that any neighbourhood in $S(0,T]$.  In addition, for any $(Y_0,s_0)\in S(0,T]$, then whenever $|t-s_0|<s_0/2$ we have $t>s_0/2$, so
$$ \Gamma(X,t;Y,0)\le Ct^{-\frac {n+a}2}e^{-\frac{|X-Y|^2}{6t}}\l(1+\frac {y^2}{t}\r)^{-\frac a2}\le Cs_0^{-\frac {n+a}2}\l(1+\frac {y^2}{s_0}\r)^{-\frac a2}.$$
Thus,
\begin{align*}
\dint_{ Q} s_0^{-\frac {n+a}2}\l(1+\frac {y^2}{s_0}\r)^{-\frac a2}|y|^adY&=\dint_{\{|y|^2\le {s_0}\}\cap Q} s_0^{-\frac {n+a}2}\l(1+\frac {y^2}{s_0}\r)^{-\frac a2}|y|^adY\\
&\quad+ \dint_{\{|y|^2> {s_0}\}\cap Q} s_0^{-\frac {n+a}2}\l(1+\frac {y^2}{s_0}\r)^{-\frac a2}|y|^adY\\
&\le  C\dint_{\{|y|^2\le {s_0}\}\cap Q} s_0^{-\frac {n+a}2}|y|^adY+ C\dint_{\{|y|^2> {s_0}\}\cap Q} {s_0}^{-\frac {n}2}dY\\
&\le C({s_0}^{-\frac {n}2}+{s_0}^{-\frac {n-1}2}).
\end{align*}
From these, and the Lebesgue's Dominated Convergence Theorem show
that $V_2$ is continuous at $(Y_0, s_0)$.

We now consider $V_1$. Since $\vz$ is continuous on $L$,
given $\ez>0$, we can find a relative neighbourhood $N$ of $(X_0, t_0)$ in $L$ such that for $(Y,\tau),\ (X,t)\in { N}$
$$|\vz(Y,\tau)-\vz(X,t)|<\ez.\eqno(3.12)$$
By Lemma 3.2, we have
$$
V_1(X,t)\le C\ez\dint\dint_{N}\l|\frac{\pz \Gamma(X,t;Y,\tau)}{\pz v(Y)}\r||y|^ad\sz(Y)d\tau\le C\ez.\eqno(3.13)$$
For brevity, we write $\xi=(X,t)$, $\zeta=(Y,\tau)$, $\xi_0=(X_0,t_0)$ and $\sz(\zeta)=d\sz(Y)d\tau$,  $\phi(\xi,\zeta)$ is defined by
$$\phi(\xi,\zeta):=\frac{\pz \Gamma(X_0,t_0;Y,\tau)}{\pz v(Y)}\l[\vz(Y,\tau)-\vz(X,t)\r]|y|^a\eqno(3.14)$$
if $\tau<t$, otherwise it is $0$.

 From  the proof of Lemma 3.3, we know that the integral
$$\Phi(\xi)=\dint_{L\setminus N}\phi(\xi,\zeta)d\sz(\zeta)$$
is continuous at $\xi_0$. So, we can find $\gz>0$ such that
${\cal B}(\xi_0, \gz)\cap L \subset N$(${\cal B}(\xi_0, \gz)=\{\xi\in\rr^{n+1}: |\xi-\xi_0|<\gz \}$) and
$$|\Phi(\xi)-\Phi(\xi_0)|\le \ez.\eqno(3.15)$$
By (3.12)-(3.15), if $|\xi-\xi_0|<\gz$, we then obtain
\begin{align*}
|V_1(\xi)-V_1(\xi_0)|&\le \l|\dint_L\phi(\xi,\zeta)d\sz(\zeta)-\dint_L\phi(\xi_0,\zeta)d\sz(\zeta)\r|\\
&\le \l|\dint_N\phi(\xi,\zeta)d\sz(\zeta)\r|+\l|\dint_N\phi(\xi_0,\zeta)d\sz(\zeta)\r|\\
&\quad+\l|\dint_{L\setminus N}[\phi(\xi,\zeta)-\phi(\xi_0,\zeta)]d\sz(\zeta)\r|\\
&\le 3\ez.
\end{align*}
Thus,  $V_1$ is continuous at $(X_0,t_0)$.
\end{proof}
We are now in a position to prove Theorem 3.1.

\begin{proof}
It suffices to prove that there is a function $U\in C( \overline{Q_T})$ such that
$$ \dlim_{(X,t)\to (Y,s)}u(X,t)=f(Y,s),\quad (Y,s)\in\pz_p Q_T.  $$
For if the weaker statement has been proved, we can extend $Q_T=Q\times(0,T)$ to $\oz^*=Q\times(0,T^*)$ with $T^*>T$,
and $f$ and $f^*\in C(\pz Q\times(0,T^*)\cup I)$ and apply that result to $Q_T^*$ and $f^*$.

Let $g$ be a continuous function on $\pz_p Q_T$ such that $g=0$ on $I$. We seek a $\mathscr{L}_a$-parabolic $u=u_g$ on $Q_T$ such that
$$ \dlim_{(X,t)\to (Y,s)}u_g(X,t)=g(Y,s),\quad (Y,s)\in\pz_p Q_T \eqno(3.16) $$
in the form of a double-layer potential $u$ in (3.1) . From Lemma 3.3 and Remark 3.2, we know that $u$ is a $\mathscr{L}_a$-parabolic on $Q_T$.

In order to satisfy (3.16) at a point $(X_0,t_0)\in L$,   Lemma 3.5  shows that
$$g(X_0,t_0)=u(X_0,t_0)-\frac 12\vz(X_0,t_0)-\dz(X_0,t_0)\vz(X_0,t_0), \eqno(3.17)$$
 where $\dz(X_0,t_0)=\frac 14$ if $(X_0,t_0)$ is a corner point, otherwise $\dz(X_0,t_0)=0$.

In addition,  (3.17) can be turned into that
 $$\vz(X_0,t_0)=2\dint_0^t\dint_{\pz Q}\frac{\pz \Gamma(X_0,t_0;Y,\tau)}{\pz v(Y)}\vz(Y,\tau)|y|^ad\sz(Y)d\tau-2\dz(X_0,t_0)\vz(X_0,t_0)- 2g(X_0,t_0).$$
 As \cite{A}, we will show that this integral equation for $\vz$ has a solution using the contraction mapping principle.

 For any $\vz\in C( \pz Q\times[0, T])$ define
 $$({\cal F}\vz)(X,t)=2\dint_0^t\dint_{\pz Q}\frac{\pz \Gamma(X,t;Y,\tau)}{\pz v(Y)}\vz(Y,\tau)|y|^ad\sz(Y)d\tau-2\dz(X,t)\vz(X,t)- 2g(X,t)$$
 for any $(X,t)\in \pz Q\times[0, T]$.

Since the restriction of $u(X,t)-\dz(X,t)\vz(X,t)$ to $\pz Q\times[0, T]$
 is continuous, so (3.17) implies that $ {\cal F}$ maps $C( \pz Q\times[0, T])$ into itself. The space $C( \pz Q\times[0, T])$ with the weighted sup-norm
 $$\|\vz\|=\dsup_{(X,t)\in \pz Q\times[0, T]}|\vz(X,t)|e^{-4l(X,t)},$$  where
 $$l(X,s)=\dint_0^s\dint_{\pz Q}2\l|\frac{\pz \Gamma(X,t;Y,\tau)}{\pz v(Y)}\r||y|^ad\sz(Y)d\tau,$$
 is a Banach space. Note that for each $t\in (0,T]$ the function $l(X,s)$ is differentiable in $(0,t)$ with
 $$dl(X,s)=\dint_{\pz Q}2\l|\frac{\pz \Gamma(X,t;Y,s)}{\pz v(Y)}\r||y|^ad\sz(Y)ds.$$
 For $t\in [0,T]$ and $\vz,\wt\vz\in C( \pz Q\times[0, T])$, we have
 $$\begin{array}{cl}
 &|({\cal F}\vz)(X,t)-({\cal F}\wt\vz)(X,t)|e^{-4l(X,t)}\\
 &\le  e^{-4l(X,t)}
 \dint_0^t\dint_{\pz Q}2\l|\frac{\pz \Gamma(X,t;Y,\tau)}{\pz v(Y)}\r||\vz(Y,\tau)-\wt\vz(Y,\tau)||y|^ae^{-4l(X,\tau)}e^{4l(X,\tau)}d\sz(Y)d\tau\\
 &\quad+\frac 12 \l|\vz(Y,t)-\wt\vz(Y,t)\r|e^{-4l(X,t)}\\
  &\le  e^{-4l(X,t)}\|\vz-\wt\vz\|\dint_0^t\dint_{\pz Q}2\l|\frac{\pz \Gamma(X,t;Y,\tau)}{\pz v(Y)}\r|e^{4l(X,\tau)}d\sz(Y)d\tau+\frac 12 \|\vz-\wt\vz\|\\
 &\le \frac 14 (1- e^{-4l(X,t)})\|\vz-\wt\vz\|+\frac 12 \|\vz-\wt\vz\|\le \frac 34 \|\vz-\wt\vz\|.
  \end{array}$$
 It follows that there is a unique continuous function $\vz$ on $\pz Q\times[0, T]$ such that ${\cal F}\vz=\vz$. For this function $\vz$, the double-layer potential
 $$u_g=\dint_0^t\dint_{\pz Q}\frac{\pz \Gamma(X,t;Y;\tau)}{\pz v(Y)}\vz(Y,\tau)|y|^ad\sz(Y)d\tau$$
 is a $\mathscr{L}_a$-parabolic on $Q_T$ which satisfies
 $$ \dlim_{(X,t)\to (Y,s)}u_g(X,t)=g(Y,s),\quad (Y,s)\in\pz_p Q_T, $$
 and so solves the Dirichlet problem on $Q_T$ for the function $g$.
\end{proof}

\section{Perron method and barrier functions }
\subsection{Perron process}
We first give a preliminary version of the comparison principle.
\begin{lemma}
Suppose that $u$ is a supersolution and $v$ is a subsolution to  $\mathscr{L}_a$ in $Q_T=Q\times(0,T]$. If $u$ and $-v$ are lower
semicontinuous on $ \overline{Q_T}$ and $v\le u$ on $\pz_p Q_T$, then $v\le u$ a.e. in  $Q_T$.

\end{lemma}

\begin{proof}The proof is similar to that of Lemma 3.1 in \cite{KL} or  Lemma 3.5 in \cite{KKP}, we omit the details.\end{proof}

The uniform H\"older estimate (see \cite{CF}) combined Theorem 3.1 and Lemma 4.1 lead a convergence result.

\begin{lemma}
Suppose that $u_k$ is a locally uniformly bounded sequence of $\mathscr{L}_a$-parabolic functions in  the open set $\oz$.
Then it has a subsequence that converges locally uniformly in $\oz$ to $\mathscr{L}_a$-parabolic function.
\end{lemma}
\begin{proof}Ref the proof  of Lemma 3.4 in \cite{KL}.\end{proof}

Next we recall the definition of $\mathscr{L}_a$-superparabolic (subparabolic) functions.

\begin{definition}
A function $u: \oz\to(-\fz,\fz]$ is called $\mathscr{L}_a$-superparabolic if
\begin{enumerate}
\item[$\mathrm{(i)}$] $u$ is lower semicontinuous,
\item[$\mathrm{(ii)}$] $u$ is finite in a dense subset of $\oz$,
\item[$\mathrm{(iii)}$] $u$ satisfies the comparison principle on each box $Q_{t_1;t_2}=Q\times(t_1,t_2)$ with closure in $\oz$: if $h$ is $\mathscr{L}_a$ in $Q_{t_1;t_2}$ and
continuous on $ \overline{Q}_{t_1;t_2}$ and if $h\le u$ on the parabolic boundary of $Q_{t_1;t_2}$, then $h\le u$ in the whole $Q_{t_1;t_2}$.
\end{enumerate}
\end{definition}
We remark that the definition of $\mathscr{L}_a$-superparabolic function coincides to that defined in Section 1; see Proposition 4.3 below.

Now we consider the connection between superparabolic functions and supersolutions to $\mathscr{L}_a$.

\begin{lemma}
Let $u$ be a lower semicontinuous supersolution to $\mathscr{L}_a$ in $\oz$. Then for all $(X_0,t_0)\in\oz$,
$$u(X_0,t_0)=\liminf_{(X,t)\to (X_0,t_0)} u(X,t)={\displaystyle\mathop{{\rm ess\liminf}}_{(X,t)\to (X_0,t_0)} u(X,t)}.$$
\end{lemma}
\begin{proof}Ref the proof  of Lemma 2.4 in \cite{BBP}.\end{proof}

As a corollary we have the following result.
\begin{corollary}
 Let $u$ be a lower semicontinuous supersolution to $\mathscr{L}_a$ in $\oz$. Then $u$ is a $\mathscr{L}_a$-superparabolic.
\end{corollary}
\begin{proof} It suffices to check the property (iii) of a superparabolic function. Suppose $h$ is a $\mathscr{L}_a$-parabolic function, which continuous up to the closure of $Q\times(t_1,t_2)$,
and $h\le u$ on the parabolic boundary. Since $u$ is lower semicontinuous, the lower semicontinuous relaxation $u^*$ of $u$ with respect to $Q\times(t_1,t_2)$ satisfies
$$u^*\ge u\ {\rm on}\ \pz_p(Q\times(t_1,t_2))$$
and equals $u$ in $Q\times(t_1,t_2)$. By Lemma 4.1, $h\le u$ a.e. in $Q\times(t_1,t_2)$.  So, by Lemma 4.3, $h\le u$  in $Q\times(t_1,t_2)$.
Now given any point $(X,t_2)$, by Lemma 4.3 again, we can take a sequence of points of $(X_n,t_n)\in Q\times(t_1,t_2)$ such that $(X_n,t_n)\to (X,t_2)$ and
$u(X_n,t_n)\to u(X,t_2)$. Then by the continuity of $h$, we have that $h(X,t_2)\to u(X,t_2)$. The conclusion thus follows.
\end{proof}

Now, we show that a bounded $\mathscr{L}_a$-superparabolic function is a supersolution.
\begin{proposition}
 Let $\oz$ is a domain in $\rr^{n+1}$. Suppose $u$ is a $\mathscr{L}_a$-superparabolic function and $u$ is locally bounded or $u\in L^2((t_1,t_2), H^1(Q,|x|^a))$ whenever
 $\overline{Q_{t_1,t_2}}\subset\oz$.  Then $u$ be a  supersolution to $\mathscr{L}_a$ in $\oz$.
 \end{proposition}
\begin{proof} It follows by a step-by-step imitation of the arguments of Theorem 5.8 in \cite{KKP}, we omit the details.\end{proof}

Next, we  give a elliptic  version of the comparison principle.

\begin{lemma}
Suppose that $u$ is  $\mathscr{L}_a$-superparabolic  and $v$ is  $\mathscr{L}_a$-subparabolic in a bounded open  set $\oz$. If
$$\fz\not=\limsup_{(Y,s)\to (X,t)}v(Y,\tau)\le \liminf_{(Y,s)\to (X,t)} v(Y,\tau)\not=-\fz\eqno(4.1)$$
at each point $(X,t)$ on the Euclidean boundary of $\oz$, then $v\le u$ in $\oz$.
\end{lemma}
\begin{proof}We first assume that $u$ and $v$ are bounded, so by Proposition 4.1, $u$ is  $\mathscr{L}_a$-supsolution  and $v$ is  $\mathscr{L}_a$-subsolution.

For each $\ez>0$, consider the set
$$K_\ez=\{(Y,\tau)\in\oz ~|~ v(Y,\tau)\ge u(Y,\tau)+\ez\},$$
which is a compact subset of $\oz$ by (4.1) with the bounds of $u$ and $v$. So, there is an open set $D_\ez\subset\oz$ such that $K_\ez\subset D_\ez$,
where $D_\ez$ is a union of finitely many boxes $Q_i\times(t_{i1},t_{i2})$ and $\pz D_\ez\subset \oz\setminus K_\ez$.

Because $v$ is upper semicontinuous, $u$ is lower semicontinuous, and the parabolic boundary $\Gamma_\ez$ of $D_\ez$, we find a continuous function $\vz$ on $\Gamma_\ez$
such that $u\le\vz\le u+\ez$ on $\Gamma_\ez$. From Remark 3.1, we know that there exists a $\mathscr{L}_a$-parabolic function $h$ in $D_\ez$ that coincides with
$\vz$ on $\Gamma_\ez$, then from the definition of $\mathscr{L}_a$-superparabolic and $\mathscr{L}_a$-subparabolic functions, we have
$$v\le h\le u+\ez,\quad{\rm in}\ D_\ez.$$
Hence, $v\le u+\ez$, and the lemma follows by letting $\ez\to 0$.

Now we drop the assumption that $u$ and $v$ are bounded. In fact, by compactness, (4.1) and semicontinuously, $u$ is bounded from below and $v$ is bounded from above.
Let $M=\dsup_{\oz}v, m=\dinf_\oz u, u_M=\min\{u, M\}$ and $v_m=\max\{v,m\}$. Then $u_M$ and $v_m$ satisfy a similar comparison on the boundary as in (4.1). Then from previous proof,
we know that $v_m\le u_M$, hence $v\le u$ in $\oz$.
\end{proof}
Now, we  give a parabolic  version of the comparison principle which is more nature.
\begin{lemma}
Suppose that $u$ is  $\mathscr{L}_a$-superparabolic  and $v$ is  $\mathscr{L}_a$-subparabolic in a  open set $\oz$. Let $T\in\rr$.
\begin{enumerate}
\item[$\mathrm{(i)}$]If $\oz$ is bounded,   and assume that $(4.1)$
holds for all  $(X,t)\in\oz$ with $t<T$, then $v\le u$ in $\oz_{-}=\{(X,t)\in\oz: t<T\}.$
\item[$\mathrm{(ii)}$]If $\oz$ is unbounded,   and assume that $u$ and $v$ are bounded  and the following condition that
$$\limsup_{(Y,\tau)\to (X,t)} (v-u)(Y,\tau)\le 0\eqno(4.2)$$
at each point $(X,t)$ on the bounded Euclidean boundary of $\oz$, and
$$\limsup_{(Y,\tau)\to \fz} (v-u)(Y,\tau)\le 0\eqno(4.3)$$
hold, then  $v\le u$ in $\oz_{-}=\{(X,t)\in\oz: t<T\}.$
\end{enumerate}
\end{lemma}
\begin{proof}If $\oz$ is bounded, we can prove (i) by adapting the proof of Theorem 2.10 in \cite{BBP} by using Lemma 4.4. In addition, Lemma 4.4 can be deduced by the condition (4.2) holds if $u$ and $v$ are bounded, so (i) also holds under the same condition. If $\oz$ is unbounded, set
$w=u-v$.   By (4.3), for every $\ez>0$ there exists $R_\ez>0$ such that $w(X,t)\ge -\ez $ for every $(X,t)\in\oz$,
$|(X,t)|\ge R_\ez$. We can suppose $R_\ez\to\fz$, as $\ez\to 0$. Now we apply what we have already proved to the function $w_\ez(X,t)=w(X,t)+\ez $ on the bounded open set
$\oz_\ez=\{(X,t)\in\oz:\ |(X,t)|<R_\ez\}$. We obtain $w_\ez\ge 0$ in $\oz_\ez\cap\{t<T\}$. Let $\ez$ go to zero, we obtain $w\ge 0$ in $\oz_{-}$. Thus, (ii) is also proved.
\end{proof}
By Lemma 4.5, we can obtain the following pasting lemma, which is useful when constructing new $\mathscr{L}_a$-superparabolic functions.
\begin{lemma}
Let $U\subset\oz$ be open. Also let $u$ and $v$ be  $\mathscr{L}_a$-superparabolic in $\oz$ and $U$, respectively, and let
$$\begin{array}{cl}
w=\left\{ \begin{aligned}
& \min\{u,v\}\quad \ {\rm in}\ U,\\
&u\qquad\qquad \quad {\rm in}\ \oz\setminus U.
\end{aligned}\right.
\end{array}$$
is lower semicontinuous, it is $\mathscr{L}_a$-superparabolic in $\oz$.

\end{lemma}
\begin{proof}Ref the proof of Lemma 2.12 in \cite{BBP}.
\end{proof}

The main tool in the Perron method is the modification of $\mathscr{L}_a$-superparabolic functions. Let $Q_T=Q\times(0,T)$ be a box with closure in $\oz$. If $u$ is
$\mathscr{L}_a$-superparabolic in $\oz$ and bounded on $Q_T$, we define the  $\mathscr{L}_a$-parabolic modification
$$\begin{array}{cl}
U=\left\{ \begin{aligned}
& u,\quad \ {\rm in}\ \oz\setminus Q\times(0,T],\\
&v,\quad \ {\rm in}\  Q\times(0,T],
\end{aligned}\right.
\end{array}$$
where
$$v(\xi)=\dsup\{h(\xi): h\in C(\overline{Q_T})\ {\rm is}\ \mathscr{L}_a\text{-parabolic\ and}\ h\le u\ {\rm on}\ \pz_p Q_T\}.$$
Then it is clear that $U\le u$ on $\oz$. Moreover, $U$ is $\mathscr{L}_a$-superparabolic in $\oz$  and $\mathscr{L}_a$-parabolic
in $Q_T$. To see this, choose an increasing sequence $\vz_j$ of continuous function on $\pz_p Q_T$ such that
$$u=\dlim_{j\to\fz}\vz_j$$
on $\pz_p Q_T$. Let $h_j$ be the $\mathscr{L}_a$-parabolic function in $Q_T$ that coincides with $\vz_j$ on $\pz_p Q_T$. Then it
follows from Lemma 4.4 that the sequence $h_j$ is increasing on $\overline{Q_T}$ and that the limit function is $v$. Moreover, since the sequence $h_i$ is bounded, so
$v$ is $\mathscr{L}_a$-parabolic by Lemma 4.2. Obviously, it is immediate that $U$ is a $\mathscr{L}_a$-superparabolic in $\oz$.

In what follows, we let $\oz$ be a bounded open set in $\rr^{n+1}$. Let $f:\pz\oz\to\rr$ be any bounded function.
 We now introduce the relevant notions of upper and lower Perron solutions.

A function $u$ is said to belong to the upper class ${\cal U}_f$ if $u$ is $\mathscr{L}_a$-superparabolic in $\oz$  and bounded below
$$\liminf_{\eta\to \xi} u(\eta)\ge f(\xi)$$
at each point $\xi\in\pz\oz$. Notice that the upper class ${\cal U}_f$ is never empty, for $f$ is bounded so that large constants are members of ${\cal U}_f$.

The lower class ${\cal L}_f$ is defined analogously. It consists of $\mathscr{L}_a$-superparabolic functions $v$, bounded above, satisfying
$$\dlim_{\eta\to \xi}\sup u(\eta)\le f(\xi)$$
at each point $\xi\in\pz\oz$; also, the constant $\displaystyle{\min_{\pz \oz}} f$ is in ${\cal L}_f$.

Next, the upper solution $ \overline{H}_f$ and the lower solution $\underline{H}_f$ are defined by
$$ \overline{H}_f(\xi)=\inf\{u(\xi): u\in {\cal U}_f\}$$ and
$$\underline{H}_f(\xi)=\sup\{u(\xi): u\in {\cal L}_f\}.$$
Notice that $\min\{u,\|f\|_\fz\}\in {\cal U}_f$ if $u\in  {\cal U}_f$, and $\max\{v,-\|f\|_\fz\}\in {\cal L}_f$ if $v\in  {\cal L}_f$.
Since $f$ is bounded, we can take the infimum over bounded $u$'s in ${\cal U}_f$  and the supremum over bounded $v$'s in ${\cal L}_f$.
Moreover, $\overline{H}_f, \underline{H}_f$ are bounded by the same constant as $f$.

If there exists a function $h \in  C(\overline{Q}_T)$ solving the boundary value problem to the equation $\mathscr{L}_a$, then
$$h= \overline{H}_f=\underline{H}_f.$$
To see this, simply note that the function $h$ belongs to both the upper class and the lower
class. As we will see, both $\overline{H}_f$ and $\underline{H}_f$ are local weak solutions to the equation $\mathscr{L}_a$.

An
immediate consequence of the comparison principle (Lemma 4.4) is that if $v\in \overline{H}_f$ and $u\in \underline{H}_f$, then $u\le v$. Thus
$$\underline{H}_f\le \overline{H}_f.$$
for the bounded boundary function $f$.

By using the parabolic modification, it is quite classical to show that $\overline{H}_f$ and  $\underline{H}_f$ are solutions to $\mathscr{L}_a$
(ref the proof of Theorem 5.1 in \cite{BBP}).

\begin{theorem}
If the boundary function $f: \pz\oz\to\rr$ is bounded, then the Perron solutions $\overline{H}_f$ and  $\underline{H}_f$ are $\mathscr{L}_a$-parabolic.
\end{theorem}
\begin{remark}Let $\oz=Q\times(0,T)$ and suppose that $f: \pz\oz\to\rr$ is continuous. Then the upper and lower Perron solutions coincide and
$$H_f=\underline{H}_f= \overline{H}_f$$ is the  $\mathscr{L}_a$-parabolic function that coincides with $f$ on $\pz_p \oz$. As anticipated, the values of $f$ at the top
of the box $\oz$ do note have any influence on the solution.This is so since if $h$ is $\mathscr{L}_a$-parabolic in $\oz$ in Theorem $3.1$ corresponding to
boundary $f$, then in a routine way one can show that the function $h+\frac \ez{T-t}$ belong to ${\cal U}_f$ for $\ez>0$, and to ${\cal L}_f$ for $\ez<0$.
By letting $\ez\to 0$, one see that $\overline{H}_f=\underline{H}_f$.
\end{remark}

\subsection{Barriers}
We shall define the barrier function for the boundary value problem as in classical theory.

\begin{definition}
A function $\omega$ is a barrier in $\oz$ at the point $\xi_0\in \pz\oz$ if:
\begin{enumerate}
\item[$\mathrm{(i)}$] $\omega$ is a positive  $\mathscr{L}_a$-superparabolic in $\oz$;
\item[$\mathrm{(ii)}$]$\displaystyle\liminf_{\zeta\to\xi}\omega(\zeta)>0$ if $\xi\in\oz$ and $\xi\not=\xi_0$;
\item[$\mathrm{(iii)}$]$\dlim_{\xi\to\xi_0}\omega(\xi)=0.$
\end{enumerate}
\end{definition}

As for heat equation, the existence of a barrier is a completely local question. Suppose there exists a neighborhood $N$ of $\xi_0=(X_0,t_0)$ such that a barrier at $\xi_0$
can be found in $N\cap\oz$. Then, as in \cite {BG}, we can define a barrier in $\oz$ as follows. Let $B=B_r(X_0)\times(t_0-r,t_0+r)$ be compactly contained in $N$.
Let $m=\dinf_{(N\setminus B)\cap\oz}\omega$. Without loss of generality, we can assume that $m>0$. If we define
$$\begin{array}{cl}
v=\left\{ \begin{aligned}
& \min\{\omega,m\},\quad \ {\rm in}\ \oz\cap B,\\
& m ,\qquad\qquad \quad\ {\rm in}\  \oz\setminus B,
\end{aligned}\right.
\end{array}$$
then it is easy to see that $v$ is a barrier in $\oz$ at $\xi_0$ by Lemma 4.6.

\begin{proposition}
Suppose $f: \pz\oz\to\rr$ is bounded and continuous at $\xi_0\in\pz\oz$. If there exists a barrier in $\oz$ at $\xi_0$, then
$$\dlim_{\xi\to \xi_0}\underline{H}_f(\xi)=f(\xi_0)= \dlim_{\xi\to \xi_0}\overline{H}_f(\xi). $$
\end{proposition}
\begin{proof}
Given $\ez>0$ there exists $\dz>0$ such that $|f(\xi)-f(\xi_0)|\le\ez$ if $|\xi-\xi_0|<\dz$. Because of the lower semicontinuity of $\omega$,
then there exists a constant $M>0$ such that
$$M\omega(\xi)\ge 2\dsup|f|$$
for $\xi\in\overline{\oz},\ |\xi-\xi_0|<\dz$. Then the function $M\omega+\ez+f(\xi_0)$ belongs to the upper class ${\cal U}_f$ and has the limit $f(\xi_0)+\ez$ at
$\xi_0$. Similarly, the function $-M\omega-\ez+f(\xi_0)$ belongs to  the lower class ${\cal L}_f$ and has the limit $f(\xi_0)-\ez$ at $\xi_0$, since $\ez$ is arbitrary,
the conclusion follows.
\end{proof}

\begin{definition}
A boundary point $\xi_0$ is called regular if $\dlim_{\xi\to \xi_0}\overline{H}_f(\xi)=f(\xi_0)$ whenever $f:\pz\oz\to\rr$ is continuous. In addition, if any point on $\pz\oz$ is regular, then we call $\oz$ is a regular set.
\end{definition}
Because $\underline{H}_f=-\overline{H}_{-f}$, we could replace $\overline{H}_{f}$ by $\underline{H}_f$ in the definition above.

We have the classical characterization for regularity in terms of barriers.

\begin{theorem}
A boundary point $\xi_0$ is regular if and only if there is a barrier at $\xi_0$.

\end{theorem}

\begin{proof}
The sufficiency has been established in Proposition 4.2. For the necessity, let $\xi_0=(X_0,t_0)=(x'_0,x_0,t_0)$ and define
$$g(X,t)=|X-X_0|^2+(t-t_0)^2.$$
Then $g\in C(\pz \oz)$ and since $\xi_0$ is a regular, so
$$\dlim_{\xi\to\xi_0}\overline{H}_{g}(\xi)=g(\xi_0)=0.$$
Also since $g>0$ on $\pz\oz\setminus{\xi_0}$, adapting the same argument of Lemma 6.4.4 in \cite{AG}, then there exists a Borel measure $\mu$ such that
$$\overline{H}_{g}(\xi)=\dint_{\pz\oz}gd\mu(\xi)>0,\quad\xi\in \oz.$$
Hence, $\overline{H}_{g}$ is a barrier at $\xi_0$.
\end{proof}
Since the existence of a barrier is a local property, so is the regularity of a boundary point. Moreover, if $\xi_0$ is a boundary
point of $\oz$, then $\xi_0$ is regular with respect to each subdomain to whose boundary it belongs.

We next prove an auxiliary exterior ball condition. We let ${\cal B}(\zeta,R)=\{\xi\in\rr^{n+1}: |\zeta-\xi|<R\}$ denote a ball in $\rr^{n+1}$.
\begin{lemma}
Let $\xi_0=(X_0,t_0)=(x_0',x_0,t_0)\in\pz\oz$. Suppose that there exists a ball ${\cal B}={\cal B}(\xi_1,R_1),\ \xi_1=(X_1,t_1)=(x_1',x_1,t_1)$, such that
$B\cap\oz=\emptyset$ and $\xi_0\in\pz {\cal B}\cap \pz\oz$. If $X_1\not=X_0$, or if $\xi_0$ is the north pole of ${\cal B}$ $(\text{that is},~\xi_1=(X_0,t_0-R_1))$
and the additional radius condition $R_1>n+|a|$ is satisfied, then $\xi_0$ is regular with respect to $\oz$.
\end{lemma}
\begin{proof}
By choosing a smaller ball, if necessary, we may without loss of generality assume that $\pz {\cal B}\cap \pz\oz=\{\xi_0\}$. For $\xi=(X,t)$ define
$$\omega(\xi)=e^{-jR_1^2}-e^{-jR^2},$$
where $R=|\xi-\xi_1|$ and $R_1=|\xi_0-\xi_1|$, while $j$ will be chosen later. Then $\omega >0$ in $\overline{\oz}\setminus\{\xi_0\}$ and
$\dlim_{\xi\to\xi_0}\omega(\xi)=0$.  Elementary calculations show that
$$\omega_t(X,t)=2je^{-jR^2}(t-t_1),$$
$$\nabla_x\omega(X,t)=2je^{-jR^2}(x-x_1),$$
and
$$\triangle_X\omega(X,t)=2je^{-jR^2}[n-2j|X-X_1|^2].$$
Next we consider four cases for proving the existence  of the barrier $\omega$.

Case 1: $X_1\not=X_0$ and $x_1=0$. For $x\not=0$, we then have
$$-D_t\omega(X,t)+\Delta_X\omega(X,t)+\frac a xD_x \omega(X,t)=2je^{-jR^2}[n-2j|X-X_1|^2+\frac{a(x-x_1)}x-(t-t_1)].
\eqno(4.4)$$
We can assume that $(X,t)\in\oz$ satisfies $|X-X_0|, |t-t_0|<\dz:=\frac 12|X_1-X_0|$. In particular,
$|X-X_1|>\dz$ and $t_1-t<t_1+\dz$. Hence we can choose $j$ so that the bracket in (4.4) is non-positive and thus
$$-D_t\omega(X,t)+\Delta_X\omega(X,t)+\frac a xD_x \omega(X,t)\le 0$$ for all such $X$ and $t$.
In addition, notice that
$$|x|^aD_x \omega(X,t)=|x|^a x2je^{-jR^2}\to 0,\quad {\rm if}\ x\to 0.$$
From these, then there exists a neighborhood   ${\cal N}$ of  $\xi_0$ is a  $\mathscr{L}_a$-supersolution in ${\cal N}\cap \oz$. So, $\omega(X,t)$ is a  $\mathscr{L}_a$-superparabolic in ${\cal N}\cap \oz$ by Corollary 4.1.

Case 2: $X_1\not=X_0$, $x_1\not=0$. We can assume that $(X,t)\in\oz$ satisfies $|X-X_0|, |t-t_0|, |x-x_1|<\dz:=\frac 12\min\{|X_1-X_0|,|x_1|\}$. In particular,
$|X-X_1|>\dz$, $t_1-t<t_1+\dz$ and $|x-x_1|\le |x_1|/2$. Hence we can choose $j$ so that the bracket in (4.4) is non-positive and thus
$$-D_t\omega(X,t)+\Delta_X\omega(X,t)+\frac a xD_x \omega(X,t)\le 0$$ for all such $X$ and $t$. Then there exists a neighborhood   ${\cal N}$ of  $\xi_0$ such that $\omega(X,t)$ is a  $\mathscr{L}_a$-superparabolic  in ${\cal N}\cap \oz$.

Case 3: $X_1=X_0$, $t_1=t_0-R_1$ and $x_1=0$. For $x\not=0$, we can assume that $t-t_0>n+|a|-R_1$, then
\begin{align*}
&-D_t\omega(X,t)+\Delta_X\omega(X,t)+\frac a xD_x \omega(X,t)\\
&=2je^{-jR^2}[n-2j|X-X_1|^2+\frac{a(x-x_1)}x-(t-t_1)]\\
&\le 2je^{-jR^2}[n+a-(t-t_1)]\le 2je^{-jR^2}[n+a-(n+|a|)]\le 0
\end{align*}
for all such $X$ and $t$. In addition, notice that
$$|x|^aD_x \omega(X,t)=|x|^a x2je^{-jR^2}\to 0,\quad {\rm if}\ x\to 0.$$
From these, then there exists a neighborhood   ${\cal N}$ of  $\xi_0$ such that $\omega(X,t)$ is a  $\mathscr{L}_a$-superparabolic  in ${\cal N}\cap \oz$.

Case 4: $X_1=X_0$, $t_1=t_0-R_1$ and $x_1\not=0$. We can assume that $|x-x_1|\le |x_1|/2$ and $t-t_0>n+|a|-R_1$, then
\begin{align*}
&-D_t\omega(X,t)+\Delta_X\omega(X,t)+\frac a xD_x \omega(X,t)\\
&=2je^{-jR^2}[n-2j|X-X_1|^2+\frac{a(x-x_1)}x-(t-t_1)]\\
&\le 2je^{-jR^2}[n+|a|-(t-t_1)]\le 2je^{-jR^2}[n+|a|-(n+|a|)]= 0
\end{align*}
for all such $X$ and $t$. Hence, there exists a neighborhood   ${\cal N}$ of  $\xi_0$ such that $\omega(X,t)$ is a  $\mathscr{L}_a$-superparabolic  in ${\cal N}\cap \oz$.
\end{proof}

As an application of the exterior sphere condition, we prove the characterization
of $\mathscr{L}_a$-superparabolic functions, which will be used later.

\begin{proposition}Suppose that $ u: \oz\to(-\fz, \fz]$ is lower semicontinuous in $\oz\subset\rr^{n+1}$ and finite
in  a dense subset of $\oz$. Then $u$ is $\mathscr{L}_a$-superparabolic if and  only if
for each domain $\Xi$ with compact closure in $\oz$ and each $h\in C(\Xi)$, $\mathscr{L}_a$-parabolic in $\Xi$,
the condition $h \le u$ on $\pz \Xi$ implies $h \le u$ in $\Xi$.
\end{proposition}
\begin{proof}
Ref the proof of Lemma 6.3 in \cite{KL} by using Lemmas 4.1 and 4.7.
\end{proof}

From Proposition 4.3, we know that the definitions of $\overline{H}_f$ and $\underline{H}_f$ coincide to classical definitions; see \cite{GL}.
Thus, it is quite classical to show that $\overline{H}_f=\underline{H}_f$ (Ref \cite{W} and \cite{AG}).

\begin{theorem}
If $\oz$ is a bounded domain, and $f\in C(\pz\oz)$, then
$$H_f=\overline{H}_f=\underline{H}_f.$$

\end{theorem}

\begin{remark} From Lemma $4.7$, we know that  parabolic boundary of a circle cylinder is regular set.  In addition, the cone $V_{(X_0,t_0)}: |X-X_0|\le \az(t_0-t)$ is also a regular set for $0\le t_0-t\le \bz$ with $\az,\bz>0$. Then Theorem $3.1$ also holds provided that space-time boxes
replaced by a circle cylinder or the cone $V_{(X_0,t_0)}$  by Theorem $4.3$.
\end{remark}

\section{$\mathscr{L}_a$-capacity and potentials}
This section collects together various facts concerning $\mathscr{L}_a$-capacity and potentials which we will need for the proof of Theorem 1.1 in Section 8.

Recall the definition (1.4) of $\mathscr{L}_a$-capacity. Henceforth $K$ and $K_i~(i=1,2,\ldots)$ will denote compact subsets of $\rr^{n+1}.$
\begin{proposition} $\mathrm{(a)}$ $\mathrm{cap}(K)<+\infty,$\\
${}$\qquad\qquad\qquad\qquad~$\mathrm{(b)}$ $\mathrm{cap}(K_1\cup K_2)\leq \mathrm{cap}(K_1)+\mathrm{cap}(K_2),$\\
${}$\qquad\qquad\qquad\qquad~$\mathrm{(c)}$ $K_1\subset K_2$ implies $\mathrm{cap}(K_1)\leq \mathrm{cap}(K_2)$,\\
${}$\qquad\qquad\qquad\qquad~$\mathrm{(d)}$ $\mathrm{cap}(K)=\displaystyle\lim_{i\to\infty}\mathrm{cap}(K_i),$\\
if $K_i\supset K_{i+1} (i=1,2,\ldots)$ and $K:=\displaystyle\bigcap^\infty_{i=1}K_i,$\\
${}$\qquad\qquad\qquad\qquad~$\mathrm{(e)}$ $\mathrm{cap}(K)=\mathrm{cap}(\hat{K}),$\\ where $$\hat{K}:=\{(X,t)\in\rr^{n+1}|~(X,-t)\in K\}.$$
${}$\qquad\qquad\qquad\qquad~$\mathrm{(f)}$ $\mathrm{cap}({\xi})=0$\quad  for all\quad $\xi\in\rr^{n+1}$.

\end{proposition}
\begin{proof}
(a)-(e), it follows by a step-by-step imitation of the arguments in \cite{L} and \cite{B1}.

We now prove (f). In fact, we will show a general result, that is,
let $F=A\times\{\tau\}$, where $A$ be a compact subset of $\rz$ and  $\tau\in (T_1, \fz)$,  then
$$\mathrm{cap}(F)=\dint_A|y|^adY.\eqno(5.1)$$
For any $(X,t)\in\rr^{n+1}$, let $\nu$ be the weight ($|y|^a$) measure supported on $F$, we then have
$$\Gamma*\nu(X,t)=\dint_A\Gamma(X,t; Y,\tau)|y|^adY\le 1.$$
So that
$$\mathrm{cap}(F)\ge \nu(F)=\dint_A|y|^adY.$$
Let us now prove the opposite inequality, that is,
$$\mathrm{cap}(F)\le\dint_A|y|^adY.$$
Let us consider a bounded open set $O\subset \rz$ containing $A$, and denote by $\nu$ the  weight ($|y|^a$) measure supported on $O\times \tau$.
From the proof of (ii) of Theorem 2.1, we know that
$$\dlim_{(X,t)\to (X_0,\tau),t>\tau}\Gamma*\nu(X,t)=1\quad {\rm for\ every}\ X_0\in O.$$
Let us pick $\mu\in M^+(F)$ such that
$\Gamma*\mu\le 1$ in $\rr^{n+1}$. Then $u=\Gamma*\nu-\Gamma*\mu$ is $\mathscr{L}_a$-parabolic in $\rz\times(\tau, \fz)$, and it satisfies
$$\liminf_{(X,t)\to (X_0,\tau),~t>\tau} u(X,t)\ge 0\quad {\rm for\ all}\ X_0\in O.$$
This inequality also holds at any point $X_0\not\in O$ since in this case we have
$$0\le \limsup_{(X,t)\to (X_0,\tau),~t>\tau} \Gamma*\mu(X,t)\le 0$$
and
$$u(\xi)\to 0\quad{\rm as}\ |\xi|\to \fz,$$
where we used the following fact that
$$G(X,t;Y,\tau)\le C(|X-Y|^2+|t-\tau|)^{-\frac n2}(|t-\tau|+|y|^2)^{-\frac a2}e^{-\frac{|X-Y|^2}{8(t-\tau)}},\quad -1<a\le 0, \eqno(5.2)$$ and
$$G(X,t;Y,\tau)\le C(|X-Y|^2+|t-\tau|)^{-\frac {n+a}2}e^{-\frac{|X-Y|^2}{8(t-\tau)}},\quad 0<a<1.\eqno(5.3)$$
Then, the comparison principle implies $u\ge 0$ in $\rz\times(\tau, \fz)$, that is,
$$\Gamma*\mu\le \Gamma*\nu\quad {\rm in}\ \rz\times(\tau, \fz).$$
This inequality extends to $\rr^{n+1}\setminus F$ since $\Gamma*\mu=0$ in $(\rz\times(T_1,\tau])\setminus F$. Then, we get
\begin{align*}
\mu(F)&=\dint_Fd\mu=\dint_F \l(\dint_\rn\Gamma(X,t;Y,\tau)|x|^adX\r)d\mu(Y,\tau)\\
&=\dint_\rn \l(\dint_F\Gamma(X,t;Y,\tau)d\mu(Y,\tau)\r)|x|^adX\\
&\le\dint_\rn \l(\dint_F\Gamma(X,t;Y,\tau)d\nu(Y,\tau)\r)|x|^adX\\
&=\dint_F \l(\dint_\rn\Gamma(X,t;Y,\tau)|x|^adX\r)d\nu(Y,\tau)=\nu(F).
\end{align*}
Since this holds true for every $\mu\in M^+(F)$ with $\Gamma*\mu\le 1$, we finally obtain
$$\mathrm{cap}(F)\le\nu(F)=\dint_A|y|^adY.$$
\end{proof}

Now let $K\subset \rr^{n+1}$ be compact and we next define a   balayage as follows.

$$v^{\mathscr{L}_a}_K(\xi):=\dinf\{u: u\in \Phi_k\},$$
where $\Phi_k=\{u(\xi)~|~u~ \text{is}~\mathscr{L}_a\text{-superparabolic}~ \text{in}~\rr^{n+1},~u\geq 1~ \text{on}~K,~u\geq 0\}$. Moreover, we define the balayage by
$$V^{\mathscr{L}_a}_K(\xi):=\liminf_{\zeta\to \xi}v^{\mathscr{L}_a}_K(\zeta).$$
In what follows, when no confusion will arise, we simply denote $v^{\mathscr{L}_a}_K$ and $V^{\mathscr{L}_a}_K$ by $v_K$ and $V_K$, respectively.
\begin{proposition}
$\mathrm{(a)}$ $0\leq V_K\leq v_K\le 1$ in $\rr^{n+1},$\\
${}$\qquad\qquad\qquad\qquad~$\mathrm{(b)}$ $V_{K_1\cap K_2}\le V_{K_1}+V_{K_2}$  for every pair of compact sets $K_1, K_2\subset\rr^{n+1}$,\\
${}$\qquad\qquad\qquad\qquad~$\mathrm{(c)}$ $V_K=v_K$ in $(\p K)^c,$\\
${}$\qquad\qquad\qquad\qquad~$\mathrm{(d)}$ $v_K=1$ on $K,$\\
${}$\qquad\qquad\qquad\qquad~$\mathrm{(e)}$ $\displaystyle\lim_{|z|\to \infty}v_K=\lim_{|z|\to\infty}V_K=0,$\\
${}$\qquad\qquad\qquad\qquad~$\mathrm{(f)}$ $V_K=0$ on the trip $S^{-}_K=\{(X,t): t\le \tau,~\forall\ (Y,\tau)\in K\}$,\\
${}$\qquad\qquad\qquad\qquad~$\mathrm{(g)}$ $V_K$ is a $\mathscr{L}_a$-superparabolic in $\rr^{n+1}$,  and a $\mathscr{L}_a$-parabolic in $K^c,$\\
${}$\qquad\qquad\qquad\qquad~$\mathrm{(h)}$ There exists a unique measure $\tilde{\mu}\in M^+(K)$ such that $$V_K=P_{\tilde{\mu}}:=\Gamma*\tilde{\mu}$$
and $$\tilde{\mu}(\rr^{n+1})=\mathrm{cap}(K).$$
Furthermore $$\mathscr{L}_aV_K=\tilde{\mu}~\text{in the sense of distributions on} ~\rr^{n+1}.$$
$V_K$ is called the $\mathscr{L}_a$-equilibrium potential of $K$ and $\tilde{\mu}$  the $\mathscr{L}_a$-equilibrium measure, and $P_{\tilde{\mu}}$
$\mathscr{L}_a$-equilibrium potential. \\
${}$\qquad\qquad\qquad\qquad~$\mathrm{(i)}$ If $\tilde{\mu}$ is the $\mathscr{L}_a$-equilibrium measure for $K$, then $$P_{\mu}\leq P_{\tilde{\mu}}$$
for all $\mu\in M^+(K)$ such that $P_\mu\leq 1$ in $\rr^{n+1}.$
\end{proposition}
\begin{proof}
The statement (a)-(d) and (g) are consequence of general results from balayage theory in abstract harmonic space, see Proposition 5.3.3 in \cite{CC}.

We now prove (e) and (f).  We first claim:

 let  $\xi_0\in\rr^{n+1}$ be fixed, $u(\xi)=\Gamma(\xi,\xi_0),\ \xi\in\rr^{n+1}$, then
$u$ is $\mathscr{L}_a$-superparabolic in $\rr^{n+1}$.

Now we prove this claim. Obviously, $u$ is lower semicontinuous. Let $V$ be an $\mathscr{L}_a$-regular open set and let $\vz\in C(\pz V)$, $\vz\le u$ on $\pz V$, and the Perron solution $H_\vz$ is taken  in $V$. If $\xi_0=(X_0,t_0)\not\in V$, then
$$\liminf_{\zeta\to\xi}(u(\xi)-H_\vz(\xi))\ge u(\xi)-\vz(\xi)\ge 0,\quad \forall \xi\in\pz V,$$
and $u-H_\vz$ is $\mathscr{L}_a$-parabolic in $V$. Thus, by Lemma 4.5, $u\ge H_\vz$ in $V$. Let us now suppose $\xi_0\in V$. Since $u=0$ on $\pz_{\tau_0}V=\pz V\cap\{\tau\le \tau_0\}$ we have $\vz\le 0$ on $\pz_{\tau_0}V$ so that, by Lemma 4.5 again, $H_\vz\le 0$ in $V_{\tau_0}=V\cap\{t\le\tau_0\}$. Thus $H_\vz(\xi_0)\le 0$. As a consequence, letting $W=V\setminus\{\xi_0\}$, $\dlim_{\zeta\to\xi}(u(\zeta)-H_\vz(\zeta))\ge 0$ for every $\xi\in\pz W$. On the other hand, $u-H_\vz$ is $\mathscr{L}_a$-parabolic in $W$. Hence,
$u\ge H_\vz$ in $W$. The inequality extends to $V$ since $u(\xi_0)=0\ge H_\vz(\xi_0)$. This completes the claim.

Let us continue to prove (e) and (f). For every fixed $\xi_0$ in the interior of $S^{-}_K$ define
$$m=\max_K\dfrac 1{\Gamma(\cdot,\xi_0)},\quad {\rm}\quad v=m\Gamma(\cdot,\xi_0).$$
From the claim above, we know that $v\in\Phi_k$ so that $V_k\le v_k\le v$. On the other hand, $\Gamma(\xi_0,\xi_0)=0$ and $\Gamma (\xi,\xi_0)\to 0$ as $\xi\to\fz$ by (5.2) and (5.3).
Thus $\displaystyle\lim_{|z|\to \infty}v_K=0$ and $v_k=0$ in the interior of $S^{-}_K$. So,
 (e) and (f) hold.

Since $V_K$ is a bounded $\mathscr{L}_a$-superparabolic in $\rr^{n+1}$, so there exists a nonnegative measure $\tilde{\mu}$ on $\rr^{n+1}$ such that
$$\mathscr{L}_aV_K=\tilde{\mu}~\text{in the sense of distributions on} ~\rr^{n+1}.$$

It is easy to see that $P_{\tilde{\mu}}=\Gamma*\tilde{\mu}$ is $\mathscr{L}_a$-superparabolic in $\rr^{n+1}$ and
$$\mathscr{L}_a P_{\tilde{\mu}}=\tilde{\mu}~\text{in the sense of distributions on} ~\rr^{n+1}.$$
Hence,
$$\mathscr{L}_a(V_K-P_{\tilde{\mu}})= 0~\text{in the sense of distributions on} ~\rr^{n+1}.$$
From this, and note that $V_K(\xi)\to 0, P_{\tilde{\mu}}(\xi)\to 0$ if $|\xi|\to \fz$, thus by Harnack inequality in \cite{CF},  we have
$$V_K(\xi)=P_{\tilde{\mu}}(\xi)\quad {\rm almost\ all}\quad  \xi\in\rr^{n+1}.$$
From this and by Lemma 4.3, we get
$$V_K(\xi)=P_{\tilde{\mu}}(\xi)\quad {\rm for\ all}\quad  \xi\in\rr^{n+1}.$$
The rest can be proved, refer to \cite{L}~(page 88) for it.
\end{proof}

\begin{remark}
In the view of Proposition $5.2$ $(h)$ there also exists a unique measure ${\mu}^*\in M^+(K)$ and a lower semicontinuous function $V^*_K$,
$0\leq V^*_K\leq 1$ such that $$\mu^*(\rr^{n+1})=\mathrm{cap}(K)$$
and $${\mathscr{L}^*_a}V^*_K=\mu^*~\text{in the sense of distributions},$$
where $${\mathscr{L}^*_a}u:=-\p_t(|y|^au)-\mathrm{div}_{Y}(|y|^a\nabla_{Y}u)\eqno(5.4)$$
is the backwards $\mathscr{L}_a$ operator. We call $V^*_K$ the backwards $\mathscr{L}_a$-equilibrium potential of $K$ and $\mu^*$ the backwards $\mathscr{L}_a$-equilibrium measure. In fact
$$V^*_K(Y,t)=V_{\hat{K}}(Y,-t),\quad (Y,t)\in \rr^{n+1}.$$
\end{remark}

We turn now to the question of principal interest, the characterization of regular boundary points of $\Omega$. Fix some $\xi_0=(X_0,t_0)\in \p\Omega$. The regularity of $\xi_0$ can be shown to be equivalent to the existence of a barrier at $\xi_0$, but for our purposes the next lemma is more convenient.

Define, for $r>0$, the closed cylinder $$C_r:=\{(X,t)\in\rr^{n+1}|-c_1r^2\leq t-t_0\leq 0,~~ |X-X_0|\leq c_2r\}.\eqno(5.5)$$
\begin{proposition}
With $C_r$ defined by $(5.5)$, the point $\xi_0$ is $\mathscr{L}_a$-regular for $\Omega$ if and only if $$V_{\Omega^c\cap C_r}(\xi_0)=1~~\text{for some}~r>0.$$
If $\xi_0$ is not $\mathscr{L}_a$-regular we have
$$\dlim_{r\to 0^+}V_{C_r\setminus\oz}(\xi_0)=0.$$
\end{proposition}
\begin{proof}
See, e.g., Lemma 1.3 in \cite{L} and Theorem 4.3.1 in \cite{Ba}.
\end{proof}

As a consequence of Proposition 5.3, we have the following result.
\begin{corollary}
With $C_r$ defined by $(5.5)$, the point $\xi_0$ is $\mathscr{L}_a$-regular for $\Omega$ if and only if  for every $c>0$ it is
$\mathscr{L}_a$-regular for the open set $\oz_{r,c}$, where
$$\oz_{r,c}=\oz\cup \{\xi\in \dot{C}_r: \Gamma(\xi_0,\xi)<(4\pi c)^{-\frac{n+a}{2}}(1+\frac{x_0^2}c)^{-\frac a2}\},$$
here and in what follows, $\dot{E}$ denotes the interior of $E$.
\end{corollary}
\begin{proof}
We first observe that $\xi_0=(x'_0,x_0,t_0)\in\pz\oz_{r,c}$. Since $\oz\subset \oz_{r,c}$, it is clear that $\xi_0$ is $\mathscr{L}_a$-regular for $\oz$ if it is $\mathscr{L}_a$-regular for
 $\oz_{r,c}$. Therefore, it suffices to show that if $\xi_0$ is $\mathscr{L}_a$-irregular for $\oz_{r,c}$, then it is $\mathscr{L}_a$-irregular for $\oz$. Now if $\xi_0$ is
$\mathscr{L}_a$-irregular for $\oz_{r,c}$, by proposition 5.3, we have
$$\dlim_{\rho\to 0^+}V_{C_{\rho}\setminus\oz}(\xi_0)=0.\eqno(5.6)$$
But for $0<\rho<r$
$$C_\rho\setminus\oz\subset \oz\subset \oz_{r,c}\cap F_{\rho,c},$$
where
$$F_{\rho,c}=\{\xi\in C_{\rho}: \Gamma(\xi_0,\xi)\le (4\pi c)^{-\frac{n+a}{2}}\l(1+\frac{x_0^2}c\r)^{-\frac a2}\}.$$
Then
$$V_{C_\rho\setminus\oz}(\xi_0)\le V_{C_\rho\setminus\oz_{r,c}}(\xi_0)+V_{ F_{\rho,c}}(\xi_0).\eqno(5.7)$$
To prove that $\xi_0$ is $\mathscr{L}_a$-irregular for $\oz$ by (5.6) and (5.7) it is then sufficient to show that
$$\dlim_{\rho\to 0^+}V_{ F_{\rho,c}}(\xi_0)=0.\eqno(5.8)$$
To this purpose we observe that if $\mu$ is the equilibrium measure of $ F_{\rho,c}$, we have
\begin{align*}
V_{ F_{\rho,c}}(\xi_0)&=\int_{F_{\rho,c}}\Gamma(\xi_0;\xi)d\mu(\xi)\le C(4\pi c)^{-\frac{n+a}{2}}\l(1+\frac{x_0^2}c\r)^{-\frac a2}\mu(F_{\rho,c})\\\tag{5.9}
&=C(4\pi c)^{-\frac{n+a}{2}}\l(1+\frac{x_0^2}c\r)^{-\frac a2}{\rm cap}(F_{\rho,c})\\
&\le (4\pi c)^{-\frac{n+a}{2}}\l(1+\frac{x_0^2}c\r)^{-\frac a2}{\rm cap}(C_\rho)\\
&\le C (4\pi c)^{-\frac{n+a}{2}}\l(1+\frac{x_0^2}c\r)^{-\frac a2}w_a(B_\rho),
\end{align*}
where $B_\rho=B(X_0,\rho)$, $w_a(E)=\int_E|x|^adX$ and $C$ depending only on $n,a$. The last inequality is used the following inequality
$${\rm cap}(C_\rho)\le C w_a(B_\rho).\eqno(5.10)$$
Hence, (5.8) can be deduced from (5.9) if (5.10) holds.

It remains to  prove (5.10). Let us put $\xi_\rho=(X_0,t_\rho)$ with $t_\rho=t_0+\rho^2$.  For every $\xi=(X,t)\in C_\rho$, by (2.5), we have
$$\Gamma(\xi_\rho,\xi)\ge c\rho^{-(n+a)}\l(1+\frac {x_0^2}{\rho^2}\r)^{-\frac a2}\exp(-\frac {c_2^2}{2})\ge \dfrac{1}{C w_a(B_\rho)}. $$
As a consequence, if $v$ and $\nu$ are respectively, a equilibrium potential and measure of $C_\rho$, we have
$$1\ge v(\xi_\rho)=\dint_{C_\rho}\Gamma(\xi_\rho,\xi)d\nu(\xi)\ge  \frac {\nu(C_\rho)}{C w_a(B_\rho)},$$
and hence
$${\rm cap}(C_\rho)=\nu(C_\rho)\le C w_a(B_\rho).$$
Thus, (5.10) holds.
\end{proof}
\section{Green functions, mean value formulas and a strong form of Harnack inequality}
\subsection{Green functions}
We first introduce Green functions on a cylinder. Set the cylinder $V=B(X_0,r)\times[T_1,T_2]$,  and the continuous function $\vz: \pz_p V\mapsto\rr$, since $\pz_p V$ is a regular set, then there exists a unique solution denoted by $H_\vz^V$ by  Lemma 4.1,  then for $\xi\in V$, the map
$$\vz\mapsto H_\vz^V(\xi)$$
defines a linear and positive functional on $C(\pz_p V)$. As a consequence, there exists a Radon measure $\mu_\xi^V$ supported on $\pz_p V$ such that
$$H_\vz^V(\xi)=\dint_{\pz_p V}\vz(\zeta)d\mu_\xi^V(\zeta)\quad \forall \vz\in C(\pz_p V).$$
In particular, we take $\vz(X,t)=\Gamma(X,t;Y,\tau) $ on $\pz_p V$.

Then the Green function on $V$ based on $H^V_\Gamma$  is defined by
\begin{align*}\tag{6.1}
G^V(X,t;Z,s)&=\Gamma(X,t;Z,s)-H^V_\Gamma\\
&=\Gamma(X,t;Z,s)-\dint_{\pz B(X_0,r)\times[s,t]}\Gamma(Y,\tau; Z,s)d\mu^V_{(X,t)}(Y,\tau)\\
\end{align*}
for some nonnegative Radon measure $\mu^V_{(X,t)}(Y,\tau)$, which vanishes if $\tau>t$.

Obviously, the Green function $G^V(X,t;Z,s)$ satisfies that
$$\mathscr{L}_a (G^V(\cdot;Z,s))=0\quad {\rm in}\ B(X_0,r)\times(T_1,T_2)\setminus\{(Z,s)\},$$
and
$$G^V(\cdot;Z,s)=0 \quad {\rm in}\ \pz B(X_0,r)\times(T_1,T_2),\ {\rm or}\ t\le s.$$

We are interested in the bounds for the Green function $G^V$. Since $G^V\le \Gamma$, of course we have
$$G^V(X,t;Z,s)\le \Gamma (X,t;Z,s),\ \forall\ (X,t), (Z,s)\in V.$$
We can adapt the arguments in \cite{BU} in order to prove a lower bound for the Green function $G^V$.
\begin{lemma}
Let $\gz\in (0,1)$ and $T\ge 1$. There exists a positive constant  $c=c(\gz,a,n, T)$ such that
$$G^V(X,t;Z,s)\ge\frac c{w_a(B(X,\sqrt{t-s}))}$$
for every $ X, Z \in  B(X_0,\gz r)$ and $t,s\in (T_1,T_2)$ satisfying $|X-Z|^2/T<t-s\le Tr^2$.
\end{lemma}
\begin{proof}
We first claim that let $\dz\in (0,1)$, there exists a  positive constant $\rho\in(0,1)$ such that
$$G^V(X,t;Z,s)\ge \frac 12 \Gamma(X,t;Z,s),$$
if $X, Z\in B(X_0,\dz r)$ and  $|X-Z|^2\le t-s<\rho r^2$.

Indeed, from (6.1), we have
$$G^V(X,t;Z,s)=\Gamma(X,t;Z,s)\l(1-\dint_{\pz B(X_0,r)\times[s,t]}\dfrac{\Gamma(Y,\tau; Z,s)}{\Gamma(X,t;Z,s)}d\mu^V_{(X,t)}(Y,\tau)\r).$$
Since $\mu^V_{\pz B(\xi_0,r)\times[s,t]}\le 1 $, to prove above claim, it is enough to bound uniformly from above the ratio $ \frac {\Gamma(Y,\tau;Z,s)}{\Gamma(X,t;Z,s)}$
with something going to $0$ as $\rho\to 0$. To this aim, we get
\begin{align*}
 \dfrac {\Gamma(Y,\tau;Z,s)}{\Gamma(X,t;Z,s)}&\le C\l(\dfrac {t-s}{\tau-s}\r)^{\frac{n}2}\l(\dfrac {t-s+|z|^2}{\tau-s+|z|^2}\r)^{\frac{a}2}
 \exp\l({\dfrac{|X-Z|^2}{2(\tau-s)}}\r)\exp\l(-\dfrac{|Y-Z|^2}{6(\tau-s)}\r)\\
&\le C\l(\dfrac {t-s}{\tau-s}\r)^{\frac{n+|a|}2}
 \exp\l({\dfrac{|X-Z|^2}{2(t-s)}}\r)\exp\l(-\dfrac{|Y-Z|^2}{6(\tau-s)}\r)\\
&\le C\l(\dfrac {t-s}{\tau-s}\r)^{\frac{n+|a|}2}\exp\l(-\dfrac{|X-Z|^2}{12(\tau-s)}\r)
 \exp\l({\dfrac{|X-Z|^2}{2(t-s)}}\r)\exp\l(-\dfrac{|Y-Z|^2}{12(\tau-s)}\r)\\
&\le C\l(\dfrac {t-s}{|Y-Z|^2}\r)^{\frac{n+|a|}2}\exp\l({\dfrac{|X-Z|^2}{2(t-s)}}\r)\exp\l(-\dfrac{|Y-Z|^2}{12(\tau-s)}\r).
\end{align*}
Note that $|Y-Z|\ge |Y-X_0|-|Z-X_0|\ge (1-\dz)r$, $\tau-s\le t-\tau\le \rho r^2$ and $|X-Z|^2<t-s $, we obtain
$$\dfrac {\Gamma(Y,\tau;Z,s)}{\Gamma(X,t;Z,s)}\le C\rho^{\frac{n+|a|}2}\exp\l(-\dfrac{|Y-Z|^2}{12(\tau-s)}\r)$$
for a suitable positive structure constant $C$. Hence, we have
$$\dfrac {\Gamma(Y,\tau;Z,s)}{\Gamma(X,t;Z,s)}\le C\rho^{\frac{n+|a|}2}.$$
Thus, this claim is proved provided that $C\rho^{\frac{n+|a|}2}=\frac 12$.

To end the proof, we need the following reproduction property of $G^V$ holds:
$$G^V(X,t;Z,s)=\dint_{B(X_0,r)}G^V(X,t;Y,\tau)G^V(Y,\tau;Z,s)|y|^adY$$
for every $t>\tau>s$ and $X,Z\in B(X_0,r)$.
In fact, we fix $Z, \tau, s$ as above, and we set $\vz=G^V(\cdot,\tau; Z,s)$. Then $\vz\in C(\overline{B(X_0,r)})$, $\vz=0$ on $\pz B(X_0,r)$. Therefore, we obtain that
$$u(X,t)=\dint_{B(X_0,r)}G^V(X,t;Y,\tau)\vz(Y)|y|^adY,\quad x\in \overline{B(X_0,r)},\ t>\tau,$$
satisfies these fact that $u$ is $\mathscr{L}_a$-parabolic in $V$, $u=0$ on $\pz B(X_0,r)\times[T_1,T_2)$ and $u(\cdot,\tau)=\vz$ in $\overline{B(X_0,r)}$.
It is easy to see that $G^V(\cdot;Z,s)$ have the same properties, hence, by Lemma 4.5, we know that $u=G^V$, so the reproduction property of $G^V$ holds.

We continue to prove.  Let $\dz=(1+\gz)/2$, and let $k$ be the smallest integer greater than $\max\{T/\rho,|X-Z|^2/(t-s)\}$ and $\sz=\frac 14 \sqrt{(t-s)/(k+1)}$. Now let $t=t_0, t_1,\cdots, t_{k+1}=\tau$ be such that
$t_j-t_{j+1}=(t-s)/(k+1)$ for $j=1,\cdots k$. It is easy to see that there a chain points of $\rn$, $X=Z_0,Z_1,\cdots, Z_{k+1}=Z$ such that
$$|Z_j-Z_{j+1}|\le\frac{|X-Z|}{k+1},\quad |Z_j-X_0|\le \frac {\gz+1}2r.$$
Moreover, for every $Y_j\in B(Z_j,\sz)$ and $Y_{j+1}\in B(Z_{j+1},\sz)$, we have
$$|Y_j-Y_{j+1}|^2<\frac {t-s}{{k+1}}.$$
Set $S=\prod_{j=1}^kB(Z_j,\sz)$ and $Y_0=X, Y_{k+1}=Z$, using the reproduction property of $G^V$ repeatedly, we obtain
\begin{align*}
 G^V(X,t;Z,s)&=\dint_{(B(X_0, r))^k}G^V(X,t;Y_1,t_1)G^V(Y_1,t_1;Y_2,t_2)\\
 &\qquad\qquad\qquad\times\cdots\times G^V(Y_k,t_k;Z,s)|y_1|^adY_1\cdots |y_k|^adY_k\\
 &\ge \dint_S G^V(X,t;Y_1,t_1)G^V(Y_1,t_1;Y_2,t_2)\\
 &\qquad\qquad\qquad\times\cdots\times G^V(Y_k,t_k;Z,s)|y_1|^adY_1\cdots |y_k|^adY_k.
\end{align*}
In addition, note that $Y_{j+1}\in B(X_0,\dz r)$, $|Y_j-Y_{j+1}|^2<\frac {t-s}{{k+1}}=t_j-t_{j+1}<Tr^2/(k+1)<\rho r^2$.
Therefore, by (2.5), we apply the claim above and obtain
\begin{align*}
 G^V(X,t;Z,s) &\ge c\l(\dfrac {t-s}{{k+1}}\r)^{-(\frac{n+a}2)(k+1)}\l(1+\dfrac{x^2}{(\frac {t-s}{{k+1}})}\r)^{-\frac a2}\\
 &\quad\times  \dint_S\l(\prod_{j=1}^k\l(1+\frac{y_j^2}{t_j-t_{j+1}}\r)^{-\frac a2}\exp\l(-\frac{|Y_j-Y_{j+1}|^2}{2(t_j-t_{j+1})}\r)|y_j|^adY_j\r)\\
 &\ge c\l(\dfrac {t-s}{{k+1}}\r)^{-(\frac{n+a}2)(k+1)}\l(1+\dfrac{x^2}{(\frac {t-s}{{k+1}})}\r)^{-\frac a2}\exp(-\frac 12 k)\\
  &\quad\times \dint_S\l(\prod_{j=1}^k\l(1+\frac{y_j^2}{t_j-t_{j+1}}\r)^{-\frac a2}|y_j|^adY_j\r)\\
 &\ge c\l({t-s}\r)^{-\frac{n+a}2}\l(1+\dfrac{x^2}{ {t-s}}\r)^{-\frac a2}\exp(-C k).
\end{align*}
Now, if $|X-Z|^2\ge T(t-s)/\rho$, from the definition of $k$ it follows that $k\le |X-Z|^2/(t-s)$ and then
$$ G^V(X,t;Z,s)\ge c\l({t-s}\r)^{-\frac{n+a}2}\l(1+\dfrac{x^2}{ {t-s}}\r)^{-\frac a2}\exp(-C\frac{|X-Z|^2}{t-s}).$$
On the other hand, if $|X-Z|^2 <T(t-s)/\rho$, the definition of $k$ gives $k<T/\rho$ and then
$$ G^V(X,t;Z,s)\ge c\l({t-s}\r)^{-\frac{n+a}2}\l(1+\dfrac{x^2}{ {t-s}}\r)^{-\frac a2}.$$
From these, we get
$$ G^V(X,t;Z,s)\ge \frac {c}{w_a(B(X,\sqrt{t-s}))}$$
for every $X, Z \in  B(X_0,\gz r)$ and $t,s\in (T_1,T_2)$ satisfying $|X-Z|^2/T<t-s\le Tr^2$, where we used the double property of the weight $w_a$.
\end{proof}
Finally, we give a Green formula, that is,
if  $u$ is $\mathscr{L}_a$-parabolic in $V$, $u$ is continuous on $\pz_p V$, then we give the represent of $u$ in terms of  $G^V$  for $(X,t)\in B(X_0,r)\times(T_1,T_2)$,
\begin{align*}
u(X,t)&=\dint_{|Y-X_0|<r}G^V(X,t;Y, T_1)u(Y,T_1)|y|^adY\\
&\quad-\dint_{T_1}^{T_2}\dint_{|Y-X_0|=r}\frac{\pz G^V(X,t;Y,s)}{\pz v(Y)}u(Y,s)|y|^ad\sz(Y)ds.
\end{align*}

\subsection{Mean value formulas}

In this section, we will prove some mean value formulas.

We let $u, v\in {\cal H}_a$, here and in what follows,
 $${\cal H}_a:={\cal H}^1_a\cap {\cal H}_a^2\cap {\cal H}_a^3,\eqno(6.2)$$
 where
$${\cal H}^1_a= \{u: u(Y,\tau)\in C(\rr^{n+1})\cap (C^2(\rr^{n+1}\setminus \{y=0\})\cap L^2_{\loc}(\rr, H^1_{\loc}(\rz, |y|^a))\},$$
$${\cal H}^2_a= \{u: \dlim_{y\to 0^+}|y|^aD_yu(Y,\tau)=\dlim_{y\to 0^-}|y|^aD_yu(Y,\tau)\},$$
and
$${\cal H}^3_a= \{u: \dsup_{0<|y|<1}|y|^a|D_yu(Y,\tau)~|\in L_{\loc}(\rr^{n-1}\times\rr)\}.$$

As in \cite{FG}, our starting point is the formula $$v\mathscr{L}_au-u\mathscr{L}^*_av=\text{div}_{Y,t}(|y|^au\nabla_{Y}v-|y|^av\nabla_{Y}u,|y|^auv),$$
where $\mathscr{L}^*_a$ is defined by (5.4). We assume that $D\subset \rr^{n+1}$ is a bounded domain with smooth boundary and $u,v\in {\cal H}_a$. Since $u, v\in {\cal H}_a$, the divergence theorem yield $$\int_{D}(v\mathscr{L}_au-u\mathscr{L}^*_av)dYdt=\int_{\pz D}|y|^a(u\nabla_{Y}v-v\nabla_{Y}u)\cdot \overrightarrow{N}_{Y}+|y|^auv{N}_tdH_{n},\eqno(6.3)$$
where $\overrightarrow{N}=(\overrightarrow{N}_{Y},N_t)$ is the outward normal to $D$ and $dH_{n}$ denotes $n$-dimensional Hausdorff measure. If we take $v\equiv 1$ in $(6.3)$, we obtain $$\int_D \mathscr{L}_audYdt=\int_{\p D}-|y|^a\nabla_{Y}u\cdot \overrightarrow{N}_{Y}+|y|^auN_tdH_{n}.\eqno(6.4)$$

Now we fix $\xi_0=(x'_0,x_0,t_0)\in \rr^{n+1}$, $r>0$ and consider the set $\Omega_r(\xi_0)$ defined by
 $$ \Omega_r(\xi_0)=\big\{(Y,t)\in \rr^{n+1}~\big|~\Gamma(X_0,t_0;Y,t)>{(4\pi r)^{-\frac{n+a}{2}}}(1+\frac{x_0^2}r)^{-\frac a2}\big\}.$$

 For $\varepsilon\in (0,r)$, we let $D_\varepsilon=\Omega_r(\xi_0)\cap\{(Y,t)~|~t<t_0-\varepsilon\}$. We set $v(Y,t)=\Gamma(X_0,t_0;Y,t)$. Denote $$\Psi_r(\xi_0)=\p \Omega_r(\xi_0)=\big\{(Y,t)\in \rr^{n+1}~\big|~\Gamma(X_0,t_0;Y,t)={(4\pi r)^{-\frac{n+a}{2}}}(1+\frac{x_0^2}r)^{-\frac a2}\big\}\cup\{\xi_0\}.\eqno(6.5)$$
For $\varepsilon$ as above, we set $$\Psi_\varepsilon=\Psi_r(\xi_0)\cap\{(Y,t)~|~t<t_0-\varepsilon\}$$ and $$I_\varepsilon=\overline{\Omega}_r(\xi_0)\cap\{(Y,t)~|~t=t_0-\varepsilon\}.$$
With this notation we have $\p D_\varepsilon=\Psi_\varepsilon\cup I_\varepsilon.$ We let $u\in {\cal H}_a$, and apply $(6.3)$ to this function, to $v$ defined above and to the set $D_\varepsilon$.  Notice that $\mathscr{L}^*_av=0$ in $D_\varepsilon\backslash\{y=0\}$ and $\displaystyle\lim_{y\to 0^+}|y|^av_y=\displaystyle\lim_{y\to 0^-}|y|^av_y$ in $D_\varepsilon$, then we deduce
$$\int_{D_\varepsilon}\Gamma \mathscr{L}_au dYdt=\int_{\Psi_\varepsilon\cup I_\varepsilon}|y|^a(u\nabla_{Y}\Gamma-\Gamma\nabla_{Y}u)\cdot \overrightarrow{N}_{Y}+|y|^au\Gamma{N}_tdH_{n},\eqno(6.6)$$
where $\Gamma=\Gamma(X_0,t_0;Y,t).$ Since  $\Gamma={(4\pi r)^{-\frac{n+a}{2}}}(1+\frac{x_0^2}r)^{-\frac a2}$ on $\Psi_\varepsilon$ and $N_t=1$ on $I_\varepsilon,$ we obtain from (6.6)
\begin{align*}
\int_{D_\varepsilon}\Gamma \mathscr{L}_au dYdt=&{(4\pi r)^{-\frac{n+a}{2}}}(1+\frac{x_0^2}r)^{-\frac a2}\int_{\Psi_\varepsilon}|y|^auN_t-|y|^a\nabla_{Y}u\cdot\overrightarrow{N}_{Y}dH_{n}\\
&+\int_{I_\varepsilon}|y|^au\Gamma dH_{n}+\int_{\Psi_\varepsilon}|y|^au\nabla_{Y}\Gamma\cdot\overrightarrow{N}_{Y}dH_{n}.\tag{6.7}
\end{align*}
Since $\Gamma(X_0,t_0;Y,t)$ is the fundamental solution with pole at $(X_0,t_0)$ we have
$$\lim_{\varepsilon\rightarrow 0^+}\int_{I_\varepsilon}|y|^au\Gamma dH_{n}=u(X_0,t_0).$$
Letting then $\varepsilon\rightarrow 0^+$ in (6.7)  finally gives
\begin{align*}
\int_{\Omega_r(\xi_0)}\Gamma \mathscr{L}_au dYdt=&{(4\pi r)^{-\frac{n+a}{2}}(1+\frac{x_0^2}r)^{-\frac a2}}\int_{\Psi_r(\xi_0)}|y|^auN_t-|y|^a\nabla_{Y}u\cdot\overrightarrow{N}_{Y}dH_{n}\\
&+u(X_0,t_0)+\int_{\Psi_r(\xi_0)}|y|^au\nabla_{Y}\Gamma\cdot\overrightarrow{N}_{Y}dH_{n}.
\end{align*}
In virtue of (6.4) we finally conclude
$$-\int_{\Psi_r(\xi_0)}|y|^au\nabla_{Y}\Gamma\cdot\overrightarrow{N}_{Y}dH_{n}=u(X_0,t_0)+\int_{\Omega_r(\xi_0)}\mathscr{L}_au\left({(4\pi r)^{-\frac{n+a}{2}}(1+\frac{x_0^2}r)^{-\frac a2}}-\Gamma
\right)dYdt.\eqno(6.8)$$
Formula (6.8) is our starting point. For convenience,  let
$$\vz(r)=(4\pi r)^{\frac{n+a}{2}}(1+\frac{x_0^2}r)^{\frac a2}.\eqno(6.9)$$
Changing $r$ in $\rho$ in (6.8), multiplying both sides by
$$\vz'(\rho)={(4\pi \rho)^{\frac{n+a}{2}}}(1+\frac{x_0^2}\rho)^{\frac a2-1}(\frac {n+a}{2\rho}+\frac{nx_0^2}{2\rho^2})>0\eqno(6.10)$$ and integrating in $\rho$ between 0 and $r$,  recalling that $\overrightarrow{N}_{Y}=-\nabla_{Y}\Gamma/|(\nabla_{Y} \Gamma, \p_t \Gamma)|$  on $\Psi_r(\xi_0)\backslash\{\xi_0\}$, we have
\begin{align*}
&\int^r_0\vz'(\rho)\displaystyle\bigg(\int_{\Gamma={(\vz(\rho))^{-1}}}\dfrac{u|y|^a\nabla_{Y}\Gamma\cdot\nabla_{Y}\Gamma}{|\nabla_{Y,t}\Gamma|} dH_{n}\bigg)d\rho\\
&=\vz(r) u(X_0,t_0)+\int^r_0 \vz'(\rho)\bigg(\int_{\Omega_\rho(\xi_0)}\mathscr{L}_au\bigg[(\vz(\rho))^{-1}-\Gamma\bigg]dY dt\bigg)d\rho.\tag{6.11}
\end{align*}
In (6.11) we have denoted by $\nabla_{Y,t}\Gamma$ the total gradient of $\Gamma$, i.e., the $(n+1)$-dimensional vector $(\nabla_{Y}\Gamma,D_t\Gamma).$  In addition, applying Federer's co-area formula \cite {Fe}, we obtain
\begin{align*}
\int^r_0\vz'(\rho)\displaystyle\bigg(\int_{\Gamma={(\vz(\rho))^{-1}}}\dfrac{u|y|^a\nabla_{Y}\Gamma\cdot\nabla_{Y}\Gamma}{|\nabla_{Y,t}\Gamma|} dH_{n}\bigg)d\rho=\int_{\Omega_r(\xi_0)}\dfrac{u|y|^a\nabla_{Y}\Gamma\cdot\nabla_{Y}\Gamma}{\Gamma^2}dYdt. \tag{6.12}
\end{align*}
From (6.11) and (6.12), we get
\begin{align*}
&(\vz(r))^{-1}\int_{\Omega_r(\xi_0)}\dfrac{u|y|^a\nabla_{Y}\Gamma\cdot\nabla_{Y}\Gamma}{\Gamma^2}dYdt\\
&= u(X_0,t_0)+(\vz(r))^{-1}\int^r_0 \vz'(\rho)\bigg(\int_{\Omega_\rho(\xi_0)}\mathscr{L}_au\bigg[(\vz(\rho))^{-1}-\Gamma\bigg]dY dt\bigg)d\rho.\tag{6.13}
\end{align*}

In short, we have the following results.
\vspace{0.4cm}
\begin{lemma}
Let $u\in {\cal H}_a$ and $\xi_0=(X_0,t_0)$. Then we have
\begin{align*}
&-\int_{\p\Omega_r(\xi_0)}u(Y,t)|y|^a\nabla_{Y}\Gamma(X_0,t_0;Y,\tau)\cdot\overrightarrow{N}_{Y}dH_{n}\\
&=u(\xi_0)+\int_{\Omega_r(\xi_0)}\mathscr{L}_au\left({(4\pi r)^{-\frac{n+a}{2}}(1+\frac{x_0^2}r)^{-\frac a2}}-\Gamma
\right)dYdt.
\end{align*}
for a.e. $r>0$. Moreover, for every $r>0$, $(6.13)$ holds.
\end{lemma}

\vspace{0.3cm}
\begin{lemma}
Let $u\in {\cal H}_a$ and $\xi_0=(X_0,t_0)$. For  $r>0$ let us define
$$u_r(\xi_0)=(\vz(r))^{-1}\int_{\Omega_r(\xi_0)}\dfrac{u(Y,t)|y|^a|\nabla_{Y}\Gamma(X_0,t_0;Y,t)|^2}{\Gamma^2(X_0,t_0;Y,t)}dYdt.\eqno(6.14)$$
Then
$$\frac{d}{dr}u_r(\xi_0)=(\vz(r))^{-2}\vz'(r)\int_{\Omega_r(\xi_0)}\mathscr{L}_au(Y,t)\ln\bigg[\dfrac{(\vz(r))^{-1}}{\Gamma(X_0,t_0;Y,\tau)}\bigg]dYdt,\eqno(6.15)$$
where $\vz(r), \vz'(r)$ were defined as in $(6.9)$ and $(6.10)$.
\end{lemma}
\begin{proof}
By (6.13), we obtain
\begin{align*}
\frac{d}{dr}u_r(z_0)=&-(\vz(r))^{-2}\vz'(r)\int^r_0 \vz'(\rho)\bigg(\int_{\Omega_\rho(\xi_0)}\mathscr{L}_au\bigg[(\vz(\rho))^{-1}-\Gamma\bigg]dY dt\bigg)d\rho\\
&+(\vz(r))^{-1}\vz'(r)\int_{\Omega_r(\xi_0)}\mathscr{L}_au\Big[(\vz(r))^{-1}-\Gamma\Big]dYdt.
\end{align*}
Let $\vz^{-1}$ be the inverse function of $\vz$.
Applying Fubini's theorem, we conclude
\begin{align*}
&\int^r_0 \vz'(\rho)\int_{\Omega_\rho(\xi_0)}\mathscr{L}_au\big[(\vz(\rho))^{-1}-\Gamma\big]dY dtd\rho\\
&=\int_{\Omega_r(\xi_0)}\mathscr{L}_au\dint_{\vz^{-1}(1/\Gamma)}^r\big[(\vz(\rho))^{-1}-\Gamma\big]dY dtd\vz(\rho)\\
&=\int_{\Omega_r(\xi_0)}\mathscr{L}_au\big[\ln(\vz(\rho))-\Gamma\vz(\rho)\big]^{r}_{\vz^{-1}(1/\Gamma)}dY dt\\
&=\int_{\Omega_r(\xi_0)}\mathscr{L}_au\ln(\vz(r)\Gamma)dY dt-\int_{\Omega_r(\xi_0)}\mathscr{L}_a u(\vz(r)\Gamma-1)dY dt.
\end{align*}
From these, then (6.15) holds.
\end{proof}

As a consequence of Lemma 6.3, we have the following result.

\begin{corollary}
Let $u\in {\cal H}_a$, $\xi_0=(X_0,t_0)$ and $\az>1$.  Suppose that $\mathscr{L}_a u\leq 0 $ in $\oz_{\az^2 r}(\xi_0)$. Then there exists a positive constant $C=C(n,\az,a)$ such that
$$u_{\az^2 r}(\xi_0)-u_{\az r}(\xi_0)\ge  C\l((\vz(\az r))^{-1}-(\vz(\az^2 r))^{-1}\r)\int_{\Omega_{ r}(\xi_0)}(-\mathscr{L}_a u)dYdt, \eqno(6.16)$$
where $\vz(r)$ was defined as in $(6.9)$.
\end{corollary}
\begin{proof}
(6.15) gives
\begin{align*}
u_{\az^2 r}(\xi_0)-u_{\az r}(\xi_0)&=\int^{\az^2 r}_{\az r}\dfrac{d}{d\rho}u_\rho(\xi_0)d\rho\\
&= \int^{\az^2 r}_{\az r}(\vz(\rho))^{-2}\vz'(\rho)\bigg(\int_{\Omega_{\rho}(\xi_0)}\mathscr{L}_au\ln\bigg[\dfrac{(\vz(\rho))^{-1}}{\Gamma(X_0,t_0;Y,t)}\bigg]dYdt\bigg)d\rho\\
&\geq \int_{\Omega_{ r}(\xi_0)}(-\mathscr{L}_a u)\ln(\az^{\frac {n-|a|}2})dYdt(\vz(\rho))^{-1}\big|^{\az^2 r}_{\az r}\\
&\geq C\l((\vz(\az r))^{-1}-(\vz(\az^2 r))^{-1}\r)\int_{\Omega_{ r}(\xi_0)}(-\mathscr{L}_a u)dYdt.
\end{align*}
Thus, (6.16) is proved.
\end{proof}
As a consequence of Lemma 6.2, we have the following result.

\begin{corollary}
If $u$ is $\mathscr{L}_a$-parabolic in an open set containing $\overline{\oz_r(\xi_0)} $, then
 $$u(\xi_0)=(\vz(r))^{-1}\int_{\Omega_r(\xi_0)}\dfrac{u|y|^a\nabla_{Y}\Gamma\cdot\nabla_{Y}\Gamma}{\Gamma^2}dYdt,$$
where $\vz(r)$ was defined as in $(6.9)$.
\end{corollary}

\subsection{A strong form of Harnack inequality}
In this section we prove the strong Harnack inequality. We first consider a special case at center point of $\Omega_r(X_0,t_0)$  with $(X_0,t_0)=(x_0',0,t_0)$, without of loss generality, we can assume $x_0'=0,t_0=0$.

 Next lemma concerning the behavior of $\mathscr{L}_a$-parabolic in the region $Q(2r)$, where
$$Q(r):=\left\{(X,t)\in\rr^{n+1}|-\frac{3r}{4}<t<0,~|X|^2<2(n+a)t\log\left(-\frac{t}{r}\right)\right\}~~(r>0).\eqno(6.17)$$
This is the set $\Omega(r):=\Omega_r(0)$ with the ``lens shaped" region below the line $t=-{3r}/{4}$ removed.

\begin{theorem}
Let $u\geq 0$ is  $\mathscr{L}_a$-parabolic  in $Q(2r)$, $r>0$, and suppose that $u$ is continuous at each point of $\p Q(2r)$, except possibly at $0$. Then there exists a positive constant $C$, depending only on $n,~ a$ and not on $r$, such that
$$\mathop{\dashint}_{|X|^2\leq \frac{3(n+a)r}{4}}u\Big(X,-\frac{3r}{2}\Big)|x|^adX\leq C\inf_{\Omega(\frac{3r}{4})}u,\eqno(6.18)$$
where $\dashint_E$ denote the weighted $(|x|^a)$-average of $E$.
\end{theorem}

\begin{proof} As in \cite {EG},  the hypotheses and conclusion of lemma are unchanged under the mapping $(X,t)\rightarrow \Big(\dfrac{X}{r^{1/2}},\dfrac{t}{r}\Big),$ it is enough to consider the case $r=1$. By homogeneity we may assume that $$\mathop{\dashint}_{|X|^2\leq \frac{3(n+a)}{4}}u\Big(X,-\frac{3}{2}\Big)|x|^adX=1.$$

Step 1: a lower bound in the interior. Consider the cylinder $C'\subset Q(2)$ defined by $$C'=\Big\{(X,t)\in \rr^{n+1}| -\frac{3}{2}<t<-1, ~|X|^2<-3(n+a)\log\Big(\dfrac{3}{4}\Big)\Big\}.$$
In view of the nonnegativity of $u$,   the representation of $u$ in terms of its values on the parabolic boundary of $C'$ and the Green function for $C'$, by Lemma 6.1 and Green formula, there exists a constant $\alpha_1>0$ such that
$$u(X,t)\geq \alpha_1>0~~\text{whenever}~(X,t)\in C'',\eqno(6.19)$$
$$C''=\Big\{(X,t)\in \rr^{n+1}| -\frac{5}{4}<t<-1,~|X|^2<\frac{3(n+a)}{4}\Big\}.$$
In the same way, we can conclude from (6.19) that for some $\alpha_2>0$ we have $$u(X,t)\geq \alpha_2>0~~\text{whenever}~(X,t)\in D:=\Omega(1)\cap\Big\{(X,t)~|~t<-\frac{1}{2e^8}\Big\}.\eqno(6.20)$$

Step 2: a lower bound on the line $X=0$. By Corollary 6.2, for all points $(0,t)$ on the line $L:=\{(0,t)~|-1<t<0\},$ $$u(0,t)=\frac{C_{n,a}}{4^{-\frac{n+a}{2}}}\int_{\Omega\big((0,t),\frac{1}{4}\big)}u(X,s)|x|^a\frac{|X|^2}{4(t-s)^2}dXds.\eqno(6.21)$$
For all $t\in\left(-\frac{1}{2e^8},0\right)$, the $(n+1)$-dimensional measure of $\Omega\big((0,0,t),\frac{1}{4}\big)\cap D$ is bounded away from zero. We can deduce from $(6.20)$ that $$u(0,t)\geq \alpha_3>0~~\text{whenever}~(0,t)\in L.\eqno(6.22)$$

Step  3: a lower bound on the paraboloid $|X|^2<-8(n+a)t.$ Define the truncated paraboloids $$P_1:=\Big\{(X,t)\in \rr^{n+1}|~|X|^2<-8(n+a)t, -\frac{1}{e^8}<t<0\Big\}$$
and $$P_2:=\Big\{(X,t)\in \rr^{n+1}~|X|^2<-16(n+a)t, -\frac{1}{e^8}<t<0\Big\}.$$
Notice that $$P_1\subset P_2\subset \Omega(1).$$
Let $S$ denote a closed cylinder of the form $$S:=\{(X,t)\in \rr^{n+1}|~-c_2\leq t\leq -c_1, |X|^2\leq c_3\},$$
where $c_1, c_2$ and $c_3$ are positive constants chosen so that $$\frac{1}{2e^8}<c_1<c_2<\frac{1}{e^8}$$
and
\begin{align*}
P_1\cap \{(X,t)~|-c_2\leq t\leq -c_1\}&\subset \{(X,t) ~|~ |X|^2\leq c_4, -c_2\leq t\leq -c_1\}\\
&\subset S\\
&\subset P_2\cap\{(X,t) ~|-c_2\leq t\leq -c_1\}
\end{align*}
for some constant $c_4\in (0,c_3).$ Define $$S^+:=S\cap \Big\{(X,t)~|-c_1-\frac{3(c_2-c_1)}{8}\leq t\leq -c_1-\frac{c_2-c_1}{8},~|X|^2\leq c_4\Big\}$$
and $$S^-:=S\cap \Big\{(X,t)~|-c_1-\frac{7(c_2-c_1)}{8}\leq t\leq -c_1-\frac{5(c_2-c_1)}{8}, ~|X|^2\leq c_4\Big\}.$$
According to the Harnack inequality (see \cite{CF}), there exists a positive constant $C$ such that $$\max_{S^-}u\leq C\min_{S^+}u.\eqno(6.23)$$
The constant $C$ depends only on $c_1$, $c_2$, $c_3$, $c_4$, $a$ and $n$, and remains unchanged if we change $S$, $S^+$ and $S^-$ by the parabolic dilation $(X,t)\rightarrow (\lambda X, \lambda^2 t)$, $\lambda >0.$ For $\lambda >0$, define $$S_{\lambda}:=\{(\lambda X,\lambda^2t)~|~(X,t)\in S\}$$
and define $S^+_{\lambda}$ and $S^-_{\lambda}$ similarly. Using (6.23), we obtain $$\max_{S^-_\lambda}u\leq C\min_{S^+_\lambda}u ~\quad\text{whenever}~0<\lambda\leq 1.$$
Since $L\cap S^-_\lambda\neq \emptyset$ for $0<\lambda\leq 1$, by (6.22), this implies $$\alpha_3\leq C\min_{S^+_\lambda}u~\quad\text{whenever}~ 0<\lambda\leq 1.$$
Since each point of $P_1\cap \left\{(X,t)~|-\frac{1}{2e^8}<t<0\right\}$ belongs to $S^+_\lambda$ for some $\lambda \in (0,1]$, we have $$\inf \left\{u(X,t)~|~(X,t)\in \Omega(1), ~|X|^2\leq -8(n+a)t\right\}\geq \alpha_4>0.\eqno(6.24)$$

Step 4: a lower bound on $\Omega (\frac{3}{4}).$ As in \cite{EG} we take
$$\Phi(s):=\arctan(s+16)-\arctan(16),$$ so that for $s\in \rr^+$
\begin{equation}\label{1}
    \left\{
   \begin{array}{ll}
\Phi(0)=0,~~ &0\leq \Phi(s)<\frac{\pi}{2},\tag{6.25}\\
\Phi'(s)>0, &0\leq -\Phi''(s)\leq \frac{1}{8}\Phi'(s).
 \end{array}
 \right.
\end{equation}
Now we set $$v(X,t):=\Phi\Big(\frac{|X|^2}{t}-2(n+a)\log(-t)+2a\ln\pi\Big).\eqno(6.26)$$
We claim that $\mathscr{L}_av\leq 0$ in the open set $$E:=\Omega(1)\backslash \{(X,t)~|~|X|^2\leq -8(n+a)t\}.\eqno(6.27)$$
In fact,
\begin{align*}
\mathscr{L}_av&=\Big(-\frac{|X|^2}{t^2}-\frac{4(n+a)}{t}\Big)|x|^a\Phi'-\frac{4|X|^2}{t^2}|x|^a\Phi''\\
&\leq \Big(-\frac{|X|^2}{2t^2}-\frac{4(n+a)}{t}\Big)|x|^a\Phi'\\
&\leq 0,
\end{align*}
since $\Phi''\geq 0$ and $|X|^2\geq -8(n+a)t$ in $E$, and $\dlim_{x\to 0}|x|^aD_x v(X,t)=0$.
By (6.3) it is obvious that $$v=0~~~\text{on}~\p\Omega(1)\backslash \{(0,0)\}.$$
Also by (6.25) and (6.26) $$v\leq \frac{\pi}{2}~~~\text{on}~\{(X,t)~|~|X|^2=-8(n+a)t\}.$$
Hence (6.27) and the maximum principle imply $$u\geq Cv~~~\text{in}~E,$$
for some positive constant $C$. If $(X,t)\in \Omega (\frac{3}{4}),$ then
\begin{align*}
v(X,t)&=\Phi\left(\frac{|X|^2}{t}-2(n+a)\log(-t)\right)\\
&= \Phi\left(4\ln\l[(4\pi)^{\frac{n+a}2}\Gamma(0,0;X,t)\r]\right)\\
&\ge\Phi\left(2(n+a)\log\left(\frac{4}{3}\right)\right)>0.
\end{align*}
From these and (6.24) we conclude (6.18).
\end{proof}

We now consider another case at center point of $\Omega_r(X_0,t_0)$  with $(X_0,t_0)=(x_0',x_0,t_0)$ and $x_0\not=0$.
 Before stating the theorem we introduce some  notation.  From (2.5), we know that there exist positive constants $C_1, C_2$ depending only on $n, a$ such that for every $(X,t), (Y,\tau)\in\rr^{n+1}$
$$\frac{C_1(1+\frac{y^2}{t-\tau})^{-\frac a2}}{(4\pi(t-\tau))^{\frac{n+a}2}}\exp\l(-\frac {|X-Y|^2}{2(t-\tau)}\r)\le \Gamma(X,t;Y,\tau)\le \frac{C_2(1+\frac{y^2}{t-\tau})^{-\frac a2}}{(4\pi(t-\tau))^{\frac{n+a}2}}\exp\l(-\frac {|X-Y|^2}{6(t-\tau)}\r).\eqno(6.28)
$$
 For $i=1,2$ and $C_i$ as in (2.5) we set for $r>0$
 $$\oz^2_r(X_0,t_0)=\l\{(X,t)\in\rr^{n+1}: \frac{(1+\frac{x_0^2}{t_0-t})^{-\frac a2}}{(4\pi(t_0-t))^{\frac{n+a}2}}\exp\l(-\frac {|X-X_0|^2}{6(t_0-t)}\r)>C_2^{-1} \frac{(1+\frac{x_0^2}{r})^{-\frac a2}}{(4\pi r)^{\frac{n+a}2}}\r\},\eqno(6.29)$$ and
  $$\oz^1_r(X_0,t_0)=\l\{(X,t)\in\rr^{n+1}: \frac{(1+\frac{x_0^2}{t_0-t})^{-\frac a2}}{(4\pi(t_0-t))^{\frac{n+a}2}}\exp\l(-\frac {|X-X_0|^2}{2(t_0-t)}\r)>C_1^{-1} \frac{(1+\frac{x_0^2}{r})^{-\frac a2}}{(4\pi r)^{\frac{n+a}2}}\r\}.\eqno(6.30)$$
Without of loss generality, we can assume $t_0=0$. Let $\xi_0=(X_0,0)$ and also $\Psi_r^i(\xi_0)=\pz\oz_r^i(\xi_0)$. (6.28)-(6.30) imply
 $$\oz^1_r(\xi_0)\subset\oz(\xi_0)\subset \oz^2_r(\xi_0).$$
Let $\az_1=2,\az_2=6$.  The heat spheres $\Psi_r^i, i=1,2,$ are
 $$|X-X_0|^2=R^i_r(t)=\az_i t\l(\frac {n}2\ln(-C_i^{-\frac 2n}r^{-1}t)+\frac a2\ln\l(\frac{-t+|x_0|^2}{r+|x_0|^2}\r)\r)$$
for $-t^i_{x_0}\le t<0$. The function $R^i_r(t)$ vanishes at $t=0$ and at $-t^i_{x_0}$.
It is easy to see that $r\le t^2_{x_0}\le C_2^{\frac 2{n+\bar a}}r$ and $C_1^{\frac 2{n+\bar a}}r\le t^1_{x_0}\le r$, where $\bar a=\min\{0,a\}$.
Therefore, if $r<|x_0|^2$, then there exist positive structure constants $\widehat{C}_1$ and $\widehat{C}_2$ depending only on $n,a$ such that
$$\oz^2_r(\xi_0)\subset \widehat{\oz}^2_r(\xi_0)=\l\{(X,t)\in\rr^{n+1}: \exp\l(\frac {|X-X_0|^2}{6t}\r)> \frac{\widehat{C}_2^{-1} }{(4\pi r)^{\frac{n}2}}\r\}\eqno(6.31)$$
and
$$\oz^1_r(\xi_0)\supset \widehat{\oz}^1_r(\xi_0)=\l\{(X,t)\in\rr^{n+1}: \exp\l(\frac {|X-X_0|^2}{2t}\r)> \frac{\widehat{C}_1^{-1} }{(4\pi r)^{\frac{n}2}}\r\}.\eqno(6.32)$$
The heat spheres $\widehat{\Psi}_r^i=\pz\widehat{\oz}^2_r(\xi_0), i=1,2,$ are
 $$|X-X_0|^2=\widehat{R^i_r}(t)=\frac {\az_i n}2t\ln(-\widehat{C}_i^{-\frac 2n}r^{-1}t)\eqno(6.33)$$
for $-\widehat{C}_i^{\frac 2n}r\le t<0$. The function $\widehat{R}^i_r(t)$ vanishes at $t=0$ and at $-\widehat{C}_i^{\frac 2n}r$.

We set $\eta_i=-\widehat{C}_i^{\frac 2n},\ i=1, 2$. Next we choose and fix a $\sz>0$ such that $\sz<2\eta_1\eta_2^{-1}$, and we let $\eta=(\eta_2\sz+2\eta_1)/2$, so that
$\eta_2\sz<\eta<2\eta_1$. Note that
$$\inf\{t|\ (X,t)\in \oz_{\sz r}\}\ge -\sz\eta_2 r>\eta r,\eqno(6.34)$$
since by (6.31) $\oz_{\sz r}\subset \widehat{\oz}_{\sz r}^2$ and the lowest time level of $\widehat{\oz}_{\sz r}^2$ is $-\sz\eta_2r$. For $r>0$ we set
$$Q_r(\xi_0)=\{(X,t)\in\oz_r(\xi_0)| t>-\eta r/2\}.$$
In what follows we will use the set $Q_{2r}(\xi_0)$. This is the parabolic ball $\oz_{2r}(\xi_0)$ with  the part below the hyperplane $\{t=-\eta r\}$ removed.
Because of (6.34) for every $r>0$ we have $\oz_{\sz r}\subset Q_{2r}$.

From (6.33), we see that the $\oz_r(\xi_0)$ is contained in the parabolic cylinder
$$C_r=\{(X,t)\in\rr^{n+1}: |X-X_0|^2\le 2ne^{-1} \widehat{C}_2^{\frac 2n}r, -\widehat{C}_2^{\frac 2n}\le t\le 0\}.\eqno(6.35)$$
Let $$I_r(\xi_0)=\{(X,t)\in\rr^{n+1}:\ |X-X_0|^2\le R^1_r(-\eta r),\ t=-\eta r\},$$
the two sets $I_r$ and $Q_{\sz r}$ are detached and there is a time-lag between then. We can  state the following.
\begin{theorem}
Let $u\geq 0$ is $\mathscr{L}_a$-parabolic  in $Q_{2r}(\xi_0)$ with $\xi_0=(x'_0,x_0,0)$ and $x_0\not=0$, and suppose that $u$ is continuous at each point of $\pz Q_{2r}(\xi_0)$, except possibly at $\xi_0$. Then there exist a positive constant $C$  depending only on $n$ and $a$, such that
$$\mathop{\dashint}_{I_r(\xi_0)}u\Big(X,-\eta r\Big)|x|^adX\leq C\inf_{\Omega(\frac{3\sz r}{4})(\xi_0)}u\eqno(6.36)$$
if $0<r<c_0|x_0|^2$ with $c_0=(4n \widehat{C}_2^{\frac 2n})^{-11}$.
\end{theorem}

\begin{proof}
By homogeneity we may suppose, and we do so that
$$\mathop{\dashint}_{I_r(\xi_0)}u\Big(X,-\eta r\Big)|x|^adX=1.$$
We thus want to show that there exists $\Lambda>0$ such that
$$\dinf_{\oz_{\sz r}(\xi_0)}u\ge \Lambda^{-1}.\eqno(6.37)$$
As in the proof of Theorem 1.4 in \cite{GL}, from the standard Harnack inequality for $\mathscr{L}_a$, for every $\dz>0$ we infer the existence of constant $\ez=\ez(\dz)>0$, independent of $r$, such that
$$\dinf_{\oz_{\sz r}(\xi_0)\cap P_{\dz}}u\ge \ez,\eqno(6.38)$$
where $P_{\dz}$ denotes one part of a paraboloid, that is,
$$P_{\dz}=\{(X,t)\in\rr^{n+1}: t\le -\dz |X-X_0|^2\}.$$
It remains to prove that
$$\dinf_{W_{\dz,r}}u\ge C>0$$
for a suitable constant $C$, where $W_{\dz,r}=\oz_{\sz r}(\xi_0)\setminus P_{\dz}$. Let $\Phi$ be defined as (6.25).
Now we set
$$v(X,t)=\Phi\l(\ln\l[(4\pi\sz r)^{\frac {n+a}2}\l(1+\frac{x_0^2}{\sz r}\r)^{\frac a2}\Gamma(\xi_0; X,t)\r]\r),$$
where $\Gamma(\xi_0; X,t)$ is the fundamental solution of
$\mathscr{L}_a^*=-(D_t+\Delta_X+\frac axD_x)$ with pole at $\xi_0=(X_0,0)\in \rr^{n+1}.$ It is obvious that
$$v(X,t)=0\quad {\rm for\ every}\ (X,t)\in \Psi_{\sz r}(\xi_0)\setminus\{\xi_0\}.\eqno(6.39)$$
Also by (6.25), we have
$$v(X,t)\le \pi/2\quad{\rm for}\ (X,t)\in\overline{\oz}_{\sz r}(\xi_0)\cap P_{\dz}.\eqno(6.40)$$
We now claim that
$$\mathscr{L}_a v(X,t)\le 0\quad {\rm for}\ (X,t)\in W_{\dz,r}.$$
Let us  assume for a moment the claim is true. Then by (6.38)-(6.40), the fact $u\ge 0$ in $\Psi_{\sz r}(\xi_0)\setminus\{\xi_0\}$, and the maximum principle we would infer that
$u\ge 2\ez/\pi, \ez$ as in (6.38). Since $\Gamma(\xi_0;X,t)>(3\pi\sz r)^{-\frac {n+a}2}(1+\frac{x_0^2}{(3\sz r)/4})^{-\frac a2}$ on $\oz_{(3/4)\sz r}(\xi_0)\setminus P_{\dz}$,
(6.40) would give
$$\dinf_{\oz_{(3/4)\sz r}(\xi_0)\setminus P_{\dz}}u\ge c\dinf_{\oz_{(3/4)\sz r}(\xi_0)\setminus P_{\dz}}v\ge C\Phi\l(\frac {n-|a|}2\ln\l(\frac 43\r)\r),$$
which, together with (6.38), implies (6.37).

It remains to prove this claim, which is equivalent to proving that if we set
$$w(X,t)=\Phi\l(\ln\l[(4\pi\sz r)^{\frac {n+a}2}\l(1+\frac{x_0^2}{\sz r}\r)^{\frac a2}\Gamma( X,t;\xi_0)\r]\r),$$
then
$$ -\mathscr{L}_a^*w\ge 0 \ {\rm in}\ W^*_{\dz,r}, \eqno(6.41)$$
where $W^*_{\dz,r}$ is the image of the $W_{\dz,r}$ under the time-reflection $(X,t)\to(X,-t)$.
Now a computation yields
\begin{align*}
-\mathscr{L}^*_a w&=D_t w+\Delta_X w+\frac ax D_x w\\\tag{6.42}
&=\Phi'\frac{D_t\Gamma}{\Gamma}+\Phi'\frac{\triangle_X\Gamma}{\Gamma}+(\Phi''-\Phi')\frac{\nabla_X\Gamma\nabla_X\Gamma}{\Gamma^2}+\Phi'\frac ax\frac{D_x\Gamma}\Gamma\\
&=2\Phi'\frac{D_t\Gamma}{\Gamma}+(\Phi''-\Phi')\frac{\nabla_X\Gamma\nabla_X\Gamma}{\Gamma^2},
\end{align*}
where the last equality we have used the fact that $\mathscr{L}_a\Gamma(X,t;\xi_0)=0$ for $(X,t)\not=\xi_0$.
Because of (6.25) $\Phi'-\Phi''\le 3\Phi'/2$. Therefore (6.41) will be true if
$$\frac{\nabla_X\Gamma\nabla_X\Gamma}{\Gamma^2}\le \frac 43 \frac {D_t\Gamma}{\Gamma} \ \ \  {\rm in}\ W^*_{\dz,r}, \eqno(6.43)$$
if $0<r<c_0|x_0|^2$ with $c_0=(4n \widehat{C}_2^{\frac 2n})^{-1}$ and $\dz$ is enough small.

In fact, using (2.8), by direct computation, we get
$$\frac{\nabla_X\Gamma\nabla_X\Gamma}{\Gamma^2}=\frac{|X-X_0|^2}{4t^2}+\frac{x_0^2}{t^2}\frac{F'^2}{F^2}-\frac{(x-x_0)x_0}{t^2}\frac{F'}F$$
and
$$ \frac {D_t\Gamma}{\Gamma}=\l(-\frac {n+a}{2t}+\frac{|X-X_0|^2}{4t^2}\r)-\frac{xx_0}{t^2}\frac{F'}F,$$
where $F'(z)=D_zF(z)$  for  $z=\frac{xx_0}t>0$.

Hence,
\begin{align*}
\dfrac 43 \dfrac {D_t\Gamma}{\Gamma}-\dfrac{\nabla_X\Gamma\nabla_X\Gamma}{\Gamma^2}&=
 \dfrac 43\l(-\dfrac {n+a}{2t}+\frac{|X-X_0|^2}{4t^2}\r)-\dfrac 43\dfrac{xx_0}{t^2}\frac{F'}F\\
 &\quad-\l(\dfrac{|X-X_0|^2}{4t^2}+\frac{x_0^2}{t^2}\frac{F'^2}{F^2}-\dfrac{(x-x_0)x_0}{t^2}\frac{F'}F\r)\\
 &=-\dfrac {2(n+a)}{3t}+\frac 13\frac{|X-X_0|^2}{t^2}-\l(\dfrac{x_0^2}{t^2}\l(\frac{F'}F+\frac{F'^2}{F^2}\r)-\dfrac{xx_0}{3t^2}\frac{F'}F\r).
\end{align*}

 Let $z=\frac{xx_0}t$, if
$0<r<c_0|x_0|^2$ with $c_0=(4n \widehat{C}_2^{\frac 2n})^{-11}$,  then $z=\frac{xx_0}t\ge (4n \widehat{C}_2^{\frac 2n})^{10}>>1 $ and $|x-x_0|<|x_0|/2$, and noticing that

 $$\frac{F'(z)}{F(z)}\sim C_a z^{-1},\quad {\rm if}\ z\to \fz.$$
Then,
$$\dfrac{x_0^2}{t^2}\l(\l|\frac{F'}F\r|+\l|\frac{F'^2}{F^2}\r|\r)+\dfrac{xx_0}{3t^2}\l|\frac{F'}F\r|\le C_3/t,$$
 where $C_3$ is a constant depending only on $n$ and $a$.

 From these, we then get
 \begin{align*}
\dfrac 43 \dfrac {D_t\Gamma}{\Gamma}-\dfrac{\nabla_X\Gamma\nabla_X\Gamma}{\Gamma^2}&\ge -\dfrac {2(n+a)}{3t}+\frac 13\frac{|X-X_0|^2}{t^2}-\dfrac{C_3}t\\
&=\dfrac 13\dfrac{|X-X_0|^2}{t^2}-\dfrac {2(n+a)+3C_3}{3t}\ge 0,
\end{align*}
 if $|X-X_0|^2\ge t/\dz$ with $\dz\le 1/(2(n+a)+3C_3)$.

 Thus, (6.43) is proved.
\end{proof}

\section{Smoothing of $\mathscr{L}_a$-superparabolic functions }
In this section, we generalize some mean value formulas relative to $\mathscr{L}_a$. We define a new function in $\rn\times\rr^m\times\rr$ by  setting for $(X,t)\in \rr^{n+1}$
$$\bar u(X,Z,t)=u(X,t),\quad Z\in \rr^m,$$ then we have
$$(\mathscr{L}_a+\Delta_Y)\bar u(X,Z,t)=\mathscr{L}_a u(X,t).$$
Therefore, we can apply to the function $\bar u$ the representation formulas established in section 6. Before stating the results we need to introduce some notation. If
$\xi=(X,t)=(x',x,t),\zeta=(Y,s)\in \rr^{n+1}$ and $r>0$, we set
$$E(\xi;\zeta)=\frac{|y|^a\nabla_Y\Gamma(\xi;\zeta)\nabla_Y\Gamma(\xi;\zeta)}{\Gamma^2(\xi;\zeta)},\quad E_r(\xi;\zeta)=(4\pi r)^{-\frac{n+a}{2}}(1+\frac{x^2}r)^{-\frac a2}E(\xi;\zeta).$$
Next we define for a fixed $m\in\nn$
$$\Phi(\xi;\zeta)=(4\pi(t-s))^{-\frac m2}\Gamma(\xi;\zeta),$$
$$R_r(\xi;\zeta)=\l(4(t-s)\ln\l[(4\pi r)^{\frac {n+m+a}2}(1+\frac{x^2}r)^{\frac a2}\Psi(\xi;\zeta)\r]\r)^{\frac 12}$$
and
$$\oz_r^m(\xi)=\l\{\zeta\in\rr^{n+1}: \Phi(\xi;\zeta)>(4\pi r)^{-\frac {n+m+a}2}(1+\frac{x^2}r)^{-\frac a2}\r\}.$$
\begin{lemma}
Let $u\in {\cal H}_a$ and let $\xi=(X,t)=(x',x,t)\in\rr^{n+1}$. Then for every $r>0$ we have
\begin{align*}
&\dfrac{\omega_m}{\vz_m(r)}\dint_{\oz_r^m(\xi)}u(\zeta)R^m_r(\xi,\zeta)\l[E(\xi,\zeta)+\frac m{m+2}\frac{R^2_r(\xi,\zeta)}{4(t-s)^2}\r]d\zeta\\
&=u(\xi)+\dfrac {2\omega_m}{m+2}\dint_0^r(\vz_m(l))^{-2}\vz_m'(l)\dint_{\oz_r^m(\xi)}(-\mathscr{L}_a u(\zeta))\frac{R^{m+2}_l(\xi,\zeta)}{4(t-s)}d\zeta dl,
\end{align*}
where $\omega_m$ denotes the measure of the unit ball in $\rn$, and where $\vz_m(r)$ was defined by
$$\vz_m(r)=(4\pi r)^{\frac {n+m+a}2}(1+\frac{x^2}r)^{\frac a2}.$$
\end{lemma}
\begin{proof}
Ref the proof of Theorem 3.1 in \cite{GL}.
\end{proof}
For convenience, we set for $m\in\nn$ and $\xi,\zeta\in\rr^{n+1}$
$$E_r^{(m)}(\xi,\zeta)=\dfrac{\omega_m}{\vz_m(r)}R^m_r(\xi,\zeta)\l[E(\xi,\zeta)+\frac m{m+2}\frac{R^2_r(\xi,\zeta)}{4(t-s)^2}\r]$$
and
$$u_r^{(m)}(\xi)=\dint_{\oz_r^m(\xi)}u(\zeta)E_r^{(m)}(\xi,\zeta)d\zeta.\eqno(7.1)$$
Obvious, if $m=0$ in (7.1) reduce to (6.14).

In what follows, we choose and fix a function $\vz\in C_0^\fz(\rr^+)$ such that $\vz\ge 0, \supp\vz\subset (1,2)$ and $\int_0^\fz\vz(r)dr=1$. For $m\in\nn$ and
$\xi\in\rr^{n+1}$ we define
$$J_r^{(m)}(u)(\xi)=\dint_0^\fz u^{(m)}_l(\xi)\vz(\frac lr)\frac{dl}r,\eqno(7.2)$$
where $u\in L^\fz_{\loc}(\rr^{n+1})$ and $u^{(m)}_r$ is as in (7.1).  By (7.1) and (7.2), we get
$$J_r^{(m)}(u)(\xi)=\dint_\rn u(\zeta)M_r^{(m)}(\xi,\zeta)d\zeta,\eqno(7.3)$$
where we have set for $\xi=(X,t),\zeta=(Z,s)$
$$M_r^{(m)}(\xi,\zeta)=\dint_{\vz_m^{-1}(1/\Phi(\xi,\eta))}^\fz E_r^{(m)}(\xi,\zeta)\vz(\frac lr)\frac{dl}r\eqno(7.4)$$
if $t>s$, whereas $M_r^{(m)}(\xi,\zeta)=0$ for $t\le s$, and $\vz_m^{-1}$ is the  inverse function of $\vz_m$.

\begin{proposition}
Let $u: \rr^{n+1}\to\rr$ be a semicontinuous function, and let $m\in \nn\cup\{0\}$. The following statements are equivalent:
\begin{enumerate}
\item[$\mathrm{(i)}$] $u$ is $\mathscr{L}_a$-superparabolic in $\rr^{n+1}$;
\item[$\mathrm{(ii)}$] For every $\xi\in\rr^{n+1}$ and $r>0$
$$u(\xi)\ge u_r^{(m)}(\xi).$$
\end{enumerate}
\end{proposition}
To prove Proposition 7.1 we will need the following result.

\begin{lemma}
If $\xi=(x',x,t)\in \rr^{n+1}, r>0$, and $\oz_r(\xi)$ is the $\mathscr{L}_a$-parabolic ball $(6.5)$, then at most two  points of
$\Psi_r(\xi)$ are not $\mathscr{L}_a$-regular.
\end{lemma}

\begin{proof}
We consider three cases.

Case 1: $x=0$. Then
$$\Gamma(x',0,t;y',y,\tau)=C_{n,a}(t-\tau)^{-\frac {n+a}2}\exp\l({-\frac{|X-Y|^2}{4(t-\tau)}}\r).$$
From this, it is easy to see that
$$|\nabla_Y\Gamma(x',0,t;y',y,\tau)|\not=0\quad {\rm or}\quad D_\tau \Gamma(x',0,t;y',y,\tau)>0.$$
Therefore, every point of $\Psi_r(\xi)\setminus\{\xi\}$ satisfies the exterior ball condition (see Lemma 4.7), so is $\mathscr{L}_a$-regular.

Case 2: $x\not=0$ and $y\not=0$. If $x'\not=y'$, then $|\nabla_Y\Gamma(x',x,t;y',y,\tau)|\not=0$, so such points
$(Y,\tau)$ satisfy the exterior ball condition,  are also $\mathscr{L}_a$-regular.

We now claim that if $x'=y'$ and $|\nabla_Y\Gamma(x',0,t;x',y,\tau)|=0$, then $x=y$.

In fact, if $|\nabla_Y\Gamma(x',0,t;x',y,\tau)|=0$, then
$$\nabla_Y\Gamma(x',0,t;x',y,\tau)=D_y\l( C_{n,a}(t-\tau)^{-\frac {n+a}2}\exp({-\frac{|x-y|^2}{4(t-\tau)}})F(\frac{xy}{t-\tau})\r)=0.\eqno(7.5)$$
Let $z=\frac{xy}{t-\tau}$, by (7.5), we  then have
$$\frac{x-y}{2(t-\tau)} F(z)+\frac x{t-\tau} F'(z)=0.\eqno(7.6)$$
In addition, $F(z)$ also satisfies (2.2), that is,
$$zF''+(z+a)F'+\frac a{2}F=0.\eqno(7.7)$$
Notice that  $$u(x,t;y)=t^{-\frac {a+1}2}e^{-\frac{(x-y)^2}{4(t-\tau)}}F\l(\frac{xy}{t-\tau}\r)$$
is a solution of
$$D_\tau u+u_{yy}+\frac ay u_y=0\quad {\rm for }\ y\not= 0, ~t>\tau.\eqno(7.8)$$
From (7.5), we know that $u_y=0$. Then (7.8) is turn into
$$D_\tau u=-u_{yy}\quad {\rm for }\ y\not= 0, ~t>\tau.\eqno(7.9)$$

Therefore, from (7.9), we have
$$\l(-\frac{1+a}{2(t-\tau)}+\frac{|x-y|^2}{4(t-\tau)^2}\r)F-\frac{xy}{(t-\tau)^2}=\l(-\frac{1}{2(t-\tau)}+\frac{|x-y|^2}{4(t-\tau)^2}\r)F+\frac {x^2F''}{(t-\tau)^2}.\eqno(7.10)$$
Taking (7.6) and (7.7) in (7.10), we deduce $x=y$. This claim is proved.

From this claim, we know that if  $X\not=Y$ and $y\not=0$, then
 $|\nabla_Y\Gamma(X,t;Y,\tau)|\not=0$, so such points
$(Y,\tau)$ satisfy the exterior ball condition,  are also $\mathscr{L}_a$-regular.

In addition, if  $X=Y$, it is easy to see that
$$D _\tau \Gamma(X,t;X,\tau)>0.$$
Therefore, the north pole   $(X,\tau) $ satisfies the exterior ball condition, so is $\mathscr{L}_a$-regular.

In short, every point of $\Psi_r(\xi)\setminus\{\xi\}$  is $\mathscr{L}_a$-regular in this case.

Case 3: $x\not=0$ and $y=0$.

Subcase 1: If $x'\not=y'$, then $|\nabla_Y\Gamma(x',x,t;y',0,\tau)|\not=0$, so such points
$(Y,\tau)$ satisfy the exterior ball condition,  are also $\mathscr{L}_a$-regular.

Subcase 2: If $x'=y', y=0$ and $a\ge 0$, then $|\nabla_Y\Gamma(x',x,t;x',0,\tau)|\not=0$, so such points
$(Y,\tau)$ satisfy the exterior ball condition,  are also $\mathscr{L}_a$-regular.

Subcase 3: If $x'=y', y=0$ and $a< 0$, it is obvious that
$$ |\nabla_Y\Gamma(x',x,t;x',0,\tau)|=0.$$
By direct computation, $\tau=t_M:=t-\frac{x^2}{2(n+a)}$ satisfies the following
$$\ D_\tau\Gamma(x',x,t;x',0,\tau)=0.\eqno(7.11)$$
Obvious, from (7.11), we know that
$\tau=t_M$ is a unique maximum point of the function $\Gamma(x',x,t;x',0,\tau)$ about the variable $\tau$.

If the point $(x',0,t_M)$ is out of $\Psi_r(\xi)$, then, we know that
the straight line $(x',0,s)$ is also out of $\Psi_r(\xi)$.

If the point $(x',0,t_M)$ is on $\Psi_r(\xi)$, then, it is easy to see  that
 $\{y=0\}\cap \Psi_r(\xi)=(x',0,t_M)$. So $(x',0,t_M)$ satisfies the exterior ball condition,  is also $\mathscr{L}_a$-regular.

If the point $(x',0,t_M)$ is  in $\oz_r(\xi)$, then, we know that
the straight line $(x',0,s)$ intersects $\Psi_r(\xi)$ two points, that is, $(x',0,t_1)$ and $(x',0,t_2)$ with $t_1<t_M<t_2$.
Obvious, $(x',0,t_1)$ satisfies the exterior ball condition,  is also $\mathscr{L}_a$-regular. But, $(x',0,t_2)$ may be not $\mathscr{L}_a$-regular.

In short, at most two  points of
$\Psi_r(\xi)$ are not $\mathscr{L}_a$-regular, one point is $\xi$, another point is $(x',0,t_2)$ in this case.
\end{proof}

\begin{remark} Lemma $7.1$ can be improved provided that $r<c_0|x|^2$ $($ $c_0$ defined in Theorem $6.2$$)$. More precisely,
if $r<c_0|x|^2$, then $|y-x|<|x|/2$, so $\{y=0\}\cap \Psi_r(\xi)=\emptyset$. Thus, every point of $\Psi_r(\xi)\setminus\{\xi\}$ is $\mathscr{L}_a$-regular.
\end{remark}
Let us continue to prove Proposition 7.1.

\begin{proof}
Ref the proof of Proposition 5.1 in \cite{GL} by using Lemma 7.2.
\end{proof}

\begin{proposition}
Let $u=P_{\tilde{\mu}}$ be a
$\mathscr{L}_a$-equilibrium potential on $\rr^{n+1}$. Then
\begin{enumerate}
\item[$\mathrm{(i)}$] $J^{(m)}_\rho u\le J^{(m)}_ru$ for every $r<\rho$.
\item[$\mathrm{(ii)}$] $J^{(m)}_ru\uparrow u$ as $r\to 0^+$.
\item[$\mathrm{(iii)}$] $J^{(m)}_ru$ is $\mathscr{L}_a$-superparabolic in $\rr^{n+1}$ for every $r>0$
\end{enumerate}
\end{proposition}
\begin{proof}

Ref the proof of Corollary 5.2 in \cite{GL} by using Proposition 7.1.
\end{proof}

Finally, we give the following result about regularizing $\mathscr{L}_a$-superparabolic function.
\begin{theorem}
Let $u=P_{\tilde{\mu}}$ be a
$\mathscr{L}_a$-equilibrium potential on $\rr^{n+1}$. Then there exists  a sequence of functions $(u_j)_{j\in\nn}$ such that
\begin{enumerate}
\item[$\mathrm{(i)}$] $u_j\in {\cal H}_a, \ j\in\nn$,
\item[$\mathrm{(ii)}$] $u_j$ is $\mathscr{L}_a$-superparabolic in $\rr^{n+1}$ for  $j\in\nn$,
\item[$\mathrm{(iii)}$] $u_j\le u_{j+1},\ j\in\nn$,
\item[$\mathrm{(iv)}$] $u_j(\xi)\to u(\xi)$ as $j\to\fz$ for every $\xi\in\rr^{n+1}$,
\item[$\mathrm{(v)}$] If for a given compact $K\subset \rr^{n+1}$, $u$ is $\mathscr{L}_a$-parabolic in $\rr^{n+1}\setminus K$, then for every $\dz>0$ there exists a
$j_0\in\nn$ such that $u_j$ is $\mathscr{L}_a$-parabolic in $K_\dz$ for every $j\ge j_0$, where
$$K_\dz=\{\xi\in\rr^{n+1}: {\rm dist}(\xi,K)\ge \dz\}.$$
\end{enumerate}

\end{theorem}
\begin{proof}
For (i),  applying (7.3) and (7.4), by direct computation, we can obtain the desired result. We omit the details here,
ref the proof of (i) in  Theorem 6.1 in \cite{GL}.

For (ii)-(v), it is simple. In fact, let $(r_j)_{j\in\nn}$ be a consequence of positive numbers such that $r_j\to 0$ as $j\to\fz$ and $r_{j+1}\le r_{j}\le 1$ for
$j\in\nn$. For every $j\in\nn$ we set
$$u_j=J_{r_j}^{(m)}(u),$$
where $m$ is a positive integer and $J_r^{(m)}$ is defined as in (7.3). By Proposition 7.2, it immediately follows that the sequence $(u_j)_{j\in\nn}$ verifies (ii), (iii) and (iv)
above. We now prove (v). By (3.35), we can find a positive constant $C$ depending only on $n$ and $a$, such that
$$\oz_r^{(m)}(\xi)\subset\{(\xi,\zeta)\in\rr^{n+1}: |X-Y|^2\le Cr, 0<t-s<Cr\}$$
for every $\xi=(X,t)\in\rr^{n+1}$ and $r>0$. Thus, for given $\dz>0$ there exists a $j_0\in\nn$ such that
$$\oz_r^{(m)}(\xi)\subset\rr^{n+1}\setminus K$$
for every $l\le 2r_j, j\ge j_0$ and $\xi\in\rr^{n+1}$  such that ${\rm dist}(\xi,K)\ge \dz$. For such $\xi$'s and $l$'s, by Lemma 7.1, we then have
$$u^{(m)}_l(\xi)=u(\xi).$$
By (7.2), note that $\supp\vz\subset(1,2)$, we obtain
$$u_j(\xi)=u(\xi)$$
for every $j\ge j_0$ and $\xi\in\rr^{n+1}$ with ${\rm dist}(\xi,K)\ge \dz$. Thus, (v) is proved.
\end{proof}

\section{Proof of Theorem 1.1}
In what follows we will proceed along the lines of the proof of Wiener's criterion in \cite{EG, GL}.
\begin{proof}
Necessity. This is the easy part of the proof. Assume for any $\lz\in(0,1)$
 $$\sum^\infty_{k=1}\lz^{{-k(n+a)}/{2}}\l(1+\frac{x_0^2}{\lz^k}\r)^{-\frac a2}\mathrm{cap}(\Omega^c\cap A(\lz^k))<+\infty \eqno(8.1)$$
{for}~$A(\lz^k):=A(\xi_0,\lz^k)$.
We will show that (8.1) implies $\xi_0$ is not a $\mathscr{L}_a$-regular point for $\Omega.$

According to Proposition 5.3 it suffices to show $$V_{\Omega^c\cap C_r}(\xi_0)<1~~\text{for~some}~r>0.\eqno(8.2)$$
Hence fix some $r>0$, recall (5.5), and define
$$
    \left\{
   \begin{array}{ll}
B_k(r):=A(\lz^k)\cap \Omega^c\cap C_r\qquad\qquad k=1,2,\ldots,\nonumber\\
B_0(r):=\overline{(\Omega^c\cap C_r)\backslash\displaystyle\bigcup^\infty_{k=1}B_k(r)}.
 \end{array}
 \right.
$$
Note that $$\Omega^c\cap C_r=\bigcup^\infty_{k=0}B_k(r)\subset C_r.$$

Let $\tilde{\mu}$ denote the $\mathscr{L}_a$-equilibrium measure for $\Omega^c\cap C_r$, define
$$\nu_k:=\tilde{\mu}|_{B_k(r)}\qquad k=0,1,2,\ldots,$$
and let $\nu_k'$ denote the $\mathscr{L}_a$-equilibrium measure for $B_k(r)$. According to Proposition 5.2 (i), $\nu_k\in M^+(B_k(r))$ and $P_{\nu_k}\leq P_{\tilde{\mu}}\leq 1,$ then $$P_{\nu_k}\leq P_{\nu_k'}\qquad k=0,1,2,\ldots.$$
Hence $$V_{\Omega^c\cap C_r}\leq \sum^\infty_{k=0}P_{\nu_k}\leq \sum^\infty_{k=0}P_{}\nu_k'~~~\text{in}~\rr^{n+1}.\eqno(8.3)$$
Applying (1.3) and (1.6), we deduce
\begin{align*}
P_{\nu_k'}(\xi_0)&=\int_{\rr^{n+1}}F(\xi_0;\zeta)d\nu_k'(\zeta)=\int_{B_k(r)}F(\xi_0;\zeta)d\nu_k'(\zeta)\\
&\leq (4\pi\lz^{k+1})^{{-(n+a)}/{2}}\l(1+\frac{x_0^2}{\lz^{k+1}}\r)^{-\frac a2}\nu_k'(B_k(r))\qquad\quad k=0,1,2,\ldots.
\end{align*}
In addition, by (8.1) there exists a $p\in\nn$ such that
$$\sum^\infty_{k=p}\lz^{{-k(n+a)}/{2}}\l(1+\frac{x_0^2}{\lz^{k+1}}\r)^{-\frac a2}\mathrm{cap}(B_k(r))<\lz^{(n+a)/2}. \eqno(8.4)$$
Now if $c=\lz^p$ we have $C_r\setminus D_{r,c}\subset\cup_{k=p}^\fz B_k(r)$ ($D_{r,c}$ defined as in Corollary 5.1).
Therefore (8.3) and (8.4) imply
\begin{align*}
V_{\Omega^c\cap C_r}(\xi_0)\le \dsum_{k=p}^\fz {\rm cap}(B_k(r)) (4\pi\lz^{k+1})^{{-(n+a)}/{2}}\l(1+\frac{x_0^2}{\lz^{k+1}}\r)^{-\frac a2}
<(4\pi)^{-\frac {n+a}2}<1.
\end{align*}
By Proposition 5.3  we conclude that $\xi_0$ is $\mathscr{L}_a$-irregular for $\oz_{r,c}$. By Corollary  5.1 it follows that $\xi_0$ is $\mathscr{L}_a$-irregular for $\oz$.

Sufficiency. We assume that (1.7) holds. To begin with, since by (1.7) at least one of the four series must diverge:
$$\sum^{\infty}_{k=1}\lz^{-(4k+i)(n+a)/2}\l(1+\frac{x_0^2}{\lz^{k+i}}\r)^{-\frac a2}\mathrm{cap}(\Omega^c\cap A(\lz^{(4k+i)}))~\quad i=0,1,2,3,$$ we may assume that $$\sum^\infty_{k=1}\lz^{-2(n+a)k}\l(1+\frac{x_0^2}{\lz^{k}}\r)^{-\frac a2}\mathrm{cap}(\Omega^c\cap A(\lz^{4k}))=+\infty.\eqno(8.5)$$
Since ${\rm cap}(C_r)\to 0$ as $r\to 0^+$ by (5.8), for every $k\in\nn$ we can choose $r_k>0$ such that replacing $ A(\lz^{4k})$ by $ A(\lz^{4k})\cap(\rr^{n+1}\setminus C_{r_k})=
E_{4k}$, (8.5) still holds. We will show $\xi_0$ is $\mathscr{L}_a$-regular
$$\widehat{\oz}=\dot{C}_1\setminus\l(\cup_{k=1}^\fz E_{4k}\cup\{\xi_0\}\r)=\dot{C}_1\setminus E.$$
Since $\oz\cap\dot{ C}_1\subset\widehat{\oz}$, it will sufficient to conclude that $\xi_0$ is $\mathscr{L}_a$-regular for $\oz$. We denote by $V$ the equilibrium of $E$ and set
$$W=1-V.$$

Claim. $W(\xi_0)=0.$
It is clear that by Proposition 5.3 the claim implies the $\mathscr{L}_a$-regularity of $\xi_0$.
For any $\varepsilon>0$ and any $k\in \mathbb{N}$ we choose a compact set $E^\varepsilon_{4k}$ such that
$$E_{4k}\subset \dot{E}^\varepsilon_{4k},~~E^\varepsilon_{4k}\subset  E^{\varepsilon'}_{4k}\quad \text{if}~ 0<\varepsilon<\varepsilon',\quad \bigcap_{\varepsilon>0}\Bigg(\bigcup^\infty_{k=1}E^\varepsilon_{4k}\cup\{\xi_0\}\Bigg)=E.\eqno(8.6)$$
Moreover, we request that
$$E^\varepsilon_{4k}\subset \left\{\xi\in\rr^{n+1}|~( 4\pi\lz^{k+2})^{-\frac{n+a}{2}}\l(1+\frac{x_0^2}{\lz^{k+2}}\r)^{-\frac a2}
\geq \Gamma(\xi_0;\xi)\geq (4\pi\lz^{k-1})^{-\frac{n+a}{2}}\l(1+\frac{x_0^2}{\lz^{k-1}}\r)^{-\frac a2}\right\}$$
and that  $$E^\varepsilon=\{\xi_0\}\cup\Bigg[\bigcup^\infty_{k=1}E^\varepsilon_{4k}\Bigg]\eqno(8.7)$$
be compact. Let $V^\ez$ and $\mu^\ez$ be respectively the equilibrium potential and the equilibrium measure of $E^\ez$. From now on, to simplify the notation we drop the subscript $k$. We assume that $k\in\nn$ has been fixed throughout the discussion and we simply write $V^*$ and $\nu^*$ to denote respectively
the equilibrium potential and measure of $E_{4k}$ with respect to the operator $\mathscr{L}^*_a$ in (5.4). By Theorem 7.1, we can find two sequences
$(V_j^\ez)_{j\in\nn}$
and $(V_j^*)_{j\in\nn}$ satisfying (i)-(v) in Theorem 7.1 ($V_j^*$ is defined as in
$V_j^\ez$ replacing $\mathscr{L}_a$ with its adjoint $\mathscr{L}^*_a$) such that
$$0\le V_j^\ez\uparrow V^\ez,\quad 0\le V_j^*\uparrow V^*;\eqno(8.8)$$
for every $\dz>0$, $\exists\ p\in\nn: V_j^\ez(\xi)=V^\ez(\xi)$ for $j\ge p$, for every $\xi\in\rr^{{n+1}}$ such that ${\rm dist}(\xi; E^\ez)\ge \dz$;
moreover, $V_j^*(\xi)=V^*(\xi)$ for $j\ge p$ and every $\xi\in\rr^{n+1}$ such that ${\rm dist}(\xi; E_{4k})\ge \dz$.
In particular, we have
$$\mathscr{L}^*_a  V_j^*=-\nu_j^*,\ \text{in the sense of distributions on} ~\rr^{n+1}.\eqno(8.9)$$
By (8.8) we know that for every Borel subset $A\subset\rr^{n+1}$, such that $\dot{A}\supset E_{4k}$
$$\nu_j^*(A)\to V^*(A),\quad {\rm as} \ j\to\fz,\eqno(8.10)$$
where $v_j^*(A)=\int_Av_j^*(\xi)d\xi$. Moreover, from (8.6) and (8.7), we know that $V^\ez=1$ on $\dot{E}^\ez_{4k}$, it is not restrictive  to suppose that
$$V^\ez(\xi)=1\quad\text{ for every }\ \xi\in\supp \nu_j^*.\eqno(8.11)$$
Next we only consider the case $x_0\not =0$, another case $x_0=0$ is similar, more simple.

We first fix $c_0$ and $\sz$ as in the discussion following (6.36). We then choose $\lz\in (0,1)$ such that $\lz\le \min\{\frac 38\sz, c_0x_0^2\}$, and set $r=\frac 12\lz^{4k-2}$ for $k\ge 1$. With this choice $\lz^{4k-1}<\sz r$ and $r< c_0x_0^2$, we then have
$$E_{4k}\subset\oz_{\lz^{4k-1}}(\xi_0)\subset \oz_{\sz r}(\xi_0)\subset Q_{2r}(\xi_0):=Q(2r).\eqno(8.12)$$
For convenience, we omit $\xi_0$ in (8.12).
 In $Q(2r)$ we can write
$$W^\varepsilon_j=P^\varepsilon_j-S^\varepsilon_j,\eqno(8.13)$$
where $P^\varepsilon_j$ solves the problem
$$
    \left\{
   \begin{array}{ll}
\mathscr{L}_aP^\varepsilon_j=0\qquad &\text{in}~Q(2r),\\
P^\varepsilon_\delta=W^\varepsilon_j&\text{on}~\p Q(2r)\backslash\{\xi_0\}
 \end{array}
 \right.
\eqno(8.14)$$
(note that $r< c_0x_0^2$, from Remark 7.1, we know that every point of $\Psi_r(\xi_0)\setminus\{\xi_0\}$ is $\mathscr{L}_a$-regular), and
$$
    \left\{
   \begin{array}{ll}
\mathscr{L}_aS^\varepsilon_j=-\mathscr{L}_aW^\varepsilon_j\geq 0\quad &\text{in}~Q(2r),\\
S^\varepsilon_j=0&\text{on}~\p Q(2r)\backslash\{\xi_0\}.
 \end{array}
 \right.\eqno(8.15)
$$

Following the argument in \cite{EG},  we obtain
\begin{align*}
\bigg(\mathop{\inf}_{\Omega(3\sz r/4)}P^\varepsilon_j\bigg)^2\nu^*_j\bigg(\Omega\bigg(\frac{3\sz r}{4}\bigg)\bigg)&\leq \int_{\Omega(3\sz r/4)}(P^\varepsilon_j)^2\nu^*_jdXdt\leq \int_{Q(2r)}(P^\varepsilon_j)^2\nu^*_jdXdt\quad(\text{by}~ (8.12))\\
&=\int_{Q(2r)}(S^\varepsilon_j)^2\nu^*_jdXdt\qquad (\text{by}~ (8.11) ~\text{and}~ (8.13))\\
&=\int_{Q(2r)}(S^\varepsilon_j)^2(\mathscr{L}^*_aV^*_j)~dXdt\\
&=\int_{Q(2r)}[\mathscr{L}_a(S^\varepsilon_j)^2]V^*_j~dXdt\qquad (S^\varepsilon_j=0~\text{on}~\p Q(2r)\backslash\{\xi_0\})\\
&=2\int_{Q(2r)}S^\varepsilon_j V^*_j \mathscr{L}_aS^\varepsilon_j dXdt-2\int_{Q(2r)}|y|^a V^*_j|\nabla_X S^\varepsilon_j|^2~dXdt\\
&\leq C\int_{Q(2r)}-\mathscr{L}^*_aW^\varepsilon_j~dXdt\qquad (\text{by}~(8.14)~\text{and}~ (8.15))\\
&\leq C\int_{\Omega(2r)}-\mathscr{L}^*_aW^\varepsilon_j~dXdt.
\end{align*}
Hence $$\bigg(\inf_{\Omega(3\sz r/4)}P^\varepsilon_j\bigg)^2\nu^*_j\bigg(\Omega\bigg(\frac{3\sz r}{4}\bigg)\bigg)\leq C\int_{\Omega(2r)}-\mathscr{L}^*_aW^\varepsilon_j~dXdt.$$

From this and by Theorem 6.2 and Corollary 6.1, we obtain
\begin{align*}
&\left(\mathop{\dashint}_{I_r}P^\varepsilon_j\bigg(X,-\eta r\bigg)|x|^adXdt\right)^2\nu^*_j\bigg(\Omega\bigg(\frac{3\sz r}{4}\bigg)\bigg)r^{-\frac{n+a}{2}}\l(1+\frac{x^2_0}r\r)^{-\frac a2}\tag{8.16}\\
&\leq C[W^\varepsilon_{j,2\az^2r}(\xi_0)-W^\varepsilon_{ j,2\az r}(\xi_0)],
\end{align*}
where $\az>1$ has to chosen.  Now we take $\az=\lz^{-1}$ so that $2\az r=\lz^{4k-3}$ and $2\az^2 r=\lz^{4k-4}$, by Corollary (6.1), then
$$W^\varepsilon_{j,2\az^2 r}(\xi_0)-W^\varepsilon_{ j,2\az r}(\xi_0)=W^\varepsilon_{j,\lz^{4k-4}}(\xi_0)-W^\varepsilon_{j,\lz^{4k-3}}(\xi_0)
\leq\dsum_{j=1}^4\l(W^\varepsilon_{j,\lz^{4k-i}}(\xi_0)-W^\varepsilon_{j,\lz^{4k-i+1}}(\xi_0)\r).\eqno(8.17)$$
By (8.9) and (8.11), note that $E_{4k}\subset \oz_{3\sz r/4}$, we have $j\to\fz$
$$W^\ez_{j,\rho}\to W^\ez_\rho\quad {\rm for}\quad \rho>0,\quad \nu_j^*(\oz_{3\sz r/4})\to \nu^*(\oz_{3\sz r/4})={\rm cap}(E_{4k}).\eqno(8.18)$$
Since $E\cap\{(X,t)\in Q(2r): t\le -\eta r\}=\emptyset$, we have $s_j^\ez(X,t)=0$, for $j$ large enough, if $(X,t)\in Q(2r)$ and $t\le -\eta r$. By (8.18) and (8.17),
and letting $j\to\fz$ we get
\begin{align*}\tag{8.19}
&\left(\dint_{|X-X_0|^2\leq R^1_r(-\eta r)}W^\ez\bigg(X,-\eta r\bigg)|x|^adXdt\right)^2\mathrm{cap}(E_{4k})~r^{-\frac{n+a}{2}}\l(1+\frac{x^2_0}r\r)^{-\frac a2}\\
&\leq C\dsum^{4}_{i=1}[W^\ez_{\lz^{4k-i}}-W^\ez_{\lz^{4k-i+1}}].
\end{align*}
Now $W^\ez\uparrow \widehat{W}$ as $\ez\downarrow 0$, where $\widehat{W}=1-V$ in an open neighborhood of the set $\{(X,-\eta r)~|~|X-X_0|^2\le R^1_r(-\eta r)\}$. Therefore,
letting $\ez\to 0^+$ in (8.19) we get
\begin{align*}\tag{8.20}
&\left(\dint_{|X-X_0|^2\leq R^1_r(-\eta r)}W\bigg(X,-\eta r\bigg)|x|^adXdt\right)^2\mathrm{cap}(E_{4k})~r^{-\frac{n+a}{2}}\l(1+\frac{x^2_0}r\r)^{-\frac a2}\\
&\leq C\dsum^{4}_{i=1}[W_{\lz^{4k-i}}-W_{\lz^{4k-i+1}}].
\end{align*}

Recalling now that $r=\frac 12\lz^{4k-2},$
 and setting $$\beta_k=\mathop{\dashint}_{|X-X_0|^2\leq R^1_r(-\frac 12\eta\lz^{4k-2} )}W(X,-\frac 12\eta\lz^{4k-2})|x|^2dX\quad k=1,2,\ldots,$$
by (8.20) and the fact that $\displaystyle\sum^\infty_{k=1}\sum^{4}_{i=1}[W_{\lz^{4k-i}}-W_{\lz^{4k-i+1}}]$ is a telescopic series, we conclude that
$$\sum^\infty_{k=1}(\beta_k)^2\lz^{-2(n+a)k}\l(1+\frac{x^2_0}{\lz^{4k}}\r)^{-\frac a2}\mathrm{cap}(E_{4k})<\infty.$$
By (8.5) we must have $$\liminf_{k\rightarrow\infty}\beta_k=0.\eqno(8.21)$$
With (8.21) in hand the rest of the proof of the claim $(W(\xi_0)=0)$ follows exactly as in the case of the heat equation. We leave out the details referring the reader to \cite{EG}.
\end{proof}

LMAM, School of Mathematical Sciences,
 Peking University, Beijing, 100871,
 P. R. China

Xi Hu,\quad
E-mail address: huximath1994@163.com

 Lin Tang,\quad
 E-mail address:  tanglin@math.pku.edu.cn

\end{document}